\documentclass[a4paper,11pt]{amsart}
\usepackage{amssymb,amscd,amsxtra,amsmath,braket}
\usepackage[all]{xy}
\usepackage{booktabs}
\usepackage{longtable}
\usepackage{nicematrix}
\usepackage{mathtools}
\usepackage{comment}
\usepackage{xcolor}

\usepackage{ulem}
\usepackage{xcolor}
\DeclareRobustCommand{\erase}{\bgroup\markoverwith{\textcolor{red}{\rule[.5ex]{2pt}{0.4pt}}}\ULon}

\usepackage[pagebackref=true, linktocpage=true]{hyperref}
\hypersetup{%
bookmarksnumbered=true,%
colorlinks=true,%
linkcolor=blue,%
citecolor=blue,%
urlcolor=blue,%
setpagesize=false,%
pdftitle={},%
pdfauthor={Takuzo Okada}}

\theoremstyle{plain}
\newtheorem{Thm}{Theorem}[section]
\newtheorem{Lem}[Thm]{Lemma}
\newtheorem{Cor}[Thm]{Corollary}
\newtheorem{Prop}[Thm]{Proposition}
\newtheorem{Conj}[Thm]{Conjecture}

\theoremstyle{definition}
\newtheorem{Def}[Thm]{Definition}
\newtheorem{Def-Lem}[Thm]{Definition-Lemma}

\newtheorem{Rem}[Thm]{Remark}

\newtheorem*{Ack}{Acknowledgments}

\numberwithin{equation}{section}

\newcommand{\Aut}{\operatorname{Aut}}
\newcommand{\Bir}{\operatorname{Bir}}
\newcommand{\codim}{\operatorname{codim}}

\newcommand{\Proj}{\operatorname{Proj}}
\newcommand{\prt}{\partial}
\newcommand{\Sing}{\operatorname{Sing}}
\newcommand{\Spec}{\operatorname{Spec}}

\newcommand{\rank}{\operatorname{rank}}

\newcommand{\Cl}{\operatorname{Cl}}

\newcommand{\Bs}{\operatorname{Bs}}

\newcommand{\wt}{\operatorname{wt}}

\newcommand{\Qsm}{\operatorname{Qsm}}
\newcommand{\NQsm}{\operatorname{NQsm}}

\newcommand{\lct}{\operatorname{lct}}

\newcommand{\bMov}{\overline{\operatorname{Mov}}}

\newcommand{\bEff}{\overline{\operatorname{Eff}}}

\newcommand{\ord}{\operatorname{ord}}
\newcommand{\Supp}{\operatorname{Supp}}
\renewcommand{\wt}{\operatorname{wt}}
\newcommand{\coeff}{\operatorname{coeff}}

\newcommand{\vol}{\operatorname{vol}}
\newcommand{\wf}{\operatorname{wf}}
\newcommand{\QI}{\operatorname{QI}}
\newcommand{\EI}{\operatorname{EI}}
\newcommand{\Sm}{\operatorname{Sm}}

\newcommand{\BR}{\mathrm{BR}}
\newcommand{\BSR}{\mathrm{BSR}}

\newcommand{\mbA}{\mathbb{A}}

\newcommand{\mbC}{\mathbb{C}}

\newcommand{\mbH}{\mathbb{H}}

\newcommand{\mbP}{\mathbb{P}}
\newcommand{\mbQ}{\mathbb{Q}}
\newcommand{\mbR}{\mathbb{R}}
\newcommand{\mbT}{\mathbb{T}}
\newcommand{\mbU}{\mathbb{U}}
\newcommand{\mbZ}{\mathbb{Z}}

\newcommand{\mcC}{\mathcal{C}}

\newcommand{\mcH}{\mathcal{H}}
\newcommand{\mcI}{\mathcal{I}}

\newcommand{\mcO}{\mathcal{O}}

\newcommand{\msF}{\mathsf{F}}
\newcommand{\msG}{\mathsf{G}}
\newcommand{\msH}{\mathsf{H}}
\newcommand{\msI}{\mathsf{I}}

\newcommand{\msi}{\mathsf{i}}
\newcommand{\msp}{\mathsf{p}}

\newcommand{\ratmap}{\dashrightarrow}

\newcommand{\bmu}{\boldsymbol{\mu}}

\makeatletter
\def\imod#1{\allowbreak\mkern10mu({\operator@font mod}\,\,#1)}
\makeatother

\title{K-stablity of Fano threefold hypersurfaces of index 1}

\author[L.~Campo]{Livia~Campo}
\address{Institut f\"{u}r Mathematik, Universit\"{a}t Wien, Oskar-Morgenstern-Platz 1, 1090 Wien, Austria}
\email{livia.campo@univie.ac.at}

\author[T.~Okada]{Takuzo~Okada}
\address{Faculty of Mathematics, Kyushu University, Fukuoka 819-0395, Japan}
\email{tokada@math.kyushu-u.ac.jp}

\date{}

\begin{document}

\maketitle

\begin{abstract}
    We settle the problem of K-stability of quasismooth Fano 3-fold hypersurfaces with Fano index 1 by providing lower bounds for their delta invariants. 
    We use the method introduced by Abban and Zhuang for computing lower bounds of delta invariants on flags of hypersurfaces in the Fano 3-fold. 
\end{abstract}

\tableofcontents

We work over the field of complex numbers. 

\section{Introduction}

\subsection{K-stability and birational superrigidity}

The question of existence of K\"{a}hler-Einstein metrics (KE for short) on Fano manifolds stems from Differential Geometry. It has been progressively formulated into algebro-geometric terms using the notion of K-stability starting from a conjecture of Tian \cite{Tian}, and Donaldson \cite{Donaldson}, later proven by Chen, Donaldson, and Sun \cite{ChenDonaldsonSun}: a smooth Fano variety $X$ admits a KE metric if and only if $X$ is K-polystable.
K-stability was initially defined by the non-negativity of the Futaki invariant, which is calculated from certain flat degenerations (test configurations) of the manifold $X$. 
Since then, new definitions of K-stability have been proposed, culminating with the Valuative Criterion. 
This is more birational-geometric in flavour: K-stability is implied by finding lower bounds for invariants such as the \textit{alpha invariant} below and the \textit{delta invariant} (see Definition \ref{def: delta invariant})
\begin{equation*}
\alpha(X) \coloneqq \sup \{ c \in \mathbb{Q}_{\geq 0} \; | \;(X,cD) \text{ is log canonical for any } D \in |-K_X|_{\mathbb{Q}} \} \; ,
\end{equation*}
where $\left|-K_X \right|_{\mbQ}$ is the set of the effective $\mbQ$-divisors on $X$ which are $\mbQ$-linearly equivalent to $-K_X$.
Alpha and delta invariants are more suitable to treating $\mathbb{Q}$-Fano varieties. 
For instance, a Fano variety $X$ of dimension $n$ is K-stable if $\alpha(X) > n/(n+1)$ (\cite{Tianc1,OdakaSano}). 
The $n/(n+1)$ lower bound is far from being optimal: indeed, many K-stable Fano 3-folds have been proved to have alpha invariant lower than $3/4$ (see \cite{KOW}). 
However, it is possible to decrease the lower bound down to $1/2$ by adding a seemingly-unrelated assumption on the geometry of $X$ as in \cite{StibitzZhuang}.
\begin{Thm}[{\cite[Theorem 1.2, Corollary 3.1]{StibitzZhuang}}] \label{thm:SZ}
	Let $X$ be a birational superrigid $\mathbb{Q}$-Fano variety with Picard rank $\rho_X=1$. 
	If $\alpha(X) \ge 1/2$, then $X$ is K-stable.
\end{Thm}
A Fano variety $X$ is \textit{birationally rigid} if its birational equivalence class consists of only one Mori fibre space (up to isomorphism). It is said to be \textit{birationally superrigid} if, in addition, $\Bir(X)=\Aut(X)$. 
An interesting class of Fano $3$-folds with this property are quasismooth Fano 3-fold hypersurfaces of index $1$ (cf \cite{CP,CPR}). 

A \textit{Fano 3-fold hypersurface} is a normal projective $\mbQ$-factorial 3-dimensional variety with ample anticanonical divisor and terminal singularities that is embedded into a weighted projective space as a well formed and quasismooth hypersurface (see Section \ref{sec:WPS} for definitions).
The class group $\Cl (X)$ of a Fano 3-fold hypersurface $X$ is isomorphic to $\mathbb{Z}$ (see for example \cite[Remark 4.2]{OkadaFanostrat}) and its index is defined to be the positive integer $\iota_X$ such that $-K_X = \iota_X A$ in $\Cl (X)$, where $A$ is the ample ample generator.
There are exactly 95 deformation families of Fano 3-fold hypersurfaces of index $1$ (cf \cite{IanoFletcher,ChenChenChen}), and any quasismooth member of these 95 families is birationally rigid (cf \cite{CP}).
Moreover, it is shown in \cite[Corollary 1.3]{ACP} that a Fano 3-fold hypersurface is birationally rigid if and only if its index is one.

The general quasismooth member of the 95 families of Fano 3-fold hypersurfaces of index $1$ is K-stable (cf \cite[Corollary 1.45]{CheltsovExtremalMetrics}). 
The major work towards removing the generality assumption in \cite{CheltsovExtremalMetrics} is \cite{KOW}. In \cite{KOW} the authors compute alpha invariants for all quasismooth members of the 95 families of Fano 3-fold hypersurfaces (\cite[Theorem 1.3.1]{KOW}), and conclude K-stability for the birationally superrigid families using \cite{StibitzZhuang}. 
\begin{Thm}[{\cite[Theorem 1.3.1 and Corollary 1.3.2]{KOW}}] \label{thm:KOW}
	Let $X$ be a quasismooth Fano 3-fold hypersurface of index $1$.
	Then, $\alpha(X) \geq 1/2$. As a consequence, if $X$ is birationally superrigid, then it is K-stable.
\end{Thm}
Moreover, they prove K-stability for some \textit{strictly rigid} families (namely, rigid but not superrigid) by proving that their alpha invariants are $\alpha(X) > 3/4$.

Our work builds on the results in \cite{KOW}. We use delta invariants to tackle K-stability of the remaining strictly rigid Fano 3-fold hypersurfaces. In this case, we give a positive answer to the following conjecture.
\begin{Conj}[{\cite[Conjecture 7.4.1]{KOW}, \cite[Conjecture 1.9]{KOW18}}] \label{conj:BRimplyK}
    A birationally rigid Fano variety is K-stable.
\end{Conj}

\subsection{Delta invariants, birational rigidity, and Abban-Zhuang method} \label{Intro.2}

Studying delta invariants has proven to be an effective approach to determine K-stability of many smooth Fano 3-folds (cf \cite{Calabi}).  
Delta invariants (see Definition \ref{def: delta invariant} below) were introduced in the context of the Valuative Criterion, and they characterise K-stability as follows.
\begin{Thm}[\cite{FujitaOdaka,FujitaValuative,ChiLi,BlumJonsson,LiuXuZhuang}] \label{thm: delta invariant criterion}
    A Fano variety $X$ is K-stable if and only if $\delta(X) >1$.
\end{Thm}
We therefore prove the following statements by computing delta invariants, which settles Conjecture \ref{conj:BRimplyK} for Fano 3-fold hypersurfaces.
\begin{Thm}[Main Theorem] \label{thm: Main Theorem}
    Any quasismooth Fano $3$-fold hypersurface of Fano index $1$ is K-stable and admits a KE metric.
\end{Thm}

During the development of this work, a new paper on K-stability of Fano hypersurfaces has been released \cite{ST24}. The authors assume that a suitable weight of the weighted projective space divides the degree of the hypersurface. This covers some Fano 3-fold hypersurfaces (see Table 1 in \cite{ST24}). Moreover, they determine K-stability of Fano hypersurfaces in higher dimension that satisfy that assumption.
However, not all Fano 3-fold hypersurfaces have such a divisibility property. 
Our work settles the K-stability of Fano 3-fold hypersurfaces of index 1. 

To prove Theorem~\ref{thm: delta invariant criterion}, we need to give the estimation $\delta_P (X) > 1$ on local delta invariants for any quasismooth Fano 3-fold hypersurface $X$ of index 1 (that are not covered by \cite{ST24}) and for any point $P \in X$, which requires a large amount of computations.
Our idea is to reduce most of such computations to the known estimation $\alpha_P (X) \ge 1/2$ on local alpha invariant obtained in \cite{KOW} for any $X$ and for any point $P \in X$.
This usually only gives the estimation $\delta_P (X) \ge (4/3) \alpha_P (X) \ge 2/3$ (see Theorem~\ref{thm:alphadelta}), which is insufficient for our goal.
We give a local version of Theorem~\ref{thm:SZ}, which shows that the inequality $\alpha_P (X) \ge 1/2$ implies $\delta_P (X) > 1$ for $P \in X$ that is not a maximal center (see Definition~\ref{def:maxcenter}).
To summarize, the local version Theorem~\ref{thm:localSZ} (see also Corollary~\ref{cor:KstBR}) reduces the proof of K-stability of $X$ to giving a bound $\delta_P (X) > 1$ only at several singular points that are maximal centers. 
We believe that Theorem~\ref{thm:localSZ} leads to the proof of K-stability of further Fano varieties which are (close to being) birationally rigid.

To give such a bound on local delta invariants at maximal centers, we use the Abban--Zhuang method \cite{AbbanZhuang} and \cite{Fujita23}.
This technique gives an estimate of $\delta_P(X)$ by linear sections-counting over refinements of multigraded linear series, which are deduced from flags of subvarieties in $X$ (or more generally flags of varieties over $X$). 
The flags $Y_\bullet \colon Y_0 \supset \dots \supset Y_l$ with $l \leq \dim X$ used in \cite{AbbanZhuang} are such that each $Y_i$ of dimension $\dim Y_i$ is smooth at the generic point of the base $Y_l$. 
These are called \textit{admissible flags}. 
Contrary to \cite{AbbanZhuang}, we do not consider admissible flags. 
Instead, we use flags that are only quasismooth (at the generic point of $Y_l$). 
Moreover, we incorporate plt blow-ups in our flag construction in what we call \textit{generalized flags of blow-up type} (see Definition \ref{def:wBL} below). 
This allows us to tackle the case of local delta invariants based at singular points (cyclic quotient singularities). 
In this we follow \cite{Fujita23}. 
Our use of the Abban--Zhuang method differs from \cite{ST24} in that we employ generalized flags of blow-up type and compute Zariski decompositions in the blow-up. 
On the other hand, \cite{ST24} find lower bounds for delta invariants of finite covers of the Fano hypersurfaces.

The structure of the paper is as follows.
In Section~\ref{Preliminaries}, we present useful lemmas about weighted projective varieties and we define delta invariants. 
In Section~\ref{Preliminaries}, we also introduce a local version of Theorem~\ref{thm:SZ}.
In Section~\ref{sect: AZ}, we showcase the types of flags that are employed in the proof of our Main Theorem. For each flag type, we compute the Zariski decompositions needed in Abban--Zhuang method, and the explicit bounds for the local delta invariant $\delta_P(X)$. 
We then proceed to verifying that $\delta_P(X)>1$ is attained with respect to the flags of Section~\ref{sect: AZ} when $P$ is a singular point which is a maximal center in Section~\ref{sect: singular points}. 

\begin{Ack}
    The authors would like to thank Taro Sano and Luca Tasin for informing the authors about their new result \cite{ST24}. 
    The authors also would like to thank Kento Fujita for informing us about the local version of the result of Stibitz-Zhuang (Theorem \ref{thm:localSZ}), which helped us to reduce a huge amount of computations of local delta invariants at non-maximal centers and increased the readability of the article drastically. 
    Moreover, the authors would like to thank the anonymous referee who has carefully read earlier versions of the manuscript, which is now greatly improved. 
    The first author was supported by the Japan Society for the Promotion of Science (JSPS), and by the Korea Institute for Advanced Study (KIAS), grant No. MG087901. The first author is also grateful for the hospitality of Saga University during her JSPS Fellowship.
    The second author is supported by JSPS KAKENHI Grant Number 23K22389.
\end{Ack}

\section{Preliminaries} \label{Preliminaries}

Let $V$ be a normal variety.
\begin{itemize}
    \item We denote by $\Sm (V)$ the set of smooth points of $V$.
    \item Suppose that $V$ has only isolated singularities. 
    Then, for a subset $C \subset V$, we denote by $\Sing_C (V)$ the set of singular points of $V$ which are contained in $C$.
    \item For a linear system $\Lambda$ on $V$ and a point $P \in V$, we denote by $\mcI_P \Lambda$ the sub-linear system of $\Lambda$ consisting of members of $\Lambda$ vanishing at $P$.
    \item For a linear system $\Lambda$ on $V$ and an irreducible and reduced closed subvariety $C \subset V$ such that $C \not\subset \Bs \Lambda$, we denote by $\Lambda|_C$ the restriction of $\Lambda$ on $C$.
\end{itemize}
In the rest of the paper except in Section \ref{sec:localSZ}, by a \textit{Fano variety}, we mean a normal projective $\mathbb{Q}$-factorial variety with only terminal singularities whose anticnaonical divisor is ample.
By a \textit{divisorial contraction} $\varphi \colon \tilde{V} \to V$, we mean a contraction of a $K_{\tilde{V}}$-negative extremal ray in the Mori category, that is, $\tilde{V}$ and hence $V$ have only terminal singularities.
The inverse $V \ratmap \tilde{V}$ of a divisorial contraction is sometimes referred to as a \textit{divisorial extraction}.

Let $f \in \mbC [x_1,\dots,x_n]$ be a polynomial.
For a monomial $\phi = x_1^{i_1} \cdots x_n^{i_n}$, we denote by $\coeff_f (\phi) \in \mbC$ the coefficient of $\phi$ in $f$.
We write $\phi \in f$ if $\coeff_f (\phi) \ne 0$.

\subsection{Definition of delta invariants}

We define \textit{delta invariants} that we will use throughout this paper. 
Let $X$ be a normal projective variety with only klt singularities.
\begin{Def} \label{def:pdiv}
    A \textit{prime divisor} $E$ over $X$ is a prime divisor on a normal variety $\tilde{X}$ admitting a projective birational morphism $f \colon \tilde{X} \rightarrow X$.
    The image of $E$ on $X$ is denoted by $C_X (E)$.
\end{Def}

\begin{Def}
    Let $L$ and $D$ be $\mbR$-divisors on $X$.
    Suppose that $D$ is pseudoeffective.
    The \textit{pseudoeffective threshold} of $D$ with respect to $L$ is defined as
    \[
    \tau_L (D) \coloneqq \sup \Set{c | \text{$L - c D$ is pseudoeffective} }.
    \]
    For a prime divisor $E$ over $X$, we define $\tau_L (E) \coloneqq \tau_{f^*L} (E)$, where $f \colon \tilde{X} \to X$ is as in Definition \ref{def:pdiv}. 
\end{Def}

The \textit{log discrepancy} of $E$ is
\begin{equation*}
    A_X(E) \coloneqq 1 + a_X(E) = 1 + \ord_E(K_{\tilde{X}} - f^*K_X) \; .
\end{equation*}
The \textit{Fujita-Li invariant} associated to $E$ is
\begin{equation*}
    S_X(E) \coloneqq \frac{1}{(-K_X)^3} \int_0^{\infty} \vol \left( - f^*K_X - tE \right) dt \; .
\end{equation*}
Note that the volume in $S_X(E)$ is positive when $- f^*K_X - tE$ is big. Moreover, it is zero if $- f^*K_X - tE$ is not pseudoeffective. Thus, $t$ varies between $0$ and $\tau_{-K_X} (E)$, and so $S_X(E)$ is in fact given by a definite integral. 
\begin{Def} \label{def: delta invariant}
    The \textit{delta invariant} at $E$ is defined as
    \begin{equation*}
        \delta_X(E) \coloneqq \frac{A_X(E)}{S_X(E)} \;.
    \end{equation*}
    Thus, the \textit{global delta invariant} for $X$ is
    \begin{equation*}
        \delta(X) \coloneqq \inf \{ \delta_X(E) \; | \; E \text{ prime divisor over } X \} \; .
    \end{equation*}
    Given a scheme point $\eta \in X$, it is also possible to define the \textit{local delta invariant} at $\eta$ as
    \begin{equation*}
        \delta_{\eta}(X) \coloneqq \inf \{ \delta_X(E) \; | \; E \text{ prime divisor over } X \text{ such that } \eta \in C_X(E) \} \; .
    \end{equation*}
\end{Def}

We recall the definition of (local) alpha invariants.

\begin{Def}
    Let $X$ be a normal projective variety with only klt singularities.
    Then the \textit{alpha invariant} of $X$ is defined as
    \[
    \alpha (X) \coloneq \sup \Set{ c \in \mbQ_{\ge 0} | \text{$(X, cD)$ is log canonical for any $D \in \left|-K_X \right|_{\mbQ}$}},
    \]
    where $\left|-K_X \right|_{\mbQ}$ is the set of effective $\mbQ$-divisors which are $\mbQ$-linearly equivalent to $-K_X$.
    For a scheme point $\eta \in X$, the \textit{local alpha invariant} of $X$ at $\eta$ is defined as
    \[
    \alpha_{\eta} (X) \coloneq \sup \Set{ c \in \mbQ_{\ge 0} | \text{$(X, cD)$ is log canonical at $\eta$ for any $D \in \left|-K_X \right|_{\mbQ}$}}.
    \]
    For an effective $\mbQ$-divisor $D$ on $X$ and a scheme point $\eta \in X$, we also define the \textit{log canonical threshold} of $(X, D)$ at $\eta$ as
    \[
    \lct_{\eta} (X;D) \coloneqq \sup \Set{ c \in \mbQ_{\ge 0} | \text{$(X, cD)$ is log canonical at $\eta$}}.
    \]
\end{Def}

We have the following relation between (local) alpha and delta invariants.

\begin{Thm}[{\cite[Theorem A]{BlumJonsson}, \cite[Proposition 2.9]{Fujita23}}] \label{thm:alphadelta}
    Let $X$ be a klt Fano variety of dimension $n$.
    Then 
    \[
    \frac{1}{n+1} \delta (X) \le \alpha (X) \le \frac{n}{n+1} \delta (X). 
    \]
    More generally, for a scheme point $\eta \in X$, we have
    \[
        \frac{1}{n+1} \delta_{\eta} (X) \le \alpha_{\eta} (X) \le \frac{n}{n+1} \delta_{\eta} (X). 
    \]
\end{Thm}

\subsection{Local version of the result of Stibitz--Zhuang}
\label{sec:localSZ}

We give a local version of the result of Stibitz--Zhuang on the K-stability of birationally superrigid Fano variety whose alpha invariant is greater than or equal to $1/2$ (Theorem \ref{thm:SZ}).
Before this, we recall necessary definitions and some relevant results.

By a $\mathbb{Q}$-\textit{Fano variety} we mean a normal projective $\mathbb{Q}$-factorial variety with only klt singularities such that its anticanonical divisor is ample.

\begin{Def}
Let $X$ be a $\mathbb{Q}$-Fano variety.
Let $F$ be a prime divisor over $X$ which is a prime divisor on a normal projective variety $\tilde{X}$ admitting a birational morphism $\pi \colon \tilde{X} \to X$.
We say that $F$ is \textit{dreamy} if the graded algebra
\[
\bigoplus_{k, j \in \mathbb{Z}_{\ge 0}} H^0 (\tilde{X}, - k r \pi^* K_X - j F)
\]
is finitely generated for some (hence for any) $r \in \mathbb{Z}_{> 0}$ with $-rK_X$ Cartier.
\end{Def}

\begin{Thm}[{\cite[Corollary 1.5 and Theorem 1.6]{FujitaValuative}, \cite[Theorem 3.7]{ChiLi}}] \label{thm:Kstdeltadreamy}
Let $X$ be a $\mathbb{Q}$-Fano variety.
Then it is K-stable (resp.\ K-semistable) if and only if $\delta_X (F) > 1$ for any dreamy prime divisor $F$ over $X$.
\end{Thm}

The following is a local version of the result of Stibitz-Zhuang (see Theorem \ref{thm:SZ}).

\begin{Thm} \label{thm:localSZ}
Let $X$ be a $\mathbb{Q}$-Fano variety of Picard number $1$ and let $\eta \in X$ be a scheme point of height at least $2$.
Suppose that $\alpha_{\eta} (X) \ge 1/2$ and that $\eta$ is not a non-klt center of the pair $(X, \frac{1}{n} \mathcal{M})$ for any movable linear system $\mathcal{M}$ on $X$, where $n > 0$ is the rational number defined by $\mathcal{M} \sim_{\mathbb{Q}} - n K_X$.
Then, for any dreamy prime divisor $F$ over $X$ whose center is $\eta$, the inequality $\delta_X (F) > 1$ holds.
\end{Thm}

\begin{proof}
The proof is completely identical to those of \cite[Theorem 1.2, Corollary 3.1]{StibitzZhuang}.

First, by the same arguments as in the proof of \cite[Theorem 1.2]{StibitzZhuang}, we can conclude that $\delta_X (F) \ge 1$ for any dreamy prime divisor $F$ over $X$ whose center is $\eta$ under the assumption of Theorem \ref{thm:localSZ}.

Next, we can derive a contradiction assuming the existence of a dreamy prime divisor $F$ over $X$ whose center is $\eta$ such that $\delta_X (F) = 1$ by the same argument as in the proof of \cite[Corollary 3.1]{StibitzZhuang} again under the assumption of Theorem \ref{thm:localSZ}.
\end{proof}

We introduce the notion of maximal singularity.

\begin{Def} \label{def:maxcenter}
We say that a subvariety $\Gamma$ is a \textit{maximal center} if there is a movable linear system $\mathcal{M}$ on $X$ such that $\Gamma$ is a center of non-canonical singularities of the pair $(X, \frac{1}{n} \mathcal{M})$, where $n > 0$ is a rational number defined by $\mathcal{M} \sim_{\mathbb{Q}} - n K_X$.
\end{Def}

Let $X$ be a Fano variety, that is, $X$ is a normal projective $\mathbb{Q}$-factorial variety with at most terminal singularities such that $-K_X$ is ample, and we assume that the Picard number of $X$ is $1$.
Then, $X$ is birational superrigid if and only if there is no maximal center on $X$ (\cite[Theorem 1.26]{CheltsovShramov}).
Fano varieties which are far from being birationally superrigid (e.g.\ $\mathbb{P}^n$) may be covered by maximal centers.
However, a Fano variety which is closer to being birationally superrigid (e.g.\ birationally rigid Fano varieties) has only a few maximal centers.
For example, any quasismooth Fano $3$-fold weighted hypersurface of index $1$ has only finitely many maximal centers (see Section \ref{sec:BGWH}).
In such a case, the above local result plays a role.

\begin{Cor} \label{cor:KstBR}
Let $X$ be a Fano variety of Picard number $1$ and let $\Sigma \subset X$ be the set of closed points that are maximal centers on $X$.
\begin{enumerate}
\item Let $P \in X$ be a closed point.
If $P \notin \Sigma$ and $\alpha_P (X) \ge 1/2$, then, for any dreamy prime divisor $F$ over $X$ whose center contains $P$, the inequality $\delta_X (F) > 1$ holds.
\item If $\delta_P (X) > 1$ for any closed point $P \in \Sigma$ and if $\alpha_Q (X) \ge 1/2$ for any closed point $Q \in X \setminus \Sigma$, then $X$ is K-stable. 
\end{enumerate}
\end{Cor}

\begin{proof}
The assertion (1) follows immediately from Theorem~\ref{thm:localSZ} and (2) follows from (1) and Theorem~\ref{thm:Kstdeltadreamy}.
\end{proof}

\subsection{Basics on weighted projective varieties}
\label{sec:WPS}

Let 
\[
\mbP \coloneq \mbP (a_0, \dots, a_N) = \Proj \mbC [x_0, \dots, x_N]
\]
be a weighted projective space, where $a_0, \dots, a_N$ are positive integers and the weight of $x_i$ is $a_i$ for $i = 0, \dots, N$.
We may assume that $\gcd \{a_0, \dots, a_N\} = 1$.

\begin{Def} 
    We say that $\mbP$ is \textit{well formed} if the greatest common divisors of any $N$ of the weights $a_i$ is $1$.
    A subscheme $X \subset \mbP$ is \textit{well formed} if $\mbP$ is well formed and 
    \[
    \codim_X (X \cap \Sing \mbP) \ge 2.
    \]
\end{Def}

Let $X$ be the subscheme of $\mbP$ defined by a homogeneous ideal $I \subset \mbC [x_0, \dots, x_N]$: $X = \Proj \mbC [x_0, \dots, x_N]/I$.

\begin{Def}
    We define 
    \[
    C_X \coloneq \Spec (\mbC [x_0, \dots, x_N]/I) \subset \mbA^N
    \]
    and call it the \textit{affine cone} of $X$.
    We say that $X$ is \textit{quasismooth} if $C_X$ is smooth outside the origin.
    Let $p \colon \mbA^N \setminus \{O\} \to \mbP$ be the natural projection, where $O$ is the origin of $\mbA^N$.
    For a point $P \in X$, we say that $X$ is \textit{quasismooth at $P$} if $C_X$ is smooth along $p^{-1} (P)$.
    The \textit{quasismooth locus} of $X$ is denoted by $\Qsm (X)$ and then the \textit{non-quasismooth locus} of $X$ is defined as $\NQsm (X) \coloneq X \setminus \Qsm (X) = p (\Sing (C_X) \setminus \{O\})$.
\end{Def}

\begin{Rem}
For a well formed and quasismooth weighted complete intersection $X \subset \mbP$, we have $\Sing (X) = X \cap \Sing (\mbP)$ (see \cite[Proposition 8]{Dimca}).
\end{Rem}

\begin{Def}
    Let $\msF_1, \dots, \msF_m \in \mbC [x_0, \dots,x_N]$ be homogeneous polynomials which generate $I$.
    We define
    \[
    J_X \coloneq \left( \frac{\prt \msF_i}{\prt x_j} \right)_{1 \le i \le m, \, 0 \le j \le N}
    \]
    and call it the \textit{Jacobian matrix} of $X$ (with respect to $\msF_1, \dots, \msF_m$).
\end{Def}

The matrix $J_X$ is usually thought of as the Jacobian matrix of the affine cone $C_X$ and, for a point $Q \in C_X$, $C_X$ is smooth at $Q$ if and only if $\rank J_X (Q) = N-n$, where $n \coloneq \dim X$.
Clearly $\rank J_X (Q)$ does not depend on the choice of homogeneous generators of $I$.
Let $P \in X$ be a point.
Then $\rank (J_X (Q))$ does not depend on the choice of $Q \in p^{-1} (P)$ and thus we denote it by $\rank (J_X (P))$.
It follows that $X$ is quasismooth at $P$ if and only if $\rank J_X (P) = N - n$.

For homogeneous polynomials $f_1, \dots, f_m \in \mbC [x_0, \dots, x_N]$, we denote by
\[
(f_1 = \cdots = f_m = 0) \subset \mbP
\]
the subscheme defined by the homogeneous ideal $(f_1, \dots, f_m)$, and then we define
\[
(f_1 = \cdots = f_m =0)_X \coloneq (f_1 = \cdots = f_m =0) \cap X,
\]
which is a scheme-theoretic intersection.
Similarly, for $f \in \mbC [x_0,\dots,x_N]$, we denote by $(f \ne 0)$ the complement of $(f = 0)$ in $\mbP$. 

\begin{Def}
    A \textit{quasi-linear} polynomial $f = f (x_0, \dots, x_N)$ is a homogeneous polynomial (with respect to $\wt (x_0, \dots, x_N) = (a_0,\dots,a_N)$) such that $x_i \in f$ for some $0 \le i \le N$.
    We say that a subvariety $S \subset \mbP$ is \textit{quasi-linear} if it is a complete intersection in $\mbP$ defined by quasi-linear equations of the form
    \[
    \ell_1 + g_1 = \ell_2 + g_2 = \cdots = \ell_k + g_k = 0,
    \]
    where $\ell_1, \ell_2, \dots, \ell_k$ are linearly independent linear forms in variables $x_0, \dots, x_N$ and $g_1, g_2, \dots, g_k \in \mbC [x_0, \dots, x_N]$ are homogeneous polynomials which are not quasi-linear.
    We say that the above $S$ is \textit{of type} $(e_1, \dots, e_k)$, where $e_i \coloneq \deg (\ell_i + g_i)$ for $1 \le i \le k$.
    A quasi-linear subspace of $\mbP$ of codimension $1$ (resp.\ dimension $1$) is called a \textit{quasi-hyperplane} (resp.\ \textit{quasi-line}).
\end{Def}

We will frequently use the following fact implicitly.

\begin{Lem}[{\cite[Lemma 3.7]{KOW}}] \label{lem:qhypnormal}
    Let $X \subset \mbP (a_0,a_1,a_2,a_3,a_4)$ be a quasismooth hypersurface in a well formed weighted projective $4$-space.
    If $X$ has only isolated singularities, then any quasi-hyperplane section of $X$ is a normal surface.
\end{Lem}

\subsection{Singularities of weighted projective varieties}

\begin{Lem} \label{lem:plt}
    Let $X \subset \mbP \coloneq \mbP (a_0,\dots,a_N)$ be an irreducible and reduced normal variety such that $X \subset \mbP$ is well formed, and let $Y \subset X$ be a prime divisor on $X$.
    Suppose that $X$ is quasismooth along $Y$ and $Y$ is quasismooth.
    Then the pair $(X, Y)$ is plt in a neighborhood of $Y$.
\end{Lem}

\begin{proof}
    We denote by $X^{\circ} \coloneq \Qsm (X)$ the quasismooth locus of $X$.
    We have $Y \subset X^{\circ}$.
    We can take an open cover $\{U_{\lambda}\}$ of $X^{\circ}$ such that $U_{\lambda}$ is the quotient of a smooth variety $\tilde{U}_{\lambda}$ by an action of the cyclic group $\bmu_{r_{\lambda}}$ of order $r_{\lambda}$ and $Y \cap U_{\lambda}$ is the quotient of a smooth hypersurface $\tilde{V}_{\lambda}$ of $\tilde{U}_{\lambda}$.
    The quotient morphism $\tilde{U}_{\lambda} \to U_{\lambda}$ is \'{e}tale in codimension $1$ since $X \subset \mbP$ is assumed to be well formed.
    The pair $(\tilde{U}_{\lambda}, \tilde{V}_{\lambda})$ is clearly plt.
    By \cite[Proposition 3.16]{Kol97}, the pair $(U_{\lambda}, Y \cap U_{\lambda})$ is plt.
\end{proof}

\begin{Lem} \label{lem:diff}
    Let $X \subset \mbP \coloneq \mbP (a_0,\dots,a_N)$ be an irreducible and reduced normal variety such that $X \subset \mbP$ is well formed, and let $Y \subset X$ be a prime divisor on $X$.
    Suppose that $X$ is quasismooth along $Y$ and $Y$ is quasismooth.
    \begin{enumerate}
        \item If $\dim X \ge 3$ and the singular locus of $X$ is of codimension at least $3$ in $X$, then $(K_X+Y)|_Y = K_Y$.
        \item If $\dim X = 2$, then $(K_X+Y)|_Y = K_Y + \Delta_Y$. 
        Here
        \[
        \Delta_Y = \sum_{P \in \Sing_Y (X)} \frac{r_P-1}{r_P} P,
        \]
        where $r_P$ is the index of the cyclic quotient singularity of $P \in X$.
    \end{enumerate}
\end{Lem}

\begin{proof}
    Suppose that $\dim X \ge 3$ and $\Sing (X)$ is of codimension at least $3$ in $X$.
    We set $S \coloneq \Sing (X) \cup \Sing (Y)$.
    Note that $S$ is of codimension at least $3$ in $X$ and $S \cap Y$ is of codimension at least $2$ in $Y$.
    Set $X^{\circ} = X \setminus S$ and $Y^{\circ} = Y \setminus (Y \cap S)$.
    We have $(K_{X^{\circ}} + Y^{\circ})|_{Y^{\circ}} = K_{Y^{\circ}}$ since $X^{\circ}$ and $Y^{\circ}$ are smooth.
    This proves (1).

    We prove (2).
    By Lemma \ref{lem:plt}, the pair $(X, Y)$ is plt.
    Thus, (2) is a consequence of \cite[Proposition 16.6]{Kol+92}.
\end{proof}

    \begin{Lem} \label{lem:alpspH}
    Let $a, b, c_1, c_2$ be positive integers and let $\mbP \coloneq \mbP (1,a,b,c_1,c_2)$ be the weighted projective space with homogeneous coordinates $x, y, z, t_1, t_2$ of degree $1, a, b, c_1, c_2$, respectively. 
    Let $X \subset \mbP$ be a well formed and quasismooth hypersurface of degree $d$ with only isolated singularities.
    We set $C \coloneq (x = t_1 = t_2 = 0) \subset \mbP$ and $H \coloneq (x = 0)_X$.
    Suppose that
    \begin{enumerate}
        \item $a, b > 1$ and $a$ is coprime to $b$.
        \item $d < c_i + 2 ab$ for $i = 1, 2$.
        \item $X$ contains the quasi-line $C$.
    \end{enumerate}
    Then, for any smooth point $P$ of $X$ contained in $C$, the singularity of $H$ at $P$ is either smooth or Du Val of type $A$.
\end{Lem}

\begin{proof}
    Let $P \in C$ be a smooth point of $X$.
    We may assume that $H$ is singular at $P$.
    We set $\mbU \coloneq (yz \ne 0) \cap \mbP$ and $U \coloneq \mbU \cap X$.
    The points $P_y$ and $P_z$ are singular points of $X$ if they are contained in $X$ since $a, b > 1$. 
    Hence $P \in U$.
    We can take positive integers $m, n$ such that $n b - m a = 1$.
    We set $M \coloneq z^n/y^m$ and $u \coloneq y^b/z^a$.
    Then we have 
    \[
    y|_U \coloneq y/M^a = u^n, \ 
    z|_U \coloneq z/M^b = u^m.
    \]
    We set $x|_U \coloneq x/M$ and $t_i|_U \coloneq t_i/M^{c_i}$ for $i = 1, 2$ and, by a slight abuse of notation, we denote them by $x$ and $t_i$, respectively.
    These form affine coordinates of $\mbU \cong \Spec \mbC [x,u,u^{-1},t_1,t_2]$.
    Rescaling $y$ and $z$, we may assume that $P = (0\!:\!1\!:\!1\!:\!0\!:\!0)$.
    Then, by setting $\bar{u} = u -1$, $\{x, \bar{u}, t_1, t_2\}$ can be chosen as a system of local coordinates of $\mbP$ at $P$.
    Let $\msF = \msF (x, y, z, t_1, t_2)$ be the defining polynomial of $X$.
    We set $e_i \coloneq d - c_i - ab$ for $i = 1, 2$.
    By the assumption that $H$ is singular at $P$, we have $(\prt \msF/\prt v)(P) = 0$ for any $v \in \{y,z,t_1,t_2\}$ and $(\prt \msF/\prt x)(P) \ne 0$.
    By the assumption (2), any homogeneous polynomial of degree $d$ in variables $y, z, t_1, t_2$ which is divisible by $(z^a - y^b)^2$ is of the form $(z^a - y^b)^2 h (y, z)$ and, by the assumption (3), $\msF$ does not contain any monomial consisting only of $y$ and $z$.
    Hence we can write
    \[
    \msF = t_1 (z^a - y^b) f_{e_1} (y,z) + t_2 (z^a - y^b) f_{e_2} (y,z) + g_d + x \msF',
    \]
    where $f_{e_i} = f_{e_i} (y,z)$, $g_d = g_d (y,z,t_1,t_2)$ and $\msF' = \msF' (x,y,z,t_1,t_2)$ are homogeneous polynomials of degree $e_i \coloneq d - c_i - ab, d$ and $d-1$, respectively, such that $g_d \in (t_1,t_2)^2$ and $\msF' (P) \ne 0$.
    By the assumption (2), we have $e_i < ab$, which implies that $f_{e_i} (1,1) \ne 0$ for $i = 1,2$ if $f_{e_i}$ is nonzero.
    We see that at least one of $f_{e_1}$ and $f_{e_2}$ is nonzero as a polynomial because otherwise $X$ is not quasismooth at any point of the set $(t_1 = t_2 = x = \msF' = 0) \ne \emptyset$ which is absurd.
    We set $\msF|_U (x,u,t_1,t_2) \coloneq \msF (x,u^n,u^m,t_1,t_2)$ and then $\overline{\msF|_U} \coloneq \msF|_U (0,\bar{u}+1,t_1,t_2)$.
    Then $\overline{\msF|_U}$ contains either $t_1 \bar{u}$ or $t_2 \bar{u}$ and it does not contain $\bar{u}^2$.
    Hence the quadratic part of $\overline{\msF|_U}$ is of rank at least $2$.
    We can choose $\{\bar{u},t_1,t_2\}$ as a system of local coordinates of $X$ at $P$ and $H$ is defined by $\overline{\msF|_U}$ around the point $P \in X$.
    It follows that the singularity of $H$ at $P$ is Du Val of type $A$ since $P \in H$ is an isolated singularity by Lemma \ref{lem:qhypnormal}.
\end{proof}

\begin{Rem}
In Section \ref{sect: singular points}, we prove quasismoothness of various varieties embedded in a weighted projective space.
We frequently apply the following technique.
\begin{itemize} 
\item Let $X \subset \mbP (a_0, \dots, a_N)$ be a quasismooth variety and let $\mcH$ be a linear system on $X$.
Then, by applying Bertini theorem on the affine cone $C_X \subset \mbA^{N+1}$, we conclude that a general member of $\mcH$ is quasismooth outside the base locus of $\mcH$.
\item Let $X \subset \mbP (a_0, \dots, a_N)$ be an irreducible and reduced variety and let $P \in X$ be a point.
If there is a hypersurface $H$ of $\mbP (a_0, \dots, a_N)$ through $P$ such that the scheme-theoretic intersection $X \cap H$ is quasismooth at $P$, then $X$ is quasismooth at $P$.
In particular, if we are given a pencil $\mcH$ on $X$, then its general member is quasismooth at any point at which the base scheme of $\mcH$ is quasismooth.
\end{itemize}
We emphasize that, throughout this paper, by a \textit{base locus} and a \textit{base scheme} of a linear system, we mean the set-theoretic and the scheme-theoretic base locus of the linear system, respectively. 
\end{Rem}

\subsection{The $95$ families and known results on K-stability}

A quasismooth Fano $3$-fold hypersurface can be expressed as $X_d \subset \mbP (a_0, a_1, a_2, a_3, a_4)$, where $a_0 \le \cdots \le a_4$ and $d$ denotes the degree of the defining polynomial.
The \textit{Fano index} of $X = X_d$ is $\sum a_i - d$.
Quasismooth Fano $3$-fold hypersurfaces of index $1$ are known to form $95$ families and each family is thus determined by the sextuple $(d, a_0, \dots, a_4)$ (with $d = a_0 + \cdots + a_4 - 1$).
The $95$ families are numbered in the lexicographical order of $(d, a_0, \dots, a_4)$ and each family is referred to as family \textnumero $\msi$.
In the following, by a member of family \textnumero $\msi$, we mean a \textit{quasismooth} member of family \textnumero $\msi$.

\begin{Def}
    The index set of the $95$ families is denoted by $\msI = \{1, 2, \dots, 95\}$.
    We define $\msI_{\BSR} \subset \msI$ as follows: $\msi \in \msI_{\BSR}$ if and only if any member of family \textnumero $\msi$ is birationally superrigid.
    We then set $\msI_{\BR} \coloneq \msI \setminus \msI_{\BSR}$.
\end{Def}

We have $|\msI_{\BSR}| = 50$ and $|\msI_{\BR}| = 45$.
The families indexed by $\msI_{\BSR}$ are exactly those which are marked $\mathrm{KE}$ in the right-most column of \cite[Table 7]{KOW} plus families \textnumero 1 and \textnumero 3 (see \cite[\S 2.3.a]{KOW} for details).
We divide the set $\msI_{\mathrm{BR}}$ as $\msI_{\BR} = \msI_{\BR}^{*} \sqcup \msI_{\BR}^{**} \sqcup \msI_{\BR}^{***}$, where 
\[
\begin{split}
    \msI_{\BR}^{***} &\coloneq \{2,5,12,13,20,23,25,33,38,40,58\}, \\
    \msI_{\BR}^{**} &\coloneq \{4,7,9,18,24,31,32,43,46\}, \\
    \msI_{\BR}^{*} &\coloneq \msI_{\BR} \setminus (I_{\BR}^{**} \sqcup \msI_{\BR}^{**}).
\end{split}
\]
Note that the families indexed by $\msI_{\BR}^*$ are exactly those that are marked $\mathrm{K}$ in the right-most column of \cite[Table 7]{KOW}.

\begin{Thm}[{\cite[Theorems 1.11 and 1.12]{KOW}}]
    Any member of family \textnumero $\msi$, where $\msi \in \msI_{\BSR} \cup \msI_{\BR}^*$, is K-stable.
\end{Thm}

\begin{Thm}[{\cite[Corollary 1.4]{ST24}}]
    Any member of family \textnumero $\msi$, where $\msi \in \msI_{\BR}^{**}$, is K-stable.
\end{Thm}

Therefore, Theorem \ref{thm: Main Theorem} follows from the following.

\begin{Thm} \label{thm: Main2}
    Let $X$ be a member of family \textnumero $\msi$, where $\msi \in \msI_{\BR}^{***}$.
    Then $\delta (X) > 1$.
\end{Thm}

The rest of the paper is devoted to the proof of Theorem \ref{thm: Main2}.

We fix some notation.
Let $X = X_d \subset \mbP \coloneq \mbP (a_0, a_1, a_2, a_3, a_4)$ be a member of family \textnumero $\msi \in \msI$. 
We use homogeneous coordinates $x, y, z, t, w$ of degree $a_0, a_1, a_2, a_3, a_4$, respectively, unless otherwise mentioned. 
We denote by $\msF \coloneq \msF (x, y, z, t, w)$ the defining polynomial of $X$.
For a coordinate $v \in \{x, y, z, t, w\}$, we define $P_v$, $H_v$ and $U_v$ as follows. 
\begin{itemize}
    \item $P_v$ is the point of $\mbP$ at which only the coordinate $v$ is nonzero,
    \item $\mbH_v \coloneq (v = 0) \subset \mbP$ is the zero locus of $v$.
    \item $\mbU_v \coloneq \mbP \setminus \mbH_v$ is an open subset of $\mbP$.
    \item $H_v \coloneq \mbH_v \cap X = (v = 0)_X$ is the divisor on $X$ cut by $v$.
    \item $U_v \coloneq \mbU_v \cap X = X \setminus H_v$ is an open subset of $X$.
\end{itemize} 

Let $P \in X$ be a singular point.
Then it is a cyclic quotient singularity of type $\frac{1}{r} (1,a,r-a)$ for some coprime positive integers $r$ and $a$, that is, the germ $P \in X$ is equivalent to $\bar{o} \in \mbC^3/\bmu_r$, where the the action of the cyclic group $\bmu_r$ of order $r$ on $\mbC^3$ is given by $(x,y,z) \mapsto (\zeta_r x, \zeta_r^a y, \zeta_r^{r-a} z)$ and $\bar{o}$ is the image of the origin under the quotient morphism $\mbC^3 \to \mbC^3/\bmu_r$.
Let $\varphi \colon \tilde{X} \to X$ be the weighted blow-up of $X$ at $P$ with weight $\frac{1}{r} (1,a,r-a)$.
The morphism $\varphi$ is called the \textit{Kawamata blow-up} of $P \in X$ and it is known that $\varphi$ is the unique divisorial contraction (in the Mori category) centered at $P \in X$ (\cite{Kawamata}).
Let $E \cong \mbP (1,a,r-a)$ be the $\varphi$-exceptional divisor.
Then we have
\[
K_{\tilde{X}} = \varphi^*K_X + \frac{1}{r} E
\]
and
\[
(E^3) = \frac{r^2}{a (r-a)}.
\]

\subsection{Birational geometry of families indexed by $\msI_{\BR}^{***}$} \label{sec:BGWH}

We briefly explain birational geometry of members of families indexed by $\msI_{\BR}^{***}$.
Let $\msi \in \msI_{\BR}^{***}$.
General members of family \textnumero $\msi$ is birationally rigid but not superrigid, which in particular means that they admit a birational self map which is not biregular.
Such a birational self map can be decomposed into chains of elementary self links.
An elementary self link is initiated by a divisorial extraction followed by a birational map which is an isomorphism in codimension $1$ and a divisorial contraction.
The center of a divisorial contraction initiating an elementary self link can be detected by maximal centers (see Definition~\ref{def:maxcenter}).

If there exists an elementary link from a Fano variety of Picard number $1$, then the center of the divisorial extraction initiating the link is a maximal center (see \cite[Lemma 2.5]{Okada3}).

\begin{Def}
    Let $X = X_d \subset \mbP (a_0,a_1,a_2,a_3,a_4)$ be a member of family \textnumero $\msi$, where $\msi \in \msI_{\BR}^{***}$, and let $P$ be a singular point of $X$.
    In this definition, we denote by $x_0, \dots, x_4$ the homogeneous coordinates of degree $a_0, \dots, a_4$, respectively.
    \begin{itemize}
        \item We say that $P$ is a \textit{QI center} if there are distinct $0 \le j, k \le 4$ such that $P = P_{x_k}$ and $d = 2 a_k + a_j$.
        \item We say that $P$ is an \textit{EI center} if $\msi$ and $P$ belong to one of the following:
        \begin{itemize}
            \item $\msi = 20$ and $P$ is of type $\frac{1}{3} (1,1,2)$.
            \item $\msi = 23$ and $P$ is of type $\frac{1}{4} (1,1,3)$.
            \item $\msi = 40$ and $P$ is of type $\frac{1}{5} (1,2,3)$. 
        \end{itemize}
        \item We say that $P$ is an \textit{IEI center} if $\msi = 23$ and $P$ is of type $\frac{1}{3} (1,1,2)$.
    \end{itemize}
    We say that $P$ is a \textit{BI center} if it is one of QI center, EI center and IEI center.
\end{Def}

In the above definition, QI, EI and IEI are abbreviation of \textit{quadratic involution}, \textit{elliptic involution} and \textit{invisible elliptic involution}, respectively.

\begin{Thm}[\cite{CP}]
    Let $X$ be a member of family \textnumero $\msi$, where $\msi \in \msI_{\BR}^{***}$.
    If $\Gamma \subset X$ is a maximal center, then $\Gamma$ is a singular point and it is a BI center.
    If a BI center $P \in X$ is a maximal center, then the Kawamata blow-up of $P \in X$ initiates an elementary self link $X \ratmap X$. 
\end{Thm}

\begin{proof}
The first assertion follows from the whole contents of \cite[Section 5]{CP} and the second assertion follows from the whole contents of \cite[Section 4]{CP}.
\end{proof}

Birational rigidity of $X$ is a consequence of the above theorem (see \cite{CP} for details): it implies that any birational map $X \ratmap Y$ to a Mori fiber space $Y/T$ is a composite of elementary self links initiated by the Kawamata blow-up of BI centers, and hence $Y \cong X$.

It should be emphasized that a BI center is not always a maximal center.
In view of Corollary \ref{cor:KstBR}, our task is to show $\delta_P (X) > 1$ for any maximal center $P \in X$, hence it is important to distinguish whether a BI center is a maximal center or not. It is not easy to give such a distinction for EI centers. However we can do this for QI and IEI centers.
Let $X = X_d \subset \mbP (a_0,a_1,a_2,a_3,a_4)$ be a member of family \textnumero $\msi$, where $\msi \in \msI_{\BR}^{***}$ and $x_0, \dots, x_4$ are homogeneous coordinates of degree $a_0,a_1,a_2,a_3,a_4$, respectively.    Let $P = P_{x_k}$ be a QI center.
Note that there is an index $j \in \{0,1,2,3,4\} \setminus \{k\}$ such that $d = 2 a_k + a_j$.

\begin{Def}
    Let $P \in X$ be a QI center.
    We say that $P$ is \textit{exceptional} if $x_k^2 x_i \notin \msF$ for any $i \in \{0,1,2,3,4\} \setminus \{k\}$.
    Suppose that $P$ is not exceptional.
    Then, by a suitable choice of homogeneous coordinates, we can write
    \[
    \msF = x_k^2 f + x_k g + h,
    \]
    where $f, g, h \in \mbC [x_0,\dots,x_4]$ are homogeneous polynomials which do not involve the variable $x_k$ such that $x_j \in f$ for some $j \in \{0,1,2,3,4\} \setminus \{k\}$.
    We say that $P$ is \textit{degenerate} (resp.\ \textit{nondegenerate}) if $g$ is divisible by $f$ (resp.\ $g$ is not divisible by $f$).
\end{Def}


\begin{Thm}[{\cite[\S 4.1 and \S 5.2]{CP}}] \label{thm:QIcenter}
    \begin{enumerate}
    \item Let $P \in X$ be a QI center.
    Then $P$ is a maximal center if and only if it is nondegenerate.
    \item Let $X$ be a member of family \textnumero $23$ and let $P \in X$ be the $\frac{1}{3} (1, 1, 2)$ point which is an IEI center.
    Then $P$ is a maximal center if and only if $z^3 w, z^2 t^2 \notin \msF$.
    \end{enumerate}
\end{Thm}

\begin{proof}
Suppose that $P \in X$ is a non-exceptional QI center.
Then, since $d = 2 a_k + a_j$, we can choose homogeneous coordinates such that
\[
\msF = x_k^2 x_j + x_k g + h,
\]
where $g, h \in \mbC [x_0, \dots, x_4]$ are homogeneous polynomials which do not involve the variable $x_k$.
If $P \in X$ is nondegenerate, then $g$ is not divisible by $x_j$, and hence $P \in X$ is indeed a maximal center by \cite[Lemma 4.1.1]{CP}.
If $P \in X$ is degenerate, then $g$ is divisible by $x_j$ and it follows from \cite[Lemma 4.1.3 and Theorem 4.1.4]{CP} that $P \in X$ is not a maximal center.
Note that \cite[Lemma 4.1.3 and Theorem 4.1.4]{CP} covers all the QI centers as this can be verified by checking the tables in \cite[Section 5]{CP}.
This proves the assertion (1).
The assertion (2) follows from \cite[Pages 68 and 69]{CP}.
\end{proof}

\section{Abban-Zhuang method for weighted hypersurfaces} \label{sect: AZ}

Let $X$ be a quasismooth Fano $3$-fold weighted hypersurface of index $1$ and $A \coloneq -K_X$ the ample generator (class) of $\Cl (X) \cong \mbZ$.
We explain methods to bound local delta invariants of $X$ from below following the Abban-Zhuang method.
Throughout the present section, let $P \in X$ be a point, which is either a smooth point or a terminal cyclic quotient singular point.
We denote by $r_P$ the index of the singularity $P \in X$.
Note that $P$ is a smooth point of $X$ if and only if $r_P = 1$.
In the following, a flag $P \in Z \subset Y \subset X$ is always a flag of irreducible and reduced subvarieties.

\subsection{Flag of type $\mathrm{I}$}

\begin{Def}
    A \textit{flag of type $\mathrm{I}$} is a flag $P \in Z \subset Y \subset X$, where $Y$ is a normal hypersurface of $X$ such that it is quasismooth at $P$ and the pair $(X, Y)$ is plt, and $Z$ is an irreducible smooth curve such that the pair $(Y, Z)$ is plt with the following property:
    \begin{itemize}
        \item there is a hypersurface $H$ of $X$ such that $Z$ is the reduced scheme $(H \cap Y)_{\mathrm{red}}$ of the scheme-theoretic intersection $H \cap Y$.
    \end{itemize}
\end{Def}

Let $l_Y$ and $l_H$ be positive integers such that $Y \in |l_Y A|$ and $H \in |l_H A|$.
Let $e$ be the positive integer such that $H|_Y = e Z$ as a divisor on $Y$.

\begin{Prop} \label{prop:flagIdelta}
    Let the notation and assumption as above.
    Let $P \in Z \subset Y \subset X$ be a flag of type $\mathrm{I}$.
    Then
    \[
    \delta_P (X) \ge \min \left\{ 4 l_Y, \ \frac{4 l_H}{e}, \  \frac{4 e}{r_P l_Y l_H (A^3)} \right\}.
    \]
\end{Prop} 

\begin{proof}
    We first note that $X$ is a Mori dream space which does not admit a small $\mbQ$-factorial modification other than $X$ itself since $X$ is a Fano $3$-fold of Picard number $1$.
    
    We have $\tau_A (Y) = 1/l_Y$.
    Let $-K_X - u Y = P (u) + N (u)$ be the Zariski decomposition with $P (u)$ positive and $N (u)$ negative parts, respectively.
    For $u \in [0, 1/l_Y]$, we have
    \[
    \begin{split}
        P (u) &= - K_X - u Y \sim (1 - l_Y u) A, \\ 
        N (u) &= 0.
    \end{split}
    \]
    We have $Z \sim_{\mbQ} (1/e) H|_Y \sim (l_H/e) A$ and
    \[
    P (u)|_Y - v Z \sim_{\mbQ} (1-l_Y u - l_H v/e) A|_Y.
    \]
    Hence $t (u) \coloneq \tau_{P (u)|_Y} (Z) = e (1 - l_Y u)/l_H$.
    Let $P (u)|_Y - v Z = P (u, v) + N (u, v)$ be the Zariski decomposition with $P (u, v)$ positive and $N (u, v)$ negative parts, respectively.
    For $v \in [0, t (u)]$, we have
    \[
    \begin{split}
        P (u, v) &= P (u)|_Y - u Z \sim_{\mbQ} (1 - l_Y u - l_H v/e) A|_Y, \\
        N (u, v) &= 0.
    \end{split}
    \]
    Let $\Delta_Z$ be the $\mbQ$-divisor on $Z$ such that $(K_Y + Z)|_Z = K_Z + \Delta_Z$.
    By \cite[Corollary 4.18]{Fujita23}, we have
    \[
    \delta_P (X) \ge \min \left\{ \frac{1}{S_A (Y)}, \ \frac{1}{S (V_{\bullet, \bullet}^Y; Z)}, \ \frac{1 - \ord_P (\Delta_Z)}{S (W_{\bullet, \bullet, \bullet}^{Y, Z}; P)} \right\},
    \]
    where $\Delta_Z$ is the divisor on $Z$ such that $(K_Y + Z)|_Z = K_Z + \Delta_Z$.
    The singularity $P \in Y$ is a cyclic quotient singularity of type $\frac{1}{r_P} (a, b)$ for a suitable $a, b$ since $Y$ is quasismooth at $P$. Then we have $\ord_P (\Delta_Z) = (r_P-1)/r_P$ by \cite[Proposition 16.6]{Kol+92} since the pair $(Y, Z)$ is plt.
    
    We have
    \[
    \begin{split}
        S_A (Y) &= \frac{1}{(A^3)} \int_0^{\frac{1}{l_Y}} \vol_X (-K_X - u Y) d u \\
        &= \frac{1}{(A^3)} \int_0^{\frac{1}{l_Y}} (1 - l_Y u)^3 (A^3) d u \\
        &= \frac{1}{4 l_Y}.
    \end{split}
    \]
    By \cite[Theorem 4.8]{Fujita23}, we have
    \[
    \begin{split}
        S (V_{\bullet, \bullet}^Y; Z) 
        &= \frac{3}{(A^3)} \int_0^{\frac{1}{l_Y}} \int_0^{t (u)} \vol_Y (P (u)|_Y - v Z) d v d u \\
        &= \frac{3}{(A^3)} \int_0^{\frac{1}{l_Y}} \int_0^{t (u)} \left(1 - l_Y u - \frac{l_H}{e} v \right)^2 l_Y (A^3) d v d u \\
        &= \frac{e}{4 l_H}.
    \end{split}
    \]
    We have $F_P (W_{\bullet, \bullet, \bullet}^{Y, Z}) = 0$ (see \cite[Definition 4.16]{Fujita23} for the definition) since $N (u, v) = 0$ for $v \in [0, t (u)]$, and by \cite[Theorem 4.17]{Fujita23}, we have
    \[
    \begin{split}
        S (W_{\bullet, \bullet, \bullet}^{Y, Z}; P)
        &= \frac{3}{(A^3)} \int_0^{\frac{1}{l_Y}} \int_0^{t (u)} ((P (u, v) \cdot Z))^2 d v d u \\
        &= \frac{3}{(A^3)} \int_0^{\frac{1}{l_Y}} \int_0^{t (u)} \left( \frac{l_Y l_H}{e} \left(1 - l_Y u - \frac{l_H}{e} v \right) (A^3) \right)^2 d v d u \\
        &= \frac{l_Y l_H (A^3)}{4 e}.
    \end{split}
    \]
    The assertion follows from the above computations.
\end{proof}

\begin{Rem} \label{rem:normal}
In Section \ref{sect: singular points}, we will consider various flags to obtain lower bounds of local delta invariants.
The surfaces $Y$ we will take as a part of such flags are hypersurfaces $Y \in |l_Y A|$ for some $l_Y$ with isolated singularities. 
In particular they are all weighted complete intersections that are regular in codimension $1$.
In this remark we explain the normality of such varieties.

Let $X \subset \mbP (a_0, \dots, a_N)$ be a weighted complete intersection of dimension at least $2$.
Let $U_i \coloneq (x_i \ne 0) \subset X$, $0 \le i \le N$, be  standard open charts of $X$.
Then $U_i$ is the quotient of a complete intersection subvariety $\tilde{U}_i \subset \mbA^N$ by the natural action of the cyclic group of order $a_i$.
The affine variety $\tilde{U}_i$ is Cohen-Macaulay since it is a complete intersection.
In particular it satisfies the Serre's $S_2$ condition.
If $X$ is regular in codimension $1$, then so is $\tilde{U}_i$ for any $i$, which implies that $\tilde{U}_i$ is normal.
Thus $U_i$ are normal since it is the quotient of a normal variety, and so is $X$.
\end{Rem}

\subsection{Flag of type $\mathrm{IIa}$}
\label{sec:flagIIa}

\begin{Def} \label{def:flagIIa}
    A \textit{flag of type $\mathrm{IIa}$} is a flag $P \in \Gamma \subset Y \subset X$ with the following properties: $Y$ is a normal hypersurface in $X$ such that it is quasismooth at $P$ and the pair $(X, Y)$ is plt, $\Gamma$ is an irreducible quasismooth curve such that the pair $(Y, \Gamma)$ is plt and there is a hypersurface $H$ on $X$ such that $H|_Y = m \Gamma + n \Delta$, where $m$ and $n$ are positive integers and $\Delta = \sum_{i=1}^k \Delta_i$ is a sum of distinct prime divisors other than $\Gamma$ satisfying the following properties:
    \begin{enumerate}
        \item $(\Gamma \cdot \Delta_i) = (\Gamma \cdot \Delta_j)$ and $(\Delta_i^2) = (\Delta_j^2)$ for any $i, j$;
        \item $(\Delta_i \cdot \Delta_j) = (\Delta_{i'} \cdot \Delta_{j'})$ for any $i, j, i', j'$ with $i \ne j$ and $i' \ne j'$;
        \item the intersection matrix of $\Delta$ is negative definite.
    \end{enumerate}
\end{Def} 

We explain how to estimate $\delta_P (X)$ by considering a flag of type $\mathrm{IIa}$.
Let $l_Y$ and $l_H$ be positive integers such that $Y \sim l_Y A$ and $H \sim l_H A$.
We have $\tau_A (Y) = 1/l_Y$ and, for the Zariski decomposition $-K_X - u Y = P (u) + N (u)$ with $P (u)$ positive and $N (u)$ negative parts, respectively, we have
\[
\begin{split}
    P (u) &= -K_X - u Y = (1- l_Y u)A, \\
    N (u) &= 0
\end{split}
\]
for $0 \le u \le 1/l_Y$.
To go further, we need to understand the pseudo-effective threshold and the Zariski decomposition of the divisor 
\[
P (u)|_Y - v \Gamma = \left(\frac{m}{l_H} (1 - l_Y u) - v \right) \Gamma + \frac{n}{l_H} (1 - l_Y u) \Delta 
\] 
on $Y$.
We set
\[
\lambda \coloneq (\Gamma^2), \ 
\mu \coloneq - (\Delta^2), \ 
\nu \coloneq (\Gamma \cdot \Delta).
\]

\begin{Lem}
    We have $m \nu > n \mu > 0$.
\end{Lem}

\begin{proof}
    We have $\mu > 0$ since the intersection matrix of $\Delta$ is negative definite.
    We have $0 < (H|_Y \cdot n \Delta) = ((m \Gamma + n \Delta) \cdot \Delta) = m \nu - n \mu$ since $H|_Y$ is ample.
\end{proof}

\begin{Lem} \label{lem:smptflagIIa}
    Let $\alpha$ and $\beta$ be real numbers such that $\beta \ge 0$.
    \begin{enumerate}
        \item The divisor $\alpha \Gamma + \beta \Delta$ is nef if and only if $-\frac{\lambda}{\nu} \alpha \le \beta \le \frac{\nu}{\mu} \alpha$.
        \item The divisor $\alpha \Gamma + \beta \Delta$ is pseudo-effective if and only if $\alpha \ge 0$.
    \end{enumerate}
\end{Lem} 

\begin{proof}
    We set $D = \alpha \Gamma + \beta \Delta$.
    The condition $-\frac{\lambda}{\nu} \alpha \le \beta \le \frac{\nu}{\mu} \alpha$ is equivalent to the condition that $(D \cdot \Gamma) \ge 0$ and $(D \cdot \Delta) \ge 0$.
    Hence if $D$ is nef, then $-\frac{\lambda}{\nu} \alpha \le \beta \le \frac{\nu}{\mu} \alpha$.
    Suppose that $-\frac{\lambda}{\nu} \alpha \le \beta \le \frac{\nu}{\mu} \alpha$.
    Note that $\alpha, \beta \ge 0$ since $\beta \ge 0$.
    It follows that $D$ is an effective divisor.
    For $1 \le i, j \le k$, we have $(D \cdot \Delta_i) = (D \cdot \Delta_j)$ by the assumptions (1) and (2) on $\Delta$.
    Hence the inequality $(D \cdot \Delta) \ge 0$ implies $(D \cdot \Delta_i) \ge 0$ for any $i$, and this together with $(D \cdot \Gamma) \ge 0$ imply that $D$ is nef.
    Thus, the assertion (1) is proved.
    
    Suppose that $D$ is pseudo-effective.
    The existence of ample divisor $H|_Y = m \Gamma + n \Delta$ supported on $\Gamma \cup \Delta$ with $n \ge 0$ implies that $-\frac{\lambda}{\nu} < \frac{\nu}{\mu}$, that is, $\nu^2 + \lambda \mu > 0$.
    We have $(D \cdot (\mu \Gamma + \nu \Delta)) = \alpha (\nu^2 + \lambda \mu) \ge 0$ since $\mu \Gamma + \nu \Delta$ is nef by (1).
    This shows $\alpha \ge 0$ and the assertion (2) is proved.
\end{proof}

We set 
\[
P_Y \coloneq \Gamma + \frac{\nu}{\mu} \Delta,
\]
which is a nef divisor on $Y$ such that $(P_Y \cdot \Delta) = 0$.
We have 
\[
t (u) \coloneq \tau_{P (u)|_Y} (\Gamma) = \frac{m}{l_H} (1 - l_Y u)
\]
for $0 \le u \le 1/l_Y$ by (2) of Lemma \ref{lem:smptflagIIa}.

\begin{Lem} \label{lem:smptNZdecIIa}
    Suppose that $0 \le u \le 1/l_Y$.
    Then, for the Zariski decomposition of $P (u)|_Y - v \Gamma = P (u, v) + N (u, v)$ with $P (u, v)$ positive and $N (u, v)$ negative parts, respectively, we have
    \[
    \begin{split}
        P (u, v) &=
        \begin{dcases}
            (t (u) - v) \Gamma + \frac{n t (u)}{m} \Delta, & \left(0 \le v \le \frac{m \nu - n \mu}{m \nu} t (u) \right), \\
            (t (u) - v) P_Y, & \left(\frac{m \nu - n \mu}{m \nu} t (u) \le v \le t (u) \right),
        \end{dcases} \\
        N (u, v) &= 
        \begin{dcases}
            0, & \left(0 \le v \le \frac{m \nu - n \mu}{m \nu} t (u) \right), \\
            \frac{m \nu v - (m \nu - n \mu) t (u)}{m \mu} \Delta, & \left(\frac{m \nu - n \mu}{m \nu} t (u) \le v \le t (u) \right),
        \end{dcases}
    \end{split}
    \]
\end{Lem} 

\begin{proof}
    Note that $(P_Y \cdot \Delta) = 0$ and this implies $(P_Y \cdot \Delta_i) = 0$  for any $i = 1, \dots, k$ by (1) and (2) of Definition \ref{def:flagIIa}.
    The assertion follows from Lemma \ref{lem:smptflagIIa} and the assumption that the intersection matrix of $\Delta$ is negative definite.
\end{proof}

\begin{Prop} \label{prop:flagIIadelta}
    Let $P \in \Gamma \subset Y \subset X$ be a flag of type $\mathrm{IIa}$.
    Then we have
    \[
    \delta_P (X) \ge \min \left\{ 4 l_Y, \ \frac{1}{S (V_{\bullet, \bullet}^Y; \Gamma)}, \ \frac{1}{r_P S (W_{\bullet, \bullet, \bullet}^{Y, \Gamma}; P)} \right\}.
    \]
    The numbers $S (V_{\bullet, \bullet}^Y; \Gamma)$ and $S (W_{\bullet, \bullet, \bullet}^{Y, \Gamma}; P)$ are computed as follows:
    \[
    \begin{split}
        S (V_{\bullet, \bullet}^Y;\Gamma) &= \frac{3}{(A^3)} \int_0^{\frac{1}{l_Y}} \left( \int_0^{\frac{m \nu - n\mu}{m \nu} t(u)} \left(C t(u)^2 - \frac{2(n \nu + m \lambda)}{m} t(u) v + \lambda v^2 \right) d v \right. \\
        & \hspace{47mm} \left. + \int_{\frac{m \nu- n\mu}{m \nu}t(u)}^{t(u)} \frac{\nu^2+\lambda \mu}{\mu} (t(u)-v)^2 d v \right) d u, \\
        S (W_{\bullet, \bullet, \bullet}^{Y, \Gamma}; P) &= \frac{3}{(A^3)} \int_0^{\frac{1}{l_Y}} \left(\int_0^{\frac{m \nu-n\mu}{m \nu}t(u)} \left(\frac{n\nu + m \lambda}{m} t (u) - \lambda v \right)^2 d v \right. \\
        & \hspace{15mm} \left. + \int_{\frac{m \nu-n\mu}{m \nu}t(u)}^{t(u)} \left(\frac{\nu^2+\lambda\mu}{\mu} \right)^2 (t(u)-v)^2 d v \right) d u + F_P (W_{\bullet, \bullet, \bullet}^{Y,\Gamma}),
    \end{split}
    \]
    where
    \[
    \begin{split}
        C &= \frac{2 m n \nu + m^2 \lambda - n^2 \mu}{m^2}, \\
        F_P (W_{\bullet, \bullet, \bullet}^{Y,\Gamma}) &= \frac{n^3}{4 l_Y l_H^3 (A^3)} \cdot \frac{\mu (\nu^2 + \lambda \mu)}{\nu^2} \cdot \ord_P (\Delta|_{\Gamma}).
    \end{split}
    \]
\end{Prop} 

\begin{proof}
    By \cite[Corollary 4.18]{Fujita23}, we have the estimate
    \[
    \delta_P (X) \ge \min \left\{ \frac{1}{S_A (Y)}, \ \frac{1}{S (V_{\bullet, \bullet}^Y; \Gamma)}, \ \frac{1}{r_P S (W_{\bullet, \bullet, \bullet}^{Y, \Gamma}; P)} \right\},
    \]
    where $S_A (Y)$, $S \left(V_{\bullet, \bullet}^Y; \Gamma \right)$ and $S (W_{\bullet, \bullet, \bullet}^{Y, \Gamma}; P)$ can be computed as follows:
    \[
    \begin{split}
        S_{A} (Y) 
        &= \frac{1}{(A^3)} \int_0^{\frac{1}{l_Y}} \vol_X (A - u Y) d u = \int_0^{\frac{1}{l_Y}} (1 - l_Y u)^3 d u = \frac{1}{4 l_Y}, \\
        S \left(V_{\bullet, \bullet}^Y; \Gamma \right) 
        &= \frac{3}{(A^3)} \int_0^{\frac{1}{l_Y}} \int_0^{t (u)} \vol_Y (P (u)|_Y - v \Gamma) d v d u \\
        &= \frac{3}{(A^3)} \int_0^{\frac{1}{l_Y}} \left( \int_0^{\frac{m \nu-n\mu}{m \nu} t (u)} \left( (t (u)-v) \Gamma + \frac{n t (u)}{m} \Delta \right)^2 d v \right.  \\ 
        & \hspace{45mm} \left. + \int_{\frac{m \nu-n\mu}{m \nu} t (u)}^{t (u)} (t (u)-v)^2 ({P_Y}^2) d v \right) d u, \\
        S \left(W_{\bullet, \bullet, \bullet}^{Y, \Gamma}; P \right) 
        &= \frac{3}{(A^3)} \int_0^{\frac{1}{l_Y}} \int_0^{t (u)} (P (u, v) \cdot \Gamma)^2 d v d u \\
        &= \frac{3}{(A^3)} \int_0^{\frac{1}{l_Y}} \left( \int_0^{\frac{m \nu - n\mu}{m \nu} t (u)} \left(((t (u) - v) (\Gamma^2) + \frac{n t (u)}{m} (\Delta \cdot \Gamma) \right)^2 d v \right. \\
        & \hspace{20mm} \left. + \int_{\frac{m \nu - n\mu}{m \nu} t (u)}^{t (u)} (t (u) - v)^2 (P_Y \cdot \Gamma)^2 d v \right) d u + F_P \left(W_{\bullet, \bullet, \bullet}^{Y, \Gamma} \right),
    \end{split}
    \]
    where
    \[
    \begin{split}
        & F_P \left(W_{\bullet, \bullet, \bullet}^{Y, \Gamma} \right) \\
        &= \frac{6}{(A^3)} \int_0^{\frac{1}{l_Y}} \int_{\frac{m \nu - n \mu}{m \nu} t (u)}^{t (u)} (P (u, v) \cdot \Gamma) \cdot \ord_P (N (u, v)|_{\Gamma}) d v d u \\
        &=\frac{6 (P_Y \cdot \Gamma) \ord_P (\Delta|_{\Gamma})}{(A^3)} \int_0^{\frac{1}{l_Y}}\int_{\frac{m \nu - n \mu}{m \nu} t (u)}^{t (u)} (t (u) -v) \cdot \frac{m \nu v - (m \nu - n\mu)t (u)}{m \mu} d v d u \\
        &= \frac{n^3}{4 l_Y l_H^3 (A^3)} \cdot \frac{\mu (\nu^2 + \lambda \mu)}{\nu^2} \cdot \ord_P (\Delta|_{\Gamma}). 
    \end{split}
    \]
    The assertions follow from direct computations.
\end{proof}

Therefore we can estimate $\delta_P (X)$ as above once we know the intersection numbers $\lambda = (\Gamma^2)$, $\mu = -(\Delta^2)$ and $\nu = (\Gamma \cdot \Delta)$.
We explain how to compute them.

\begin{Rem}
    Let $X = X_d \subset \mbP (a_0, \dots, a_4)$ be a quasismooth Fano $3$-fold, $Y \subset X$ a normal hypersurface and $\Gamma \subset Y$ an irreducible and reduced curve.
    We assume that $\Gamma$ is quasismooth and $Y$ is quasismooth along $\Gamma$.
    Let $P_1, \dots, P_m$ be the singular points of $Y$ along $\Gamma$ and let $r_i$ be the index of the singularity $P_i \in X$.
    The pair $(Y, \Gamma)$ is plt along $\Gamma$ and we have 
    \[
    (K_Y + \Gamma)|_{\Gamma} = K_{\Gamma} + \sum_{i=1}^m \frac{r_i - 1}{r_i} P_i.
    \]
    Then we have
    \[
    (\Gamma^2) = - (K_Y \cdot \Gamma) + (2 p_a (\Gamma) - 2) + \sum_{i=1}^m \frac{r_i - 1}{r_i}.
    \]
\end{Rem}

\subsection{Flag of type $\mathrm{IIb}$}

\begin{Def}
    A \textit{flag of type $\mathrm{IIb}$} is a flag $P \in \Gamma \subset Y \subset X$ with the following properties: $Y$ is a normal hypersurface of $X$ such that it is quasismooth at $P$ and the pair $(X, Y)$ is plt, $\Gamma$ is an irreducible quasismooth curve such that the pair $(Y, \Gamma)$ is plt and there is a hypersurface $H$ on $X$ such that $H|_Y = m \Gamma + n \Delta$, where $m, n$ are positive integers, $\Delta = \sum_{i = 1}^k \Delta_i$ is a sum of distinct prime divisors other than $\Gamma$ satisfying the following properties:
    \begin{enumerate}
        \item $(\Gamma \cdot \Delta_i) = (\Gamma \cdot \Delta_j)$ and $(\Delta_i^2) = (\Delta_j^2)$ for any $i, j$;
        \item $(\Delta_i \cdot \Delta_j) = (\Delta_{i'} \cdot \Delta_{j'})$ for any $i, j, i', j'$ such that $i \ne j$ and $i' \ne j'$;
        \item $(\Delta^2) = 0$.
    \end{enumerate}
\end{Def}

We explain how to estimate $\delta_P (X)$ by considering a flag of type $\mathrm{IIb}$.

Let $l_Y$ and $l_H$ be positive integers such that $Y \sim l_Y A$ and $H \sim l_H A$.
We have $\tau_A (Y) = 1/l_Y$ and, for the Zariski decomposition $-K_X - u Y = P (u) + N (u)$ with $P (u)$ positive and $N (u)$ negative parts, respectively, we have
\[
\begin{split}
    P (u) &= -K_X - u Y = (1-l_Y u)A, \\
    N (u) &= 0,
\end{split}
\]
for $0 \le u \le 1/l_Y$.
We need to understand the pseudo-effective threshold and the Zariski decomposition of the divisor
\[
P (u)|_Y - v \Gamma = \left(\frac{m}{l_H} (1-l_Y u) - v \right) \Gamma + \frac{n}{l_H}(1 - l_Y u) \Delta
\]
on $Y$.
We set
\[
\lambda \coloneq (\Gamma^2), \ 
\nu \coloneq (\Gamma \cdot \Delta).
\]

\begin{Lem} \label{lem:smptnumIIb}
    We have $\nu > 0$ and $n \nu > - m \lambda$.
\end{Lem} 

\begin{proof}
    The divisor $H|_Y = m \Gamma + n \Delta$ is ample on $Y$.
    This in particular implies that the support of $\Gamma + \Delta$ is connected, which implies that $\nu > 0$.
    Moreover we have $((m \Gamma + n \Delta) \cdot \Gamma) = m \lambda + n \nu > 0$.
\end{proof}

\begin{Lem} \label{lem:smptflagIIb}
    Let $\alpha$ and $\beta$ be rational numbers such that $\beta \ge 0$.
    \begin{enumerate}
        \item The divisor $\alpha \Gamma + \beta \Delta$ is nef if and only if $\alpha \ge 0$ and $\beta \ge - \frac{\lambda}{\nu} \alpha$.
        \item The divisor $\alpha \Gamma + \beta \Delta$ is pseudo-effective if and only if $\alpha \ge 0$.
    \end{enumerate}
\end{Lem} 

\begin{proof}
    We set $D = \alpha \Gamma + \beta \Delta$.
    We prove assertion (1).
    The conditions $\beta \ge - \frac{\lambda}{\nu} \alpha$ and $\alpha \ge 0$ are equivalent to the conditions that $(D \cdot \Gamma) \ge 0$ and $(D \cdot \Delta) \ge 0$.
    If $\alpha \ge 0$ and $\beta \ge - \frac{\lambda}{\nu} \alpha$, then it is clear that $D$ is nef since $\beta \ge 0$, $(D \cdot \Gamma) \ge 0$ and $(D \cdot \Delta_i) \ge 0$ for any $i$.
    The rest of the proof of (1) and the proof of (2) are similar to those of Lemma \ref{lem:smptflagIIa} and we leave them to readers.
\end{proof}

We have
\[
t (u) \coloneq \tau_{P (u)|_Y} (\Gamma) = \frac{m}{l_H} (1 - l_Y u),
\]
for $0 \le u \le 1/l_Y$ by (2) of Lemma \ref{lem:smptflagIIb}.

\begin{Lem} \label{lem:smptNZdecIIb}
    Suppose that $0 \le u \le 1/l_Y$.
    Then, for the Zariski decomposition $P (u)|_Y - v \Gamma = P (u, v) + N (u, v)$ with $P (u, v)$ positive and $N (u, v)$ negative parts, respectively, we have
    \[
    \begin{split}
        P (u, v) &= (t (u) - v) \Gamma + \frac{n}{m} t (u) \Delta, \\
        N (u, v) &= 0,
    \end{split}
    \]
    for $0 \le v \le t (u)$.
\end{Lem} 

\begin{proof}
    By Lemmas \ref{lem:smptnumIIb} and \ref{lem:smptflagIIb}, the divisor $P (u)|_Y - v \Gamma = (t (u) - v) \Gamma +\frac{n}{m} t (u) \Delta$ is nef for $0 \le v \le t (u)$, which proves the assertion.
\end{proof}

\begin{Prop} \label{prop:flagIIbdelta}
   Let $P \in \Gamma \subset Y \subset X$ be a flag of type $\mathrm{IIb}$.
   Then
   \[
   \delta_P \ge \min \left\{ 4 l_Y, \ \frac{4 l_Y l_H^3 (A^3)}{m^2 (3 n \nu + m \lambda)}, \ \frac{4 l_Y l_H^3 (A^3)}{r_P m (m^2 \lambda^2 + 3mn \lambda \nu + 3 n^2 \nu^2)} \right\}.
   \] 
\end{Prop} 

\begin{proof}
    By \cite[Corollary 4.18]{Fujita23}, we have the estimate
    \[
    \delta_P (X) \ge \min \left\{ \frac{1}{S_A (Y)}, \ \frac{1}{S (V_{\bullet, \bullet}^Y; \Gamma)}, \ \frac{1}{r_P S (W_{\bullet, \bullet, \bullet}^{Y, \Gamma}; P)} \right\},
    \]
    where $S_A (Y)$, $S (V_{\bullet, \bullet}^Y; \Gamma)$ and $S (W_{\bullet, \bullet, \bullet}^{Y, \Gamma}; P)$ can be computed as follows:
    \[
    \begin{split}
        S_{A} (Y) 
        &= \frac{1}{(A^3)} \int_0^{\frac{1}{l_Y}} \vol_X (A - u Y) d u = \int_0^{\frac{1}{l_Y}} (1 - l_Y u)^3 d u = \frac{1}{4 l_Y}, \\
        S (V_{\bullet, \bullet}^Y; \Gamma) 
        &= \frac{3}{(A^3)} \int_0^{\frac{1}{l_Y}} \int_0^{t (u)} \vol_Y (P (u)|_Y - v \Gamma) d v d u \\
        &= \frac{3}{(A^3)} \int_0^{\frac{1}{l_Y}} \left( \int_0^{t (u)} \left( (t (u)-v) \Gamma + \frac{n}{m} t (u) \Delta \right)^2 d v \right) du \\
        &= \frac{m^2 (3 n \nu + m \lambda)}{4 l_Y l_H^3 (A^3)}, \\
        S (W_{\bullet, \bullet, \bullet}^{Y, \Gamma}; P) 
        &= \frac{3}{(A^3)} \int_0^{\frac{1}{l_Y}} \int_0^{t (u)} (P (u, v) \cdot \Gamma)^2 d v d u \\
        &= \frac{3}{(A^3)} \int_0^{\frac{1}{l_Y}} \int_0^{t (u)} \left( (t (u) - v) (\Gamma^2) + \frac{n}{m} t (u) (\Delta \cdot \Gamma) \right)^2 d v du \\
        &= \frac{m (m^2 \lambda^2 + 3mn \lambda \nu + 3 n^2 \nu^2)}{4 l_Y l_H^3 (A^3)}.
    \end{split}
    \]
    Note that $F_P (W_{\bullet, \bullet, \bullet}^{Y, \Gamma}) = 0$ since $N (u, v) = 0$ for $0 \le v \le t (u)$.
\end{proof}

\subsection{Generalized flag of blow-up type}

We consider a generalized flag involving a weighted blow-up to estimate delta invariants at some points.

\begin{Def} \label{def:wBL}
Let $P \in X$ be a point.
A \textit{flag of type $\mathrm{wBL}$ centered at $P$} is a generalized flag of the form $Q \in C \to Y \subset X$ with the following properties: $Y$ is a normal hypersurface in $X$ passing through $P$ such that it is quasismooth at $P$ and the pair $(X, Y)$ is plt.
There is a weighted blow-up $\psi \colon \tilde{Y} \to Y$ with irreducible exceptional divisor $C$ such that the pair $(\tilde{Y}, C)$ is plt.
Finally, $Q$ is a point on $C$ and it runs  over the closed points of $C$.
\end{Def}

In this section, we consider a specific type of singular points and bound local delta invariants at these points by considering suitable flags of type $\mathrm{wBL}$.

Let $e$ and $r$ be positive integers with $r \ge 2$ and set $d = 2 r + e$.
Let $X = X_d \subset \mbP (1, 1, r-1, e, r)$ be a quasismooth Fano $3$-fold weighted hypersurface of index $1$ and set $P = P_w$.
Let $x_1, x_2, y, t, w$ be homogeneous coordinates of $\mbP (1, 1, r-1, e, r)$ of weights $1, 1, r-1, e, r$, respectively.
The point $P$ is a $\mathrm{QI}$ center and we assume that it is nondegenerate, that is, there is a non-biregular birational involution $\iota \colon X \ratmap X$ which sits in the commutative diagram
\[
\xymatrix{
& \ar[ld]_{\varphi} \tilde{X} \ar@{-->}[r]^{\tilde{\iota}} & \tilde{X} \ar[rd]^{\varphi} & \\
X \ar@{-->}[rrr]_{\iota} & & & X}
\]
where $\varphi \colon \tilde{X} \to X$ is the Kawamata blow-up of $X$ at $P$ and $\tilde{\iota}$ is a flop.

The polynomial $\msF$ defining $X$ can be written as
\[
\msF = w^2 t + w f_{r+e} (x_1, x_2, y) + f_{2 r + e} (x_1, x_2, y, t),
\]
where $f_{e+r} (x_1, x_2, y)$ and $f_{2 r + e} (x_1, x_2, y, t)$ are homogeneous polynomials of degree $e + r$ and $2 r + e$, respectively.
Let $Y \in |(r-1)A|$ be a general member.

\begin{Lem} 
\label{lem:wBLYnormal}
The surface $Y$ is quasismooth at $P$ and its singularity at $P$ is of type $\frac{1}{r} (1, 1)$.
Moreover, $Y$ is quasismooth outside a finite set of points.
In particular, $Y$ is a normal $\mbQ$-Cartier prime divisor such that $(K_X + Y)|_Y = K_Y$.
\end{Lem}

\begin{proof}
We see that $Y$ is quasismooth outside the set $\Bs |(r-1)A|$ and we have 
\[
\Bs |(r-1)A| \subset (x_1 = x_2 = y = 0)_X.
\]
It is easy to see that $(x_1 = x_2 = y = 0)_X$ is a finite set of points, hence so is $\Bs |(r-1)A|$.
It is also straightforward to see that $Y$ is quasismooth at $P = P_w$ and the singularity $P \in Y$ is of type $\frac{1}{r} (1, 1)$ since the defining equation of $Y$ is of the form $y + (\text{other terms}) = \msF = 0$.
The rest follows immediately from these observations (see Remark~\ref{rem:normal}).
\end{proof}

We assume that the pair $(X, Y)$ is plt, that is, $Y$ has a quotient singularity at $P_t$.
This assumption is satisfied in the following cases.

\begin{Lem} \label{lem:flagBLqsmY}
    Suppose that one of the following holds.
    \begin{enumerate}
        \item $e \mid r-1$.
        \item $e \mid r-2$. $($This includes the case of $r = 2$.$)$
        \item $e \mid 2r$.
        \item $e \nmid r+1$.
    \end{enumerate}
    Then $Y$ is quasismooth.
\end{Lem}

\begin{proof}
        It is enough to show that a general $Y \in |(r-1)A|$ is quasismooth along the base locus $\Bs |(r-1)A|$.
    Let $\msG = 0$ be the equation which defines $Y$ in $X$.
    It is easy to see that $Y$ is quasismooth at $P = P_w$ since $y \in \msG$.
    
    If $e \mid r - 1$, then $\Bs |(r-1)A| = (x_1 = x_2 = y = t = 0) = \{P\}$, and hence $Y$ is quasismooth.
    
    In the following, we assume that $e \nmid r-1$.
    Then $\Bs |(r-1)A| = (x_1 = x_2 = y = 0)_X$.
    
    For a point $Q \in \Bs |(r-1)A|$, we have $(\prt \msF/\prt t) (Q) = 0$ if and only if $Q = P_t$ since $w^2 t \in \msF$.
    This shows that $Y$ is quasismooth outside $P_t$.
    It remains to prove that either $P_t \notin \Bs |(r-1) A|$ or $Y$ is quasismooth at $P_t$.
    
    Suppose that $e \mid r-2$.
    Then we have $r-1 = m e + 1$ for some integer $m \ge 0$.
    Hence we can write $\msG = \alpha y + \beta t^m x_1 + \gamma t^m x_2 + \cdots$ for general $\alpha, \beta, \gamma \in \mbC$.
    It is then easy to see that $Y$ is quasismooth at $P_t$.
    
    Suppose that $e \mid 2r$.
    Suppose in addition that $e \ne r$.
    By the quasismoothness of $X$ at $P_t$, we have either $t^n \in \msF$ for some positive integer $n$ or $t^n v \in \msF$ for some positive integer $n$ and some coordinate $v \in \{x_1, x_2, y\}$.
    In the latter case we have $d = n e + \deg v$ which is impossible since $e \mid d$ and $e \nmid r-1$.
    This shows that $P_t \notin X$ in this case.
    If $e = r$, then either $t^3 \in \msF$ or $t^2 w \in \msF$ by the quasismoothness of $X$.
    It is straightforward to check that $Y$ is quasismooth at $P_t$ in the case when $P_t \in X$.
    
    In the following we assume that $e \nmid 2r$.
    Then $t^m \notin \msF$ for any positive integer $m$.
    In this case we have $\Bs |(r-1)A| = \{P_w, P_t\}$.
    By the quasismoothness of $X$, there exists a coordinate $v \in \{x_1, x_2, y, w\}$ and a positive integer $m$ such that $t^m v \in \msF$.
    By the assumption (4) and by the structure of $\msF$, $v \ne y, w$.
    It follows that $Y$ is quasismooth at $P_t$ since $y \in \msG$.
\end{proof}

\begin{Lem} \label{lem:flagBLnonqsmY}
Suppose that $(r, e) = (3, 4)$.
Then $Y$ is quasismooth outside $P_t$ and the pair $(X, Y)$ is plt.
\end{Lem}

\begin{proof}
By the proof of Lemma~\ref{lem:flagBLqsmY}, we see that $Y$ is quasismooth outside $P_t$.
Thus it is enough to show that $P_t \in Y$ is a cyclic quotient singularity.
The equation which defines $Y$ in $X$ can be written as $y - q (x_1, x_2) = 0$, where $q (x_1, x_2)$ is a general quadratic form in $x_1, x_2$.
By eliminating the variable $y$ and plugging $t =1$, the open set $Y \cap U_t$ is isomorphic to the quotient of the hypersurface
\[
\left((w^2 + w f_{r+e} (x_1, x_2, q) + f_{2r+e} (x_1, x_2, q, 1) = 0) \subset \mbA^3_{x_1,x_2,w} \right)/\bmu_4 (1, 1, 3),
\]
and the point $P_t$ correspond to the image of the origin under the quotient morphism.
We explain the above notation. 
The cyclic group of the $4$th roots of unity is denoted by $\bmu_4$, and the notation ``$/\bmu_4 (1, 1, 3)$" means that it is the quotient by the action determined by $(x_1, x_2, w) \mapsto (\zeta x_1, \zeta x_2, \zeta^3 w)$, where $\zeta$ is a primitive $4$th root of unity.  
We have $t^2 y \in \msF$, hence the germ $P_t \in Y$ is equivalent to the germ 
\[
o \in S \coloneq \left((w^2 - x_1 x_2 = 0) \subset \mbA^3_{x_1,x_2,w}\right)/\bmu_4 (1, 1, 3).
\]
The hypersurface $(w^2 - x_1 x_2 = 0) \subset \mbA^3$ is isomorphic to the quotient $\mbA^2_{s_1, s_2}/\bmu_2 (1, 1)$.
Thus $S$ is isomorphic to the iterated quotient $(\mbA^2_{s_1, s_2}/\bmu_2 (1, 1))/\bmu_4 (1, 1, 3)$ and we see that it is isomorphic to $\mbA^2_{s_1, s_2}/\bmu_8 (1, 5)$.
This shows that $P_t \in Y$ is a quotient singularity and we conclude that the pair $(X, Y)$ is plt.
\end{proof}

The singularity of $Y$ at $P$ is of type $\frac{1}{r} (1, 1)$.
Let $\tilde{Y}$ be the proper transform of $Y$ on $\tilde{X}$ and set $\psi \coloneq \varphi|_{\tilde{Y}} \colon \tilde{Y} \to Y$.
The birational morphism $\psi$ is a weighted blow-up of $Y$ at $P$ with weight $\frac{1}{r} (1, 1)$.
We set $C \coloneq E|_{\tilde{Y}}$, where $E$ is the exceptional divisor of $\varphi$.
We have $E \cong \mbP (1, 1, r-1)$ and $C \cong \mbP^1$.
We also set $A_Y \coloneq A|_Y$.
We have
\[
\begin{split}
    ((\psi^*A_Y)^2) &= (A_Y^2) = (A^2 \cdot (r-1) A) = \frac{2 r + e}{e r}, \\
    (C^2) &= ((E|_{\tilde{Y}})^2) = - \frac{r-1}{r} (E^3) = - r.
\end{split}
\]

We estimate $\delta_P (X)$ by considering a flag $Q \in C \to Y \subset X$ of type $\mathrm{wBL}$ centered at $P$, where $Q \in C$ runs over the points of $C$.
We have $\tau \coloneq \tau_A (Y) = 1/(r-1)$ and, for the Zariski decomposition $-K_X - u Y = P (u) + N (u)$ with $P (u)$ positive and $N (u)$ negative parts, respectively, we have
\[
\begin{split}
    P (u) &= (1 - (r-1)u) A, \\
    N (u) &= 0,
\end{split}
\]
for $0 \le u \le \tau$.
We need to understand the Zariski decomposition of the divisor
\[
\psi^*(P (u)|_Y) - v C = (1 - (r-1) u) \psi^*A_Y - v C.
\]

We set 
\[
\begin{split}
    P_{\tilde{Y}} &\coloneq \psi_*^{-1} (w t + f_{r + e} = 0)_Y \sim_{\mbQ} (e + r) \psi^*A_Y - \frac{d}{r} C, \\
    N_{\tilde{Y}} &\coloneq \psi_*^{-1} (t = 0)_Y \sim_{\mbQ} e \psi^*A_Y - \frac{e+r}{r} C,
\end{split}
\]

\begin{Lem} \label{lem:BLNprime}
    The divisor $N_{\tilde{Y}}$ is a prime divisor.
\end{Lem}

\begin{proof}
    The flopping curves are precisely the proper transform on $\tilde{X}$ of the curves in $(t = f_{r + e} = f_{2 r + e} = 0) \subset X$, where no curve in the latter set is contained in $\Bs |(r-1)A| = (x_1 = x_2 = y = 0)_X$.
    It follows that the surface $\tilde{Y}$ does not contain any flopping curve since $Y \in |(r-1)A|$ is general, and $\tilde{\iota}_*\tilde{Y} = \tilde{Y}$.
    Then the birational involution $\tilde{\iota} \colon \tilde{X} \ratmap \tilde{X}$ induces a biregular involution $\tilde{\iota}|_{\tilde{Y}} \colon \tilde{Y} \to \tilde{Y}$.
    The divisor $\varphi_*^{-1} (t = 0)_X$ is the divisor $\tilde{\iota}_*E$ and hence it is irreducible and reduced.
    It follows that its restriction $\psi_*^{-1} (t = 0)_Y$ is irreducible and reduced as well.
\end{proof}

\begin{Lem} \label{lem:BLnefpsef}
    The divisor $P_{\tilde{Y}}$ is nef.
    We have $(P_{\tilde{Y}} \cdot N_{\tilde{Y}}) = 0$ and $(N_{\tilde{Y}}^2) < 0$.
    \begin{enumerate}
        \item The divisor $\psi^*A_Y - c C$ is nef if and only if $0 \le c \le \frac{d}{(e + r) r}$.
        \item The divisor $\psi^*A_Y - c C$ is pseudo-effective if and only if $c \le \frac{e+r}{e r}$.
    \end{enumerate}
\end{Lem}

\begin{proof}
    The divisors $P_{\tilde{Y}}$, $\psi_*^{-1} (t x_1^r = 0)|_Y$ and $\psi_*^{-1} (t x_2^r = 0)|_Y$ are contained in the linear system $|P_{\tilde{Y}}|$.
    This implies that the base locus of $|P_{\tilde{Y}}|$ is a finite set of points, and hence $P_{\tilde{Y}}$ is nef.
    We have
    \[
    \begin{split}
    (N_{\tilde{Y}}^2) &= e^2 \frac{2 r + e}{e r} - \frac{(e+r)^2}{r^2} r 
    = - r \\
    (P_{\tilde{Y}} \cdot N_{\tilde{Y}}) &= e (e+r) (A_Y^2) + \frac{(2 r + e)(e + r)}{r^2} (C^2) = 0.
    \end{split}
    \]
    This shows that $P_{\tilde{Y}}$ and $N_{\tilde{Y}}$ spans extremal rays of the nef cone and the pseudo-effective cone of $\tilde{Y}$, respectively.
    This proves all the assertions.
\end{proof}

We have
\[
t (u) \coloneq \tau_{\psi^*(P (u)|_Y)} (C) = \frac{e+r}{e r} (1 - (r-1) u)
\]
for $0 \le u \le \tau$.

\begin{Lem}
    Suppose that $0 \le u \le \tau$.
    We define $P (u, v)$ and $N (u, v)$ as follows:
    \[
    \begin{split}
        P (u, v) &\coloneq
        \begin{dcases}
            \frac{er}{e+r} t (u) \psi^*A_Y - v C, & \left(0 \le v \le \frac{e d}{(e+r)^2} t (u) \right), \\
            \theta_P P_{\tilde{Y}}, & \left(\frac{e d}{(e+r)^2} t (u) \le v \le t (u) \right), 
        \end{dcases} \\
        N (u, v) &\coloneq 
        \begin{dcases}
            0, & \left(0 \le v \le \frac{e d}{(e+r)^2} t (u) \right), \\
            \theta_N N_{\tilde{Y}}, & \left(\frac{e d}{(e+r)^2} t (u) \le v \le t (u) \right),
        \end{dcases}
    \end{split}
    \]
    where
    \[
    \begin{split}
        \theta_P &\coloneq \frac{(e + r)(1 - (r-1)u) - e r v}{r^2} = \frac{e}{r} (t (u) - v), \\
        \theta_N &\coloneq \frac{r (e + r) v - d (1 - (r-1) u)}{r^2} = \frac{e+r}{r} \left(v - \frac{ed}{(e+r)^2} t (u) \right).
    \end{split}
    \]
    Then $\psi^*(P (u)|_Y) - v C = P (u, v) + N (u, v)$ gives the Zariski decomposition with $P (u, v)$ positive and $N (u, v)$ negative parts, respectively.
\end{Lem}

\begin{proof}
    This follows from Lemmas \ref{lem:BLNprime} and \ref{lem:BLnefpsef}.
\end{proof}

\begin{Lem} \label{lem:genflagord}
    For any point $Q \in C$, we have
    \[
    \ord_Q (N_{\tilde{Y}}|_C) \le e + r.
    \]
\end{Lem} 

\begin{proof}
    We can choose $\{x_1, x_2, y\}$ as a system of local orbifold coordinates of $X$ at $P$ and the Kawamata blow-up $\varphi$ can be realized as the weighted blow-up with weight $\wt (x_1, x_2, y) = \frac{1}{r} (1, 1, r-1)$.
    We have $\ord_E (t) = (e+r)/r$.
    Let $\Phi \colon \tilde{\mbP} \to \mbP$ be the weighted blow-up of $\mbP$ at $P$ with weights 
    \[
    \wt (x_1, x_2, y, t) = \frac{1}{r} (1, 1, r-1, e+r).
    \]
    Then we can identify $\tilde{X}$ with the proper transform $\Phi_*^{-1} X$ and we can identify $\varphi$ to the restriction $\Phi|_{\tilde{X}}$. 
    By a slight abuse of notation, we can identify the exceptional divisor of $\Phi$ with $\mbP (1, 1, r-1, e+r)$ with homogeneous coordinates $x_1, x_2, y, t$ of weights $1, 1, r-1, e+r$, respectively.
    We have an isomorphisms
    \[
    \begin{split}
    E &\cong (t + f_{e+r} (x_1, x_2, y) = 0) \subset \mbP (1, 1, r-1, e+r) \\
    &\cong \mbP (1, 1, r-1).
    \end{split}
    \]
    Let $\msG = 0$ be the equation which defines $Y$ in $X$.
    We can write $\msG = y - g_{r-1} (x_1, x_2, t)$ for a general homogeneous polynomial $g_{r-1} (x_1, x_2, t)$ of degree $r-1$. 
    Then, by eliminating the variable $y$, $C$ is isomorphic to the weighted hypersurface
    \[
    (t + f_{e + r} (x_1, x_2, \bar{g}_{r-1}) = 0) \subset \mbP (1, 1, e+r),
    \]
    where $\bar{g}_{r-1} = g_{r-1} (x_1, x_2, 0)$.
    We have $f_{e + r} (x_1, x_2, \bar{g}_e) \ne 0$ since $f_{e + r} (x_1, x_2, y) \ne 0$ and $\bar{g}_{r-1}$ is general.
    The scheme-theoretic intersection $N_{\tilde{Y}} \cap C$ is isomorphic to the scheme
    \[
    \begin{split}
        & (t = f_{e + r} (x_1, x_2, \bar{g}_{r-1}) = 0) \subset \mbP (1, 1, e+r) \\
        \cong \ & (f_{e + r} (x_1, x_2, \bar{g}_{r-1}) = 0) \subset \mbP^1,
    \end{split}
    \]
    which consists of $r + e$ points counting with multiplicity.
    Thus $\ord_Q (N_{\tilde{Y}}|_C) \le e + r$ for any point $Q \in C$.
\end{proof}

In the setting of this section, a flag $Q \in C \to Y \subset X$, where $Q \in C$ is a point, will be referred to a \textit{flag of type $\mathrm{wBL}_{(r, e)}$ centered at $P$}.

\begin{Prop} \label{prop:flagBLdelta}
    Let $Q \in C \to Y \subset X$ be a flag of type $\mathrm{wBL}_{(r,e)}$ centered at $P$.
    Then we have
    \[
    \delta_P (X) \ge \frac{4e}{e+r}.
    \]
\end{Prop} 

\begin{proof}
    By \cite[Corollary 4.18]{Fujita23}, we have the estimate
    \[
    \delta_P (X) \ge 
    \left\{ \frac{1}{S_A (Y)}, \ \frac{2}{r S (V_{\bullet, \bullet}^Y; C)}, \ \inf_{Q \in C} \frac{1}{S (W_{\bullet, \bullet, \bullet}^{Y, C}; Q)} \right\}.
    \]
    We compute (or estimate) $S_A (Y)$, $S (Y_{\bullet, \bullet}^Y;C)$ and $S (W_{\bullet, \bullet, \bullet}^{Y,C};Q)$ as follows.
    We have
    \[
    S_A (Y) = \frac{1}{(A^3)} \int_0^{\frac{1}{r-1}} \vol_X (A - u Y) d u
    = \frac{1}{4 (r-1)}
    \]
    and
    \[
    \begin{split}
        S (V_{\bullet, \bullet}^Y; C) &= \frac{3}{(A^3)} \int_0^{\frac{1}{r-1}} \left( \int_0^{\frac{e d}{(e+r)^2} t (u)} \left(\frac{er}{e+r} t (u) \psi^*A_Y - v C \right)^2 d v + \right. \\
        & \left. \hspace{55mm} + \int_{\frac{e d}{(e+r)^2} t (u)}^{t (u)} \theta_P^2 (P_{\tilde{Y}}^2) d v \right) d u. \\
        &= \frac{3}{(A^3)} \int_0^{\frac{1}{r-1}} \left(\int_0^{\frac{ed}{(e+r)^2} t (u)} \left(\frac{erd}{(e+r)^2} t (u)^2 - r v^2 \right) d v +\right. \\
        & \left. \hspace{55mm} + \frac{ed}{r} \int_{\frac{ed}{(e+r)^2} t (u)}^{t (u)} (t (u) - v)^2 \right) d v \\
        &= \frac{3}{(A^3)} \int_0^{\frac{1}{r-1}} \left( \frac{re^2d^2}{(e+r)^4} - \frac{r e^3 d^3}{3(e+r)^6} + \frac{r^5ed}{3(e+r)^6} \right) t (u)^3 d u \\
        &= \frac{3 (r-1)re}{d} \frac{red(3ed(e+r)^2-e^2d^2+r^4)}{3(e+r)^6} \frac{(e+r)^3}{e^3r^3} \frac{1}{4(r-1)} \\
        &= \frac{r^2 + 4 r e + 2 e^2}{4er(e+r)}.
    \end{split}
    \]
    We give an estimate of $S (W_{\bullet, \bullet, \bullet}^{Y, C}; Q)$.
    For any $Q \in C$, we have $\ord_Q (N_{\tilde{Y}}|_C) \le r + e$ by Lemma \ref{lem:genflagord}, and hence
    \[
    \begin{split}
        F_Q (W_{\bullet, \bullet, \bullet}^{Y, C}) &=
        \frac{6}{(A^3)} \int_0^{\tau} \int_{\frac{e d}{(e + r)^2} t (u)}^{t (u)} (2 r + e) \theta_P \theta_N \ord_P (N_{\tilde{Y}}|_C) d v d u \\
        &\le \frac{6ed(e+r)^2}{r^2(A^3)} \int_0^{\frac{1}{r-1}} \int_{\frac{e d}{(e + r)^2} t (u)}^{t (u)} (t(u)-v)\left(v - \frac{ed}{(e+r)^2} t(u) \right) d v d u \\
        &=\frac{6ed(e+r)^2}{r^2(A^3)} \int_0^{\frac{1}{r-1}} \frac{r^6}{6(e+r)^6} t (u)^3 d u \\
        &=\frac{6ed(e+r)^2}{r^2} \frac{(r-1)er}{d} \frac{r^6}{6(e+r)^6} \frac{(e+r)^3}{e^3r^3} \frac{1}{4(r-1)} \\
        &= \frac{r^2}{4e(e+r)}.
    \end{split}
    \] 
    Thus we have
    \[
    \begin{split}
        & S (W_{\bullet, \bullet, \bullet}^{Y, C}; Q) \\
        &\le \frac{3}{(A^3)} \int_0^{\frac{1}{r-1}} \left( \int_0^{\frac{e d}{(e + r)^2} t (u)} r^2 v^2 d v + \int_{\frac{e d}{(e + r)^2} t (u)}^{t (u)} \frac{e^2d^2}{r^2} (t(u)-v)^2 d v \right) d u + \\
        & \hspace{100mm} + \frac{r^2}{4e(e+r)} \\
        &= \frac{3}{(A^3)} \left(\frac{1}{3} \frac{r^2e^3d^3}{(e+r)^6} + \frac{e^2d^2}{r^2} \frac{r^6}{3(e+r)^6} \right) \int_0^{\frac{1}{r-1}} t(u)^3du + \frac{r^2}{4e(e+r)} \\   
        &= \frac{3(r-1)re}{d} \frac{r^2e^2d^2(ed+r^2)}
        {3(e+r)^6} \frac{(e+r)^3}{e^3r^3} \frac{1}{4(r-1)} + \frac{r^2}{4e(e+r)} \\
        &= \frac{d(ed+r^2)}{4(e+r)^3} + \frac{r^2}{4e(e+r)} \\
        &= \frac{e+r}{4e}.
    \end{split}
    \]
    Therefore we have
    \[
    \delta_P (X) \ge \min \left\{ 4 (r-1), \ \frac{8e(e+r)}{r^2+4re+2e^2}, \ \frac{4e}{r+e} \right\}.
    \]
    It is straightforward to check that
    \[
    \frac{4e}{r+e} < \frac{8e(e+r)}{r^2+4re+2e^2},\quad
     \frac{4e}{r+e} < 4 \le 4(r-1).
    \]
    This proves the assertion.
\end{proof}

\begin{Rem}
In view of the further generalization \cite[Theorem~11.14]{Fujita23} by Fujita, for the flag $Q \in C \to Y \subset X$ as above, we do not actually need the assumption that $(X, Y)$ is (globally) plt.
The inequality on $\delta_P (X)$ given in Proposition~\ref{prop:flagBLdelta} holds if $(X, Y)$ is plt around $P$.
Thus, in view of Lemma~\ref{lem:wBLYnormal}, Lemmas~\ref{lem:flagBLqsmY} and \ref{lem:flagBLnonqsmY} are redundant in this sense.
\end{Rem}
 
\section{Local delta at maximal centers} \label{sect: singular points}

This section is entirely devoted to the proof of the following.

\begin{Thm} \label{thm:deltasingpt}
    Let $X$ be a member of family \textnumero $\msi$, where $\msi \in \msI_{\BR}^{***}$.
    Then, 
    \[
    \delta_P (X) > 1
    \]
    for any singular point $P \in X$ which is a maximal center.
\end{Thm}

By Corollary \ref{cor:KstBR}, Theorem \ref{thm: Main Theorem} follows from Theorems \ref{thm:deltasingpt} and \ref{cor:KstBR}.

\subsection{Family \textnumero 2: $X_5 \subset \mbP (1, 1, 1, 1, 2)$}

Let $X$ be a member of family \textnumero 2.
Then $\Sing (X) = \{\frac{1}{2} (1, 1, 1)^{\QI}\}$.
Let $P$ be the $\frac{1}{2} (1, 1, 1)$ point and suppose that $P$ is nondegenerate.
    Then, by considering the flag of type $\mathrm{wBL}_{(2,1)}$ centered at $P$, we obtain $\delta_P (X) \ge 4/3$ by Proposition \ref{prop:flagBLdelta}.
Therefore Theorem \ref{thm:deltasingpt} is proved for family \textnumero 2.

\begin{table}[h]
\caption{\textnumero \! 2: $X_5 \subset \mbP (1, 1, 1, 1, 2)$}
\label{table:No2}
\centering
\begin{tabular}{cllll}
\toprule
Point & Case & $\delta_P$ & flag & Ref. \\
\midrule
$\frac{1}{2} (1, 1, 1)^{\QI}$ & ndgn & $\delta_P \ge 4/3$ & $\mathrm{wBL}_{(2,1)}$ & \\
\bottomrule
\end{tabular}
\end{table}

%
%
%

\subsection{Family \textnumero 5: $X_7 \subset \mbP (1, 1, 1, 2, 3)$}

Let $X$ be a member of family \textnumero 5.
Then $\Sing (X) = \{\frac{1}{2} (1, 1, 1)^{\QI}, \frac{1}{3} (1, 1, 2)^{\QI}\}$.
We have $\delta_P (X) > 1$ for any maximal center $P \in \Sing (X)$ by the following results.

\begin{itemize}
    \item[(i)] Let $P$ be the $\frac{1}{2} (1, 1, 1)$ point and suppose that $P$ is nondegenerate.
    Then $\alpha_P (X) \ge 6/7$ by \cite[Lemma 3.29]{KOW}, and hence $\delta_P (X) \ge 8/7$.
    \item[(ii)] Let $P$ be the $\frac{1}{3} (1, 1, 2)$ point.
    Suppose that $P$ is nondegenerate and $t^2 w \in \msF$.
    Let $Y, H \in |A|$ be general members.
    Then, by Lemma \ref{lem:N5sgqsm1}, $P \in Z \coloneq H \cap Y \subset Y \subset X$ is a flag of type $\mathrm{I}$ with $(l_Y, l_H, e, r_P) = (1,1,1,3)$ and we obtain $\delta_P (X) \ge 8/7$ by Proposition \ref{prop:flagIdelta}.
    \item[(iii)] Let $P$ be the $\frac{1}{3} (1, 1, 2)$ point and suppose that it is nondegenerate and that $t^2 w \notin \msF$.
    Then, we will show in \S \ref{Sec:No5singvi} that $\delta_P (X) \ge 56/43$ by applying Proposition~\ref{prop:flagIIbdelta} to a flag of type $\mathrm{IIb}$ with
    \[
    (l_Y, l_H, m, n, \lambda, \nu, r_P) = (1, 1, 1, 1, -5/6, 1, 3).
    \]
\end{itemize}


\begin{table}[h]
\caption{\textnumero \! 5: $X_7 \subset \mbP (1, 1, 1, 2, 3)$}
\label{table:No5}
\centering
\begin{tabular}{cllll}
\toprule
Point & Case & $\delta_P$ & flag & Ref. \\
\midrule
$\frac{1}{2} (1, 1, 1)^{\QI}$ & (i) ndgn & $\delta_P \ge 8/7$ & & \cite[3.29]{KOW} \\
\cmidrule{1-5}
$\frac{1}{3} (1, 1, 2)^{\QI}$ & (ii) ndgn, $\exists t^2 w$ & $\delta_P \ge 8/7$ & $\mathrm{I}$ &  \\
\cmidrule{2-5}
& (iii) ndgn, $\not\exists t^2 w$ & $\delta_P \ge 56/43$ & $\mathrm{IIb}$ & \S \ref{Sec:No5singvi} \\
\bottomrule
\end{tabular}
\end{table}

\subsubsection{Some results on quasismoothness}

\begin{Lem} \label{lem:N5sgqsm1}
    Suppose that $t^2 w \in \msF$.
    Let $Y, H \in |A|$ be general members.
    Then $Y$ and $Z \coloneq H \cap Y$ are both quasismooth.
\end{Lem}

\begin{proof}
    We have $\Bs |A| = (x = y = z =0)_X = \{P_t, P_w\}$ since $t^2 w \in \msF$.
    The equation which defines $Y$ in $X$ is of the form $\alpha x + \beta y + \gamma z = 0$, where $\alpha, \beta, \gamma \in \mbC$ are general.
    By the quasismoothness of $X$ at $P_w$, at least one of $w^2x, w^2y, w^2z$ appear in $\msF$ with nonzero coefficient.
    This shows that $Y$ is quasismooth at $P_t$ and $P_w$, and hence it is quasismooth.
    Set $Z \coloneq H \cap Y$.
    We can view $Z$ as a general member of $|A|_Y$ whose base locus is $\{P_t, P_w\}$.
    The equation which defines $H$ in $X$ is of the form $\zeta x + \eta y + \xi z = 0$, where $\zeta, \eta, \xi \in \mbC$ are general. 
    It is also straightforward to see that $Z$ is quasismooth at $P_t$ and $P_w$, and hence it is quasismooth.
\end{proof}

\begin{Lem} \label{lem:N5sgqsm2}
    Suppose that $t^2 w \notin \msF$.
    Let $Y, H \in |A|$ be general members.
    Then $Y$ is quasismooth and $H|_Y = \Gamma + \Delta$, where $\Gamma$ and $\Delta$ are irreducible and reduced curves with the following properties.
    \begin{itemize}
        \item $\Gamma$ is quasismooth.
        \item $\{P_t, P_w\} \subset \Gamma \setminus \Delta$.
        \item $(\Gamma^2) = -5/6$, $(\Delta^2) = 0$ and $(\Gamma \cdot \Delta) = 1$.
    \end{itemize}
\end{Lem}

\begin{proof}
    By the quasismoothness of $X$, we can write
    \[
    \msF = t^3 x + w^2 y + \msG,
    \]
    where $\msG \in (x, y, z)^2 \subset \mbC [x, y, z, t, w]$ is a homogeneous polynomial of degree $7$.
    Let $Y \in |A|$ be a general member.
    
    We observe that $Y$ is quasismooth.
    The base locus of $|A|$ is the curve $(x = y = z = 0) \subset X$ and $Y$ is cut by an equation $\alpha x + \beta y + \gamma z = 0$ for general $\alpha, \beta, \gamma \in \mbC$.
    It is then straightforward to see that $Y$ is quasismooth along $(x = y = z = 0)$ (this holds if $\gamma \ne 0$).
    Hence $Y$ is quasismooth.
    
    Let $H \in |A|$ be a general member.
    We may assume that $H \cap Y$ is cut by the equation $y = \zeta x, z = \eta x$, where $\zeta, \eta \in \mbC$ are general.
    Then $H|_Y = \Gamma + \Delta$, where
    \[
    \begin{split}
        \Gamma &\coloneq (x = y = z = 0), \\
        \Delta &\coloneq (y - \zeta x = z - \eta x = \zeta t^3 + \eta w^2 + \msG (x, \zeta x, \eta x, t, w)/x = 0).
    \end{split}
    \]
    We see that $\Gamma$ and $\Delta$ are irreducible and reduced curves and $\{P_t, P_w\} \subset \Gamma \setminus \Delta$.
    The curve $\Gamma$ is clearly quasismooth and it is a smooth rational curve.
    Moreover, $\Sing_{\Gamma} (Y) = \left\{\frac{1}{2} (1, 1), \frac{1}{3} (1, 2) \right\}$ and $K_Y \sim 0$.
    It follows that
    \[
    (\Gamma^2) = -2 + \frac{1}{2} + \frac{2}{3} = - \frac{5}{6}.
    \]
    By taking intersection number of $H|_Y = \Gamma + \Delta$ and $\Gamma$, and then $\Delta$, we have
    \[
    (\Gamma \cdot \Delta) = 1, \ 
    (\Delta^2) = 0,
    \]
    and the proof is complete.
\end{proof}

%
%

\subsubsection{Case (iii): $P$ is the $\frac{1}{3} (1,1,2)$ point which is nondegenerate, and $t^2 w \notin \msF$}
\label{Sec:No5singvi}

Let $P = P_w$ be the $\frac{1}{3} (1, 1, 2)$ point.
Suppose that $P$ is nondegenerate and $t^2 w \notin \msF$.
Let $Y, H \in |A|$ be general members and $H|_Y = \Gamma + \Delta$ be as in Lemma~\ref{lem:N5sgqsm2}.
Then, $P \in \Gamma \subset Y \subset X$ is a flag of type $\mathrm{IIb}$ with
\[
(l_Y, l_H, m, n, \lambda, \nu, r_P) = (1, 1, 1, 1, -5/6, 1, 3).
\]
By Proposition \ref{prop:flagIIbdelta}, we have
\[
\delta_P (X) \ge 
\min \left\{ 4, \ \frac{28}{13}, \ \frac{56}{43} \right\} = \frac{56}{43}.
\]
Therefore Theorem \ref{thm:deltasingpt} is proved for family \textnumero 5.

\subsection{Family \textnumero 12: $X_{10} \subset \mbP (1, 1, 2, 3, 4)$}

Let $X$ be a member of family \textnumero 12.
Then $\Sing (X) = \{2 \times \frac{1}{2} (1, 1, 1), \frac{1}{3} (1, 1, 2)^{\QI}, \frac{1}{4} (1, 1, 3)^{\QI}\}$.
We have $\delta_P (X) > 1$ for any maximal center $P \in \Sing (X)$ by the following results.
\begin{itemize}
    \item[(i)] Let $P$ be the $\frac{1}{3} (1, 1, 2)$ point and suppose that $P$ is nondegenerate.
    Then we have $\delta_P (X) \ge 16/7$ by applying Proposition \ref{prop:flagBLdelta} (see also Lemma~\ref{lem:flagBLnonqsmY}) to the flag of type $\mathrm{wBL}_{(3,4)}$ centered at $P$.
    \item[(ii)] Let $P$ be the $\frac{1}{4} (1, 1, 3)$ point and suppose that $P$ is nondegenerate.
    Then we have $\delta_P (X) \ge 4/3$ by applying Proposition~\ref{prop:flagBLdelta} to the flag of type $\mathrm{wBL}_{(4,2)}$ centered at $P$.
\end{itemize}

Therefore Theorem \ref{thm:deltasingpt} is proved for family \textnumero 12.

\begin{table}[h]
\caption{\textnumero \! 12: $X_{10} \subset \mbP (1, 1, 2, 3, 4)$}
\label{table:No7}
\centering
\begin{tabular}{cllll}
\toprule
Point & Case & $\delta_P$ & flag & Ref. \\
\midrule
$\frac{1}{3} (1, 1, 2)^{\QI}$ & (i) ndgn  & $\delta_P > 16/7$ & $\mathrm{wBL}_{(3,4)}$ & \ref{prop:flagBLdelta} \\
\cmidrule{1-5}
$\frac{1}{4} (1, 1, 3)^{\QI}$ & (ii) ndgn & $\delta_P \ge 4/3$ & $\mathrm{wBL}_{(4,2)}$ & \ref{prop:flagBLdelta} \\
\bottomrule
\end{tabular}
\end{table}

\subsection{Family \textnumero 13: $X_{11} \subset \mbP (1, 1, 2, 3, 5)$}

Let $X$ be a member of family \textnumero 13.
Then we have $\Sing (X) = \{\frac{1}{2} (1, 1, 1), \frac{1}{3} (1, 1, 2)^{\QI}, \frac{1}{5} (1, 2, 3)^{\QI}\}$.
We have $\delta_P (X) > 1$ for any maximal center $P \in \Sing (X)$ by the following results.
\begin{itemize}
    \item[(i)] Let $P$ be the $\frac{1}{3} (1, 1, 2)$ point and suppose that $P$ is nondegenerate.
    Then, $t^2 w \in \msF$.
    Let $Y \in |2 A|$ and $H \in |A|$ be general members.
    Then, by Lemma \ref{lem:No13qsm2}, $P \in Z \coloneq H \cap Y \subset Y \subset X$ is a flag of type $\mathrm{I}$ with $(l_Y, l_H, e, r_P) = (2, 1, 1, 3)$ and we have $\delta_P (X) \ge 20/11$ by Proposition \ref{prop:flagIdelta}.
    \item[(ii)] Let $P$ be the $\frac{1}{5} (1, 2, 3)$ point.
    Suppose that $t^2 w \in \msF$.
    Let $Y \in |2 A|$ and $H \in |A|$ be general members.
    Then, by Lemma \ref{lem:No13qsm2}, $P \in Z \coloneq Y \cap H \subset Y \subset X$ is a flag of type $\mathrm{I}$ with $(l_Y, l_H, e, r_P) = (2, 1, 1, 5)$ and we have $\delta_P (X) \ge 12/11$ by Proposition \ref{prop:flagIdelta}.
    \item[(iii)] Let $P$ be the $\frac{1}{5} (1, 2, 3)$ point.
    Suppose that $P$ is nondegenerate, $t^2 w \notin \msF$ and $z^3 w \in \msF$.
    It will be proved in \S \ref{Sec:No13singvi} that $\delta_P (X) \ge 132/101$ by considering a flag of type $\mathrm{IIa}$ with
    \[
    (l_Y, l_H, m, n, \lambda, \mu, \nu, r_P) = (1, 1, 1, 1, -8/15, 3/10, 3/5, 5).
    \]
    \item[(iv)] Let $P$ be the $\frac{1}{5} (1, 2, 3)$ point.
    Suppose that $P$ is nondegenerate, $t^2 w, z^3 w \notin \msF$ and $z^4 t \in \msF$.
    It will be proved in \S \ref{Sec:No13singvii} that $\delta_P (X) \ge 88/71$ by considering a suitable flag of type $\mathrm{wBL}$ centered at $P$.
    \item[(v)] Let $P$ be the $\frac{1}{5} (1, 2, 3)$ point.
    Suppose that $t^2 w \notin \msF$, $P$ is nondegenerate, $z^3 w \notin \msF$ and $z^4 t \notin \msF$.
    It will be proved in \S \ref{Sec:No13singviii} that $\delta_P (X) \ge 132/101$ by considering a flag of type $\mathrm{IIa}$ with
    \[
    (l_Y, l_H, m, n, \lambda, \mu, \nu, r_P) = (1, 1, 1, 3, -8/15, 1/30, 1/5, 5).
    \]
\end{itemize}

\begin{table}[h]
\caption{\textnumero \! 13: $X_{11} \subset \mbP (1, 1, 2, 3, 5)$}
\label{table:No13}
\centering
\begin{tabular}{cllll}
\toprule
Point & Case & $\delta_P$ & flag & Ref. \\
\midrule
$\frac{1}{3} (1, 1, 2)^{\QI}$ & (i) ndgn  & $\delta_P \ge 20/11$ & $\mathrm{I}$ &  \\
\cmidrule{1-5}
$\frac{1}{5} (1, 2, 3)^{\QI}$ & (ii) $t^2 w \in \msF$ & $\delta_P \ge 12/11$ & $\mathrm{I}$ & \\
\cmidrule{2-5}
& (iii) ndgn, $t^2 w \notin \msF$, $z^3 w \in \msF$ & $\delta_P \ge 132/101$ & $\mathrm{IIa}$ & \S \ref{Sec:No13singvi} \\
\cmidrule{2-5}
& (iv) ndgn, $t^2 w, z^3 w \notin \msF$, $z^4 t \in \msF$ & $\delta_P \ge 88/71$ & $\mathrm{wBL}$ & \S \ref{Sec:No13singvii} \\
\cmidrule{2-5}
& (v) ndgn, $t^2 w, z^3 w, z^4 t \notin \msF$ & $\delta_P \ge 132/101$ & $\mathrm{IIa}$ & \S \ref{Sec:No13singviii} \\
\bottomrule
\end{tabular}
\end{table}

\subsubsection{Some results on quasismoothness}

\begin{Lem} \label{lem:No13qsm2}
    Suppose that $t^2 w \in \msF$.
    Let $Y \in |2 A|$ and $H \in |A|$ be general members.
    Then $Y$ and $Z \coloneq Y \cap H$ are both quasismooth.
\end{Lem}

\begin{proof}
    The base locus of $|2A|$ is the set $(x = y = z = 0)_X = \{P_t, P_w\}$.
    We have $t^2 w \in \msF$ by the assumption and also we have either $w^2 x \in \msF$ or $w^2 y \in \msF$ since $X$ is quasismooth at $P_w$.
    A general $Y \in |2 A|$ is defined by $z + q (x, y) = 0$ in $X$ for some general quadric $q (x, y)$, and hence it is easy to see that $Y$ is quasismooth at $P_t$ and $P_w$.
    This shows that $Y$ is quasismooth.
    
    Let $H \in |A|$ be a general member which is defined by $\alpha x + \beta y = 0$ for some general $\alpha, \beta \in \mbC$.
    The intersection $Z \coloneq H \cap Y$ is a general member of the linear system $|A|_Y$ whose base locus is the set $(x = y = 0)_Y = \{P_t, P_w\}$.
    It is again easy to see that $Z$ is quasismooth at $P_t$ and $P_w$.
    Thus $Z$ is quasismooth.
\end{proof}

\begin{Lem} \label{lem:No13qsm3}
    Suppose that $t^2 w \notin \msF$ and either $z^3 w \in \msF$ or $z^4 t \in \msF$.
    Then a general $Y \in |A|$ is quasismooth.
\end{Lem}

\begin{proof}
    By the quasismoothness of $X$ at $P_t$ we have $t^3 z \in \msF$ since $t^2 w \notin \msF$ by the assumption.
    We can write
    \[
    \msF (0, 0, z, t, w) = \alpha w z^3 + t^3 z + \beta t z^4,
    \]
    where $\alpha, \beta \in \mbC$.

    Suppose first that $\alpha \ne 0$.
    Replacing $w \mapsto w - (\beta/\alpha)t z$ and then rescaling $z$, we may assume that
    \[
    \msF (0, 0, z, t, w) = w z^3 + t^3 z.
    \]
    The base scheme of the pencil $|A|$ is 
    \[
    (x = y = 0)_X = (x = y = w z^3 + t^3 z = 0),
    \]
    which is quasismooth outside $P_w$.
    It is obvious that $Y$ is quasismooth at $P_w$ because it is defined by $\gamma x + \delta y = 0$ for general $\gamma, \delta \in \mbC$.
    Hence $Y$ is quasismooth in this case.

    Suppose that $\alpha = 0$.
    Then $\beta \ne 0$ by the assumption and, by rescaling $z$, we have
    \[
    \msF (0, 0, z, t, w) = t^3 z + t z^4.
    \]
    Then the base scheme of the pencil $|A|$ is
    \[
    (x = y = 0)_X = (t^3 z + t z^4 = 0),
    \]
    which is quasismooth outside $P_w$.
    It is obvious that $Y$ is quasismooth at $P_w$.
    Therefore $Y$ is quasismooth.
\end{proof}

\begin{Lem} \label{lem:No13qsm4}
    Suppose that $t^2 w \notin \msF$ and $z^3 w \in \msF$ and let $Y, H \in |A|$ be general members.
    Then $H|_Y = \Gamma + \Delta$, where $\Gamma$ and $\Delta$ are irreducible and reduced curves with the following properties.
    \begin{enumerate}
        \item $\Gamma$ is quasismooth.
        \item $\Gamma \cap \Delta = \{P_w\}$.
        \item $(\Gamma^2) = -8/15$, $(\Delta^2) = -3/10$ and $(\Gamma \cdot \Delta) = 3/5$.
    \end{enumerate}
\end{Lem}

\begin{proof}
    The surface $Y$ is quasismooth by Lemma \ref{lem:No13qsm3}.
    By the quasismoothness of $X$ at $P_t$, we have $t^3 z \in \msF$ and we can write
    \[
    \msF (0,0,z,t,w) = w z^3 + t^3 z.
    \]
    We have $H|_Y = \Gamma + \Delta$, where
    \[
    \begin{split}
        \Gamma &\coloneq (x = y = z = 0), \\
        \Delta &\coloneq (x = y = w z^2 + t^3 = 0).
    \end{split}
    \]
    The curve $\Gamma$ is clearly quasismooth and $\Gamma \cap \Delta = \{P_w\}$.
    We have $K_Y \sim 0$ and $\Sing_{\Gamma} (Y) = \{\frac{1}{3} (1,2), \frac{1}{5} (2,3)\}$.
    It follows that
    \[
    (\Gamma^2) = -2 + \frac{2}{3} + \frac{4}{5} = - \frac{8}{15}.
    \]
    By taking intersection numbers of $H|_Y = \Gamma + \Delta$ and $\Gamma$, and then $\Delta$, we obtain $(\Delta^2) = - 3/10$ and $(\Gamma \cdot \Delta) = 3/5$.
\end{proof}

\begin{Lem} \label{lem:No13qsm5}
    Suppose that $t^2 w, z^3 w, z^4 t \notin \msF$ and let $Y, H \in |A|$ be general members.
    Then there exists a point $P_0 \in Y$ such that $Y$ is quasismooth outside $P_0$, $X$ is smooth at $P_0$ and $Y$ has a Du Val singularity of type $A_2$ at $P_0$.
    In particular $(X, Y)$ is plt.
    Moreover, $H|_Y = \Gamma + 3 \Delta$, where $\Gamma$ and $\Delta$ are quasismooth rational curves with the following properties.
    \begin{enumerate}
        \item $P_0 \in \Delta$ and $\Gamma \cap \Delta = \{P_w\}$.
        \item $(\Gamma^2) = -8/15$, $(\Delta^2) = - 1/30$ and $(\Gamma \cdot \Delta) = 1/5$.
    \end{enumerate}
\end{Lem}

\begin{proof}
    By the quasismoothness of $X$ at $P_t$, we have $t^3 z \in \msF$ since $t^2 w \notin \msF$.
    Then we can write
    \[
    \msF (0,0,z,t,w) = t^3 z.
    \]
    We have $H|_Y = \Gamma + 3 \Delta$, where
    \[
    \begin{split}
        \Gamma &\coloneq (x = y = z = 0), \\
        \Delta &\coloneq (x = y = t = 0).
    \end{split}
    \]
    The curves $\Gamma$ and $\Delta$ are clearly quasismooth and $\Gamma \cap \Delta = \{P_w\}$.

    We consider quasismoothness of $Y$.
    The surface $Y$ is quasismooth outside $\Delta$ since the base scheme $H \cap Y$ of $|A|$ is quasismooth outside $\Delta$.
    By the quasismoothness of $X$ at $P_w$, either $w^2 x \in \msF$ or $w^2 y \in \msF$, we may assume that $\coeff_{\msF} (w^2 x) = 1$ and $\coeff_{\msF} (w^2 y) = 0$ by replacing $x, y$ suitably.
    Then we can write
    \[
    \msF = x (w^2 + \alpha z^5) + \beta y z^5 + \msF', 
    \]
    where $\alpha, \beta \in \mbC$ and $\msF' \in (x,y,t)^2 \subset \mbC [x,y,z,t,w]$.
    We have $\beta \ne 0$ because otherwise $X$ is not quasismooth at the point $(w^2 + \alpha z^5 = 0) \cap \Delta$.
    We may assume that $\beta = 1$ and $\alpha = 0$ by replacing $y$ suitably and we can write
    \[
    \msF = x w^2 + y z^5 + \msF', 
    \]
    where $\msF' \in (x,y,t)^2 \subset \mbC [x,y,z,t,w]$.
    Let 
    \[
    \xi x + y= 0
    \]
    be the equation which defines $Y$ in $X$, where $\xi \in \mbC$ is general.
    It is easy to see that $Y$ is quasismooth outside the point $P_0 \coloneq (0\!:\!0\!:\!1\!:\!0\!:\!\sqrt{\xi}) \in Y$.
    By eliminating the variable $y$ by plugging $y = - \xi x$, $Y$ is isomorphic to the hypersurface in $\mbP (1_x,2_z,3_t,5_w)$ defined by the equation
    \[
    x (w^2 -\xi z^5) + \msF' (x,-\xi x,z,t,w) = 0.
    \]
    We can choose local coordinates $x, t, u$ of $\mbP (1,2,3,5)$ at $P_0$ such that global homogeneous coordinates $x, t$ restrict to the local coordinates $x, t$ and $w^2 - \xi z^5$ restricts to a polynomial in $u$ of the form $\sum_{i \ge 1} \zeta_i u^i$, where $\zeta_i \in \mbC$ with $\zeta_1 \ne 0$.
    We have $\overline{\msF}' \coloneq \msF' (x,-\xi x,z,t,w) \in (x, t)^2$, $t^2 w \notin \overline{\msF}'$ and $t^3 z \in \overline{\msF}'$.
    Moreover the homogeneous coordinate $z$ restricts to $\sum_{i \ge 0} \eta_i u^i$ with $\eta_0 \ne 0$.
    From this we see that $P_0 \in Y$ is equivalent to the singularity $o \in (x u - t^3 = 0) \subset \mbC^3_{x,t,u}$, which is a Du Val singularity of type $A_2$.

    The curve $\Gamma$ is a smooth rational curve, $K_Y \sim 0$ and $\Sing_{\Gamma} (Y) = \{\frac{1}{3} (1,2), \frac{1}{5} (2,3)\}$.
    It follows that
    \[
    (\Gamma^2) = -2 + \frac{2}{3} + \frac{4}{5} = - \frac{8}{15}.
    \]
    By taking intersection numbers of $H|_Y = \Gamma + 3 \Delta$ and $\Gamma$, and then $\Delta$, we obtain $(\Delta^2) = - 1/30$ and $(\Gamma \cdot \Delta) = 1/5$.
\end{proof}

\subsubsection{Case (iii): $P$ is the $\frac{1}{5} (1, 2, 3)$ point which is nondegenerate, $t^2 w \notin \msF$, $z^3 w \in \msF$}
\label{Sec:No13singvi}

Let $P = P_w$ be the $\frac{1}{5} (1, 2, 3)$ point which is nondegenerate.
We assume that $t^2 w \notin \msF$ and $z^3 w \in \msF$.
Let $Y \in |A|$ and $H \in |A|$ be general members.
By Lemma~\ref{lem:No13qsm4}, $P \in \Gamma \subset Y \subset X$ is a flag of type $\mathrm{IIa}$ with
\[
(l_Y, l_H, m, n, \lambda, \mu, \nu, r_P) = (1, 1, 1, 1, -8/15, 3/10, 3/5, 5).
\]

By the formulae given in Proposition \ref{prop:flagIIadelta}, we have
\[
\begin{split}
    S (V_{\bullet, \bullet}^Y; \Gamma)
    &= \frac{90}{11} \int_0^1 \left( \int_0^{\frac{1}{2} (1-u)} \left(\frac{11}{30} (1-u)^2 - \frac{2}{15} (1-u)v - \frac{8}{15} v^2 \right) d v + \right. \\
    & \hspace{25mm} \left. + \int_{\frac{1}{2} (1-u)}^{1-u} \frac{2}{3} (1-u-v)^2 d v \right) d u \\
    &= \frac{31}{88}, \\
    S (W_{\bullet, \bullet, \bullet}^{Y, \Gamma};P) &=
    \frac{90}{11} \int_0^1 \left(\int_0^{\frac{1}{2} (1-u)} \left(\frac{1-u+8v}{15}\right)^2 dv + \right. \\
    & \hspace{25mm} \left. + \int_{\frac{1}{2} (1-u)}^{1-u} \frac{4}{9} (1-u-v)^2 dv \right) du + F_P (W_{\bullet, \bullet, \bullet}^{Y, \Gamma}) \\
    &= \frac{101}{660},
\end{split}
\]
where $F_P (W_{\bullet, \bullet, \bullet}^{Y, \Gamma}) = 3/44$ since $\ord_P (\Delta|_{\Gamma}) = (\Gamma \cdot \Delta) = 3/5$. 
Thus we have
\[
\delta_P (X) \ge \min \left\{ 4, \ \frac{88}{31}, \ \frac{132}{101} \right\} = \frac{132}{101}.
\]

\subsubsection{Case (iv): $P$ is the $\frac{1}{5} (1, 2, 3)$ point which is nondegenerare, $t^2 w, z^3 w \notin \msF$ and $z^4 t \in \msF$}
\label{Sec:No13singvii}

Let $P = P_w$ be the the $\frac{1}{5} (1, 2, 3)$ point.
Suppose that $P$ is nondegenerate, $t^2 w, z^3 w \notin \msF$ and $z^4 t \in \msF$.
Let $Y \in |A|$ be a general member which is quasismooth by Lemma~\ref{lem:No13qsm3}.
Let $\varphi \colon \tilde{X} \to X$ be the Kawamata blow-up at $P$ with exceptional divisor $E$ and let $\psi \colon \tilde{Y} \to Y$ be the restriction of $\varphi$ to $\tilde{Y} \coloneq \varphi_*^{-1} Y$.
We set $A_Y \coloneq A|_Y$ and $C \coloneq E|_{\tilde{Y}}$.
We have
\begin{equation} \label{eq:No13singvii-1}
    (\psi^*A_Y)^2 = (A_Y^2) = (A^3) = \frac{11}{30}, \ (C^2) = (E^2 \cdot \tilde{Y}) = - \frac{1}{5} (E^3) = - \frac{5}{6}.
\end{equation}

We estimate $\delta_P (X)$ by considering a flag $Q \in C \to Y \subset X$ of type $\mathrm{wBL}$, where $Q$ runs over the closed points of $C$.
We have $\tau_A (Y) = 1$, and, for the Zariski decomposition $-K_X - u Y = P (u) + N (u)$ with $P (u)$ positive and $N (u)$ negative parts, respectively, we have
\[
P (u) = (1-u)A, \ 
N (u) = 0
\]
for $0 \le u \le 1$.
We need to understand the Zariski decomposition of
\[
\psi^*(P (u)|_Y) - v C = (1-u)\psi^*A_Y - v C.
\]

We can write
\[
\msF = w^2 x + w f_6 + f_{11},
\]
where $f_i \in \mbC [x, y, z, t]$ is a homogeneous polynomial of degree $i$.
Filtering off terms divisible by $x$ in $f_6$ and then replacing $w$ suitably, we may assume that $f_6 = f_6 (y, z, t)$ does not involve the variable $x$.
We have $t^2, z^3 \notin f_6$ by the assumption, and hence $f_6 = y f_5$ for some $f_5 = f_5 (y, z, t)$.
Since $t^3z, z^4 t \in \msF$, we can write
\[
\msF = w^2 x + w y f_5 + t^3 z - t z^4 + g_{11},
\]
where $g_{11} \in (x, y) \subset \mbC [x, y, z, t]$.
We set $H \coloneq H_x \in |A|$.
We have
\[
H|_Y = \Gamma + \Delta + \Xi,
\]
where 
\[
\begin{split}
    \Gamma &\coloneq (x = y = z = 0), \\
    \Delta &\coloneq (x = y = t = 0), \\
    \Xi &\coloneq (x = y = t^2 - z^3 = 0).
\end{split}
\]
For a divisor or a curve $\Sigma$ on $X$ (or on $Y$), we denote by $\tilde{\Sigma}$ its proper transform on $X$ (or on $Y$). 
We have
\[
\tilde{H}|_{\tilde{Y}} = \tilde{\Gamma} + \tilde{\Delta} + \tilde{\Xi} + C \sim \psi^*A_Y - \frac{6}{5} C.
\]

\begin{Lem} \label{Lem:No13singviiint}
    We have the following:
    \[
    (\tilde{\Gamma}^2) = -\frac{2}{3}, \ (\tilde{\Delta}^2) = -1, \ (\tilde{\Xi}^2) = -2,
    \]
    \[
    (\tilde{\Gamma} \cdot \tilde{\Delta}) = (\tilde{\Delta} \cdot \tilde{\Xi}) = (\tilde{\Xi} \cdot \tilde{\Gamma}) = 0,
    \]
    \[
    (\tilde{\Gamma} \cdot C) = \frac{1}{3} , \ (\tilde{\Delta} \cdot C) = \frac{1}{2}, \ (\tilde{\Xi} \cdot C) = 1.
    \]
\end{Lem}

\begin{proof}
    We see that $\Gamma \cap \Delta = \Delta \cap \Xi = \Xi \cap \Gamma = \{P\}$.
    The Kawamata blow-up $\varphi$ is realized as the weighted blow-up at $P \in X$ with weight $\wt (x, y, z, t) = \frac{1}{5} (6, 1, 2, 3)$.
    We have an isomorphism
    \[
    E \cong (x + y f_5 (y, z, t) = 0) \subset \mbP (6_x, 1_y, 2_z, 3_t),
    \]
    and hence
    \[
    C \cong E \cap (y = 0) \cong \mbP (2_z, 3_t).
    \]
    We have $\tilde{\Gamma} \cap C = \{Q_t\}$, $\tilde{\Delta} \cap C = \{Q_z\}$ and $\tilde{\Xi} \cap C = \{Q'\}$, where $Q_t = (0\!:\!1)$, $Q_z = (1\!:\!0)$ and $Q' = (1\!:\!1) \in C$.
    Note that $Q_t$ is a $\frac{1}{3} (1, 2)$ point, $Q_z$ is a $\frac{1}{2} (1, 1)$ point, and $Q'$ is a smooth point of $\tilde{Y}$.
    Hence $\tilde{\Gamma} \cap \tilde{\Delta} = \tilde{\Delta} \cap \tilde{\Xi} =  \tilde{\Xi} \cap \tilde{\Gamma} = \emptyset$ and
    \[
    (\tilde{\Gamma} \cdot \tilde{\Delta}) = (\tilde{\Delta} \cdot \tilde{\Xi}) = (\tilde{\Xi} \cdot \tilde{\Gamma}) = 0.
    \]

    We have $K_{\tilde{Y}} = \psi^*K_Y \sim 0$.
    The curve $\tilde{\Gamma}$ (resp.\ $\tilde{\Delta}$, resp.\ $\tilde{\Xi}$) is a smooth rational curve and $\Sing_{\tilde{\Gamma}} (\tilde{Y}) = \{2 \times \frac{1}{3} (1, 2)\}$ (resp.\ $\Sing_{\tilde{\Delta}} (\tilde{Y}) = \{2 \times \frac{1}{2} (1, 1)\}$, $\Sing_{\tilde{\Xi}} (\tilde{Y}) = \emptyset$).
    It follows that
    \[
    \begin{split}
        (\tilde{\Gamma}^2) &= -2 + \frac{2}{3} + \frac{2}{3} = - \frac{2}{3}, \\
        (\tilde{\Delta}^2) &= -2 + \frac{1}{2} + \frac{1}{2} = -1, \\
        (\tilde{\Xi}^2) &= -2.
    \end{split}
    \]
    The remaining intersection numbers can be computed by taking intersection numbers of $\tilde{H}|_{\tilde{Y}}$ and $\tilde{\Gamma}, \tilde{\Delta}, \tilde{\Xi}$.
\end{proof}

We define
\[
\begin{split}
    P_{\tilde{Y}} &\coloneq \psi^*A_Y - \frac{1}{5} C, \\
    N_{\tilde{Y}} &\coloneq \tilde{\Gamma} + \tilde{\Delta} + \tilde{\Xi} = \psi^*A_Y - \frac{11}{5} C.
\end{split}
\]

\begin{Lem} \label{lem:No13singviicone}
    The divisor $P_{\tilde{Y}}$ is nef, $(P_{\tilde{Y}} \cdot \tilde{\Gamma}) = (P_{\tilde{Y}} \cdot \tilde{\Delta}) = (P_{\tilde{Y}} \cdot \tilde{\Xi}) = 0$ and the intersection matrix of $N_{\tilde{Y}}$ is negative definite.
    \begin{enumerate}
        \item The divisor $\psi^*A_Y - c C$ is nef if and only if $c \le 1/5$.
        \item The divisor $\psi^*A_Y - c C$ is pseudo effective if and only if $c \le 11/5$.
    \end{enumerate}
\end{Lem}

\begin{proof}
    The proper transform of the linear system generated by homogeneous polynomials of degree $6$ in variables $x, y, z, t$ is a base point free linear subsystem of $6P_{\tilde{Y}}$.
    This shows that $P_{\tilde{Y}}$ is nef.
    By \eqref{eq:No13singvii-1} and Lemma \ref{Lem:No13singviiint}, it is straightforward to check the equalities $(P_{\tilde{Y}} \cdot \tilde{\Gamma}) = (P_{\tilde{Y}} \cdot \tilde{\Delta}) = (P_{\tilde{Y}} \cdot \tilde{\Xi}) = 0$ and check that the intersection matrix of $N_{\tilde{Y}}$ is negative definite.
    The assertions (1) and (2) follow from the fact that $P_{\tilde{Y}}$ is nef and $(P_{\tilde{Y}} \cdot N_{\tilde{Y}}) = 0$.
\end{proof}

\begin{Lem} \label{Lem:No13singviiF}
    For any point $Q \in C$, we have $\ord_Q (N_{\tilde{Y}}|_C) \le 1$.
\end{Lem}

\begin{proof}
    Note that any $2$ of $\tilde{\Gamma}, \tilde{\Delta}$ and $\tilde{\Xi}$ are disjoint.
    Let $\Sigma$ be one of $\tilde{\Gamma}, \tilde{\Delta}$ and $\tilde{\Xi}$.
    For any point $Q \in \Sigma$, we have $\ord_Q (\Sigma|_C) \le (C \cdot \Sigma) \le 1$.
    This proves the assertion.
\end{proof}

\begin{Lem}
    Suppose that $0 \le u \le 1$.
    We define $P (u, v)$ and $N (u, v)$ as follows:
    \[
    \begin{split}
        P (u, v) &=
        \begin{dcases}
            (1-u)\psi^*A_Y - v C, & \left(0 \le v \le \frac{1}{5} (1-u) \right), \\
            \frac{11 (1-u) - 5 v}{10} P_{\tilde{Y}}, & \left(\frac{1}{5} (1-u) \le v \le \frac{11}{5} (1-u) \right),
        \end{dcases} \\
        N (u, v) &=
        \begin{dcases}
            0, & \left(0 \le v \le \frac{1}{5} (1-u) \right), \\
            \frac{5 v - (1-u)}{10} N_{\tilde{Y}}, & \left(\frac{1}{5} (1-u) \le v \le \frac{11}{5} (1-u) \right).
        \end{dcases}
    \end{split}
    \]
    Then $\psi^*(P (u)|_Y) - v C = P (u, v) + N (u, v)$ gives the Zariski decomposition with $P (u, v)$ positive and $N (u, v)$ negative parts, respectively.
\end{Lem}

\begin{proof}
    This follows from Lemma \ref{lem:No13singviicone}.
\end{proof}

We have
\[
\begin{split}
    S (V_{\bullet, \bullet}^{Y};C) &= \frac{90}{11} \int_0^1 \left( \int_0^{\frac{1}{5} (1-u)} \left(\frac{11}{30} (1-u)^2 - \frac{5}{6} v^2 \right) d v + \right. \\ 
    & \left. \hspace{35mm} + \int_{\frac{1}{5} (1-u)}^{\frac{11}{5} (1-u)} \frac{1}{3} \left(\frac{11(1-u)-5v}{10}\right)^2 d v \right) d u \\
    &= \frac{3}{5}
\end{split}
\]
By Lemma \ref{Lem:No13singviiF}, we have
\[
\begin{split}
    F_Q (W_{\bullet, \bullet, \bullet}^{Y,C}) &\le \frac{180}{11} \int_0^1 \int_{\frac{1}{5}(1-u)}^{\frac{11}{5} (1-u)} \frac{1}{6} \cdot \frac{11(1-u)-5v}{10} \cdot \frac{5v - (1-u)}{10} d v \\
    &= \frac{5}{22},
\end{split}
\]
and hence
\[
\begin{split}
    S (W_{\bullet, \bullet, \bullet}^{Y,C};Q) &\le \frac{90}{11} \int_0^1 \left(\int_0^{\frac{1}{5} (1-u)} \left(\frac{5}{6} v \right)^2 d v + \right. \\
    & \left. \hspace{20mm} + \int_{\frac{1}{5} (1-u)}^{\frac{11}{5} (1-u)} \left(\frac{1}{6} \cdot \frac{11(1-u) - 5 v}{10} \right)^2 d v \right) d u + \frac{5}{22} \\
    &= \frac{71}{264}.
\end{split}
\]
for any point $Q \in C$.
We have $(K_{\tilde{Y}} + C)|_C = K_C + \frac{2}{3} Q_t + \frac{1}{2} Q_z \eqcolon K_C + \Delta_C$, and hence $A_{C, \Delta_C} (Q) \ge 1/3$ for any point $Q \in C$.
Therefore, we have
\[
\delta_P (X) \ge \min \left\{ 4, \ \frac{5}{3}, \ \frac{88}{71} \right\} = \frac{88}{71}.
\]

\subsubsection{Case (v): $P$ is the $\frac{1}{5} (1, 2, 3)$ point which is nondegenerate, $t^2 w, z^3 w, z^4 t \notin \msF$}
\label{Sec:No13singviii}

Let $P = P_w$ be the $\frac{1}{5} (1, 2, 3)$ point.
Let $Y, H \in |A|$ and $H|_Y = \Gamma + 3 \Delta$ be as in Lemma \ref{lem:No13qsm5}.
Then $P \in \Gamma \subset Y \subset X$ is a flag of type $\mathrm{IIa}$ with
\[
(l_Y, l_H, m, n, \lambda, \mu, \nu, r_P) = (1, 1, 1, 3, -8/15, 1/30, 1/5, 5).
\]
By the formulae given in Proposition \ref{prop:flagIIadelta}, we have
\[
\begin{split}
    S (V_{\bullet,\bullet}^Y; \Gamma) &= \frac{90}{11} \int_0^1 \left( \int_0^{\frac{1}{2} (1-u)} \left(\frac{11}{30} (1-u)^2 - \frac{2}{15} (1-u)v - \frac{8}{15} v^2 \right) d v \right. \\
    & \hspace{3cm} \left. + \int_{\frac{1}{2} (1-u)}^{1-u} \frac{2}{3} (1-u-v)^2 dv \right) d u \\
    &= \frac{31}{88}, \\
    S (W_{\bullet,\bullet,\bullet}^{Y,\Gamma};P) &= \frac{90}{11} \int_0^1 \left( \int_0^{\frac{1}{2} (1-u)} \left(\frac{1-u+8v}{15}\right)^2 dv \right. \\ 
    & \hspace{3cm} \left. + \int_{\frac{1}{2} (1-u)}^{1-u} \frac{4}{9} (1-u-v)^2 d v \right) d u + F_P (W_{\bullet,\bullet,\bullet}^{Y,\Gamma}) \\
    &= \frac{101}{660},
\end{split}
\]
where $F_P (W_{\bullet,\bullet,\bullet}^{Y,\Gamma}) = 3/44$ since $\ord_P (\Delta|_{\Gamma}) = 1/5$.

Thus, we have
\[
\delta_P (X) \ge \min \left\{4, \ \frac{88}{31}, \ \frac{132}{101} \right\} = \frac{132}{101}.
\]

Therefore Theorem \ref{thm:deltasingpt} is proved for family \textnumero 13.

\subsection{Family \textnumero 20: $X_{13} \subset \mbP (1, 1, 3, 4, 5)$}

Let $X$ be a member of family \textnumero 20.
Then $\Sing (X) = \{\frac{1}{3} (1,1,2)^{\EI}, \frac{1}{4} (1,1,3)^{\QI}, \frac{1}{5} (1,1,4)^{\QI}\}$.
We have $\delta_P (X) > 1$ for any maximal center $P \in \Sing (X)$ by the following results.

\begin{itemize}
    \item[(i)] Let $P$ be the $\frac{1}{3} (1,1,2)$ point.
    Let $Y \in |4A|$ and $H \in |A|$ be general members.
    Then, by Lemma \ref{lem:No20qsm1}, $P \in Z \coloneq H \cap Y \subset Y \subset X$ is a flag of type $\mathrm{I}$ with $(l_Y,l_H,e,r_P) = (4,1,1,3)$ and we have $\delta_P (X) \ge 20/13$ by Proposition~\ref{prop:flagIdelta}.
    \item[(ii)] Let $P$ be the $\frac{1}{4} (1,1,3)$ point.
    Suppose that $P$ is nondegenerate.
    Then $\alpha_P (X) \ge 10/13$ by \cite[\S 5.6.d]{KOW}, and hence $\delta_P \ge 40/39$.
    \item[(iii)] Let $P$ be the $\frac{1}{5} (1,1,4)$ point.
    Then $\delta_P (X) \ge 3/2$ by applying Proposition~\ref{prop:flagBLdelta} to the flag of type $\mathrm{wBL}_{(5,3)}$ centered at $P$.
\end{itemize}

\begin{table}[h]
\caption{\textnumero \! 20: $X_{13} \subset \mbP (1,1,3,4,5)$}
\label{table:No20}
\centering
\begin{tabular}{cllll}
\toprule
Point & Case & $\delta_P$ & flag & Ref. \\
\midrule
$4 \times \frac{1}{3} (1,1,2)^{\EI}$ & (i) & $\delta_P \ge 20/13$ & $\mathrm{I}$ & \\
\cmidrule{1-5}
$\frac{1}{4} (1,1,3)^{\QI}$ & (ii) ndgn & $\delta_P > 40/39$ & & \cite[\S 5.6.d]{KOW}  \\
\cmidrule{1-5}
$\frac{1}{5} (1,1,4)^{\QI}$ & (iii) & $\delta_P \ge 3/2$ & $\mathrm{wBL}_{(5,3)}$ & \ref{prop:flagBLdelta} \\
\bottomrule
\end{tabular}
\end{table}

\subsubsection{Some results on quasismoothness}

\begin{Lem} \label{lem:No20qsm1}
    Let $Y \in |4A|$ and $H \in |A|$ be general members.
    Then $Y$ and $Z \coloneq H \cap Y$ are both quasismooth.
\end{Lem}

\begin{proof}
    By the quasismoothness of $X$ at $P_w$, we have $w^2 z \in \msF$ and we may assume $\coeff_{w^2z} (\msF) = 1$.
    The base locus of $|4A|$ is $(x=y=t=0)_X = \{P_z,P_w\}$.
    We show that $Y$ is quasismooth at $P_z$ and $P_w$.
    We can write
    \[
    \msF = \alpha_x z^4 x + \alpha_y z^4 y + \alpha_t z^3 t + w^2 z + f,
    \]
    where $\alpha_x, \alpha_y, \alpha_t \in \mbC$ and $f = f (x,y,z,t,w) \in (x,y,t,w)^2 \cap (x,y,z,t)^2$.
    By the quasismoothness of $X$ at $P_z$, we have $(\alpha_x,\alpha_y,\alpha_t) \ne (0,0,0)$.
    Let 
    \[
    \msG \coloneq t + \beta_x z x + \beta_y z y + h_4 (x,y) = 0
    \]
    be the equation  which defines $Y$ in $X$, where $\beta_x, \beta_y \in \mbC$ are general and $h_4 = h_4 (x, y)$ is a general homogeneous polynomial of degree $4$.
    It is now straightforward to see that $Y$ is quasismooth at $P_z$ and $P_w$, and hence $Y$ is quasismooth.

    The curve $Z = H \cap Y$ is a general member of $|A|_Y$ whose base locus is $\{P_z,P_w\}$.
    Let $\msH \coloneq \gamma_x x + \gamma_y y = 0$ be the equation which defines $H$ in $X$.
    Then $Z$ is the complete intersection $(\msF = \msG = \msH = 0)$ and it is easy to see that $Z$ is quasismooth at $P_z$ and $P_w$, and hence $Z$ is quasismooth.
\end{proof}

Therefore Theorem \ref{thm:deltasingpt} is proved for family \textnumero 20.

\subsection{Family \textnumero 23: $X_{14} \subset \mbP (1, 2, 3, 4, 5)$}

Let $X$ be a member of family \textnumero 23.
Then $\Sing (X) = \{3 \times \frac{1}{2} (1,1,1), \frac{1}{3} (1,1,2)^{\mathrm{IEI}}, \frac{1}{4} (1,1,3)^{\EI}, \frac{1}{5} (1,2,3)^{\QI}\}$.
We have $\delta_P (X) > 1$ for any maximal center $P \in \Sing (X)$ by the following results.

\begin{itemize}
    \item[(i)] Let $P$ be the $\frac{1}{3} (1,1,2)$ point.
    Suppose that $z^3w, z^2t^2 \notin \msF$, that is, $P$ is not a maximal center (see Theorem~\ref{thm:QIcenter}).
    Then we will show in \S \ref{Sec:No23singiv} that $\delta_P (X) \ge 140/37$ by applying Proposition \ref{prop:flagIIadelta} to a flag of type $\mathrm{IIa}$ with
    \[
    (l_Y,l_H,m,n,\lambda,\mu,\nu,r_P) = (1,2,1,2,-8/15,1/12,1/3,3).
    \]
    \item[(ii)] Let $P$ be the $\frac{1}{4} (1,1,3)$ point.
    Suppose that $X$ does not contain a quasi-line of type $(1,3,4)$.
    Let $Y \in |3A|$ be a general member and set $H \coloneq H_x \in |A|$.
    By Lemma \ref{lem:N23sgqsm3}, $P \in Z \coloneq H \cap Y \subset Y \subset X$ is a flag of type $\mathrm{I}$ with $(l_Y,l_H,e,r_P) = (3,1,1,4)$ and we have $\delta_P (X) \ge 20/7$ by Proposition \ref{prop:flagIdelta}.
    \item[(iii)] Let $P$ be the $\frac{1}{4} (1,1,3)$ point.
    Suppose that $X$ contains a quasi-line of type $(1,3,4)$.
    It will be proved in \S \ref{Sec:No23singvi} that $\delta_P (X) \ge 35/12$ by considering a flag of type $\mathrm{IIa}$ with
    \[
    (l_Y,l_H,m,n,\lambda,\mu,\nu,r_P) = (3,1,1,1,-3/4,9/10,1,4)
    \]
    \item[(iv)] Let $P$ be the $\frac{1}{5} (1,2,3)$ point.
    Suppose that $P$ is nondegenerate and $X$ does not contain a quasi-line of type $(1,3,4)$.
    Let $Y \in |3A|$ be a general member and set $H \coloneq H_x \in |A|$.
    By Lemma \ref{lem:N23sgqsm3}, $P \in Z \coloneq H \cap Y \subset Y \subset X$ is a flag of type $\mathrm{I}$ with $(l_Y,l_H,e,r_P) = ( 3,1,1,5)$ and we have $\delta_P (X) \ge 16/7$ by Proposition \ref{prop:flagIdelta}.
    \item[(v)] Let $P$ be the $\frac{1}{5} (1,2,3)$ point.
    Suppose that $P$ is nondegenerate and $X$ contains a quasi-line of type $(1,3,4)$. 
    It will be proved in \S \ref{Sec:No23singviii} that $\delta_P (X) \ge 112/33$ by considering a flag of type $\mathrm{IIa}$ with
    \[
    (l_Y,l_H,m,n,\lambda,\mu,\nu,r_P) = (3,1,1,1,-9/10,3/4,1,5).
    \]
\end{itemize}

\begin{table}[h]
\caption{\textnumero \! 23: $X_{14} \subset \mbP (1,2,3,4,5)$}
\label{table:No14}
\centering
\begin{tabular}{cllll}
\toprule
Point & Case & $\delta_P$ & flag & Ref. \\
\midrule
$\frac{1}{3} (1,1,2)^{\mathrm{IEI}}$ & (i) $z^3 w, z^2 t^2 \notin \msF$ & $\delta_P \ge 140/37$ & $\mathrm{IIa}$ & \S \ref{Sec:No23singiv} \\
\cmidrule{1-5}
$\frac{1}{4} (1,1,3)^{\EI}$ & (ii) $\not\exists (1,3,4)$ & $\delta_P \ge 20/7$ & $\mathrm{I}$ &  \\
\cmidrule{2-5}
 & (iii) $\exists (1,3,4)$ & $\delta_P \ge 35/12$ & $\mathrm{IIa}$ & \S \ref{Sec:No23singvi} \\
\cmidrule{1-5}
$\frac{1}{5} (1,2,3)^{\QI}$ & (iv) ndgn, $\not\exists (1,3,4)$ & $\delta_P \ge 16/7$ & $\mathrm{I}$ & \\
\cmidrule{2-5}
 & (v) ndgn, $\exists (1,3,4)$ & $\delta_P \ge 224/93$ & $\mathrm{IIa}$ & \S \ref{Sec:No23singviii} \\
\bottomrule
\end{tabular}
\end{table}

\subsubsection{Some results on quasismoothness}

%

\begin{Lem} \label{lem:No23smiiHsm}
    One of the following holds.
    \begin{enumerate}
        \item The hypersurface $H_x$ is quasismooth along $(x = z = 0)_X$.
        \item There exists a point $P_0 \in H_x \cap (x = z = 0)_X \cap \Sm (X)$ such that $H_x$ is quasismooth along $(x = z = 0)_X \setminus \{P_0\}$.
        The surface $H_x$ has a Du Val singularity of type $A$ at $P_0$.
    \end{enumerate}
\end{Lem}

\begin{proof}
    By the quasismoothness of $X$ at $P_t$ and $P_w$, we have $t^3 y, w^2 t \in \msF$.
    Replacing $t$ suitably, we may assume that $w^2 t$ is the only monomial which is divisible by $w^2$ and $\coeff_{\msF} (w^2 t) = 1$.
    We may also assume that $\coeff_{\msF} (t^3 y) = 1$ by rescaling $y$.
    We can write
    \[
    \msF (0,y,0,t,w) = w^2 t + t^3 y + \alpha_2 t^2 y^3 + \alpha_1 t y^5 + \alpha_0 y^7,
    \]
    where $\alpha_i \in \mbC$.
    The equation $t^3 + \alpha_2 t^2 y^2 + \alpha_1 t y^5 + \alpha_0 y^6 = 0$ has $3$ distinct solutions in $\mbP (2_y,4_t)$ which correspond to the $3$ points of type $\frac{1}{2} (1,1,1)$.
    If $\alpha_0 \ne 0$, then $H_x \cap H_z$ is quasismooth, which implies that $H_x$ is quasismooth along $(x = z = 0)_X$.
    
    Suppose that $\alpha_0 = 0$.
    In this case $\alpha_1 \ne 0$ and we may assume that $\alpha_1 = -1$ by rescaling coordinates.
    Then $H_x \cap H_z$ is quasismooth outside the point $P_0 \coloneq (0\!:\!1\!:\!0\!:\!0\!:\!1)$.
    We have $(x = z = t = 0) \subset X$.
    The assertions follow from Lemma~\ref{lem:alpspH}.
\end{proof}

\begin{Lem} \label{lem:N23sgqsm0}
    Suppose that $z^3 w \notin \msF$.
    Then the surface $H_x$ has only canonical singularities.
\end{Lem}

\begin{proof}
    Let $P$ be a singular point of $H_x$.
    If $H_x$ is quasismooth at $P$, then it is one of the following types: $\frac{1}{2} (1,1)$, $\frac{1}{3} (1,2)$, $\frac{1}{4} (1,3)$ and $\frac{1}{5} (2,3)$. 
    These singularities are all canonical.
    
    By Lemma~\ref{lem:No23smiiHsm}, $H_x$ is canonical along $(x = z = 0)_X$.
    We show that $H_x$ is quasismooth or Du Val singularity of type $A$ along $(x = y = 0)_X$.
    The intersection $H_x \cap H_y$ is the hypersurface in $\mbP (3_z,4_t,5_w)$ defined by the equation
    \[
    \msF (0,0,z,t,w) = w^2 t + \beta t^2 z^2 = 0,
    \]
    where $\beta \in \mbC$.
    Suppose that $\beta \ne 0$.
    Then $H_x \cap H_y$ is quasismooth outside $\{P_z,P_t\}$.
    By the quasismoothness of $X$, we have $z^4 y \in \msF$ and $t^3 y \in \msF$, which show that $H_x$ is quasismooth at $P_z$ and $P_t$.
    Suppose that $\beta = 0$.
    Then $H_x \cap H_y$ is quasismooth along $(x = y = 0)_X \cap (w \ne 0)$.
    Let $P \in (x = y = w = 0)$ be a point.
    By Lemma~\ref{lem:alpspH}, $H_x$ is either quasismooth at $P$ or it has a Du Val singularity of type $A$ at $P$.
    Thus, $H_x$ has only canonical singularities along $(x = y = 0)_X$.
    
    In the following we assume that $P$ is contained in $(x = 0)_X \cap (yz \ne 0)$ and $H_x$ is not quasismooth at $P$.
    By replacing $t \mapsto t - \theta y^2$, $w \mapsto w - \theta' y z$ for some $\theta, \theta' \in \mathbb{C}$ and then rescaling $y$, $z$,  we may assume that $P = (0\!:\!1\!:\!1\!:\!0\!:\!0)$ and we can write
    \[
    \msF = \theta y (z^2 - y^3)^2 + \zeta t y^2 (z^2 - y^3) + \eta w z (z^2 - y^3) + g_{14} + x \msF',
    \]
    where $\theta, \zeta, \eta \in \mbC$, $g_{14} = g_{14} (y,z,t,w)$ and $\msF' = \msF' (x,y,z,t,w)$ are homogeneous polynomials of degree $14$ and $13$, respectively, such that $\msF' (P) \ne 0$ and $g_{14} \in (t,w)^2$.
    Note that we indeed have $\mathsf{F}' (P) \neq 0$ since $X$ is quasismooth at $P$.
    If $\theta = 0$, then $P \in H_x$ is Du Val of type $A$.
    Hence we assume that $\theta \ne 0$.
    By the assumption, we have $z^3 w \notin \msF$, which implies $\eta = 0$.
    By the quasismoothness of $X$, we have $t w^2, t^3 y \in \msF$.
    By rescaling coordinates, we may assume that 
    \[
    \begin{split}
        \msF &= y (z^2-y^3)^2 + 2 t y^2 (z^2-y^3) + t^2 (\lambda z^2 + \mu y^3) \\
        & \hspace{3cm} + \nu t w zy + \xi w^2 y^2 + t^3 y + t w^2 + x \msF',
    \end{split}
    \]
    where $\lambda, \mu, \nu, \xi \in \mbC$.
    We give explicit descriptions of $\mbU \coloneq (yz \ne 0) \subset \mbP (1,2,3,4,5)$ and $U \coloneq \mbU \cap X$.
    We set $M \coloneq z/y$ and $u \coloneq y^3/z^2$.
    Then $y|_U \coloneq y/M^2 = u$ and $z|_U \coloneq z/M^3 = u$.
    We also set $x|_U \coloneq x/M$, $t|_U \coloneq t/M^4$ and $w|_U \coloneq w/M^5$, and we denote them by $x, t, w$, respectively.
    We have an isomorphism $\mbU \cong \Spec \mbC [x,u,u^{-1},t,w]$ and $U$ is the hypersurface in $\mbU$ defined by the equation $\msF|_U \coloneq \msF (x,u,u,t,w) = 0$.
    The point $P$ corresponds to $(0,1,0,0) \in \mbU$.
    We set $\bar{u} \coloneq u - 1$ so that $\{x,\bar{u},t,w\}$ gives a system of local (or affine) coordinates of $\mbU$ with origin at $P$.
    We can take $\{\bar{u},t,w\}$ as a system of local coordinates of $X$ at $P$ since $\mathsf{F}' (P) \neq 0$, and $H_x$ is defined by 
    \[
    \begin{split}
        \phi &\coloneq \msF (0,\bar{u}+1,\bar{u}+1,t,w) \in \mcO_{X,P} \\
        &= \bar{u}^2 (\bar{u}+1)^5 - 2 t \bar{u} (\bar{u}+1)^4 + t^2 (\lambda (\bar{u}+1)^2 + \mu (\bar{u}+1)^3) \\
        & \hspace{1cm} + \nu t w (\bar{u}+1)^2 + \xi w^2 (\bar{u} + 1)^2 + t^3 (\bar{u}+1) + t w^2.
    \end{split}
    \]
    Let $q = q (\bar{u}, t, w)$ be the quadratic part of $\phi$.
    We have
    \[
    q = \bar{u}^2 - 2 t \bar{u} + (\lambda + \mu) t^2 + \nu t w + \xi w^2.
    \]
    If $\rank q > 1$, then $P \in H_x$ is Du Val of type $A$.
    Hence we assume that $\rank q = 1$.
    In this case, $\nu = \xi = 0$ and $\lambda + \mu = 1$.
    We make a coordinate change $\bar{u} \mapsto \tilde{u} + t$, where $\tilde{u} \coloneq \bar{u} - t$.
    Then, locally around $P$, the divisor $H_x$ is defined by
    \[
    \begin{split}
        \tilde{\phi} &\coloneq \phi (\tilde{u}+t,t,w) \\
        &= \tilde{u}^2 + (5 \tilde{u}^3 + 7 \tilde{u}^2 t + (2-\lambda) \tilde{u} t^2 + (1-\lambda)t^3 + t w^2) + \\
        & \hspace{5mm} + (10 \tilde{u}^4 + 28 \tilde{u}^3 t + (27-2\lambda) \tilde{u}^2 t^2 + (11-4\lambda) \tilde{u} t^3 + (2-2\lambda) t^4) \\
        & \hspace{5mm} + (10 \tilde{u}^5 + 42 \tilde{u}^4 t + (69-\lambda) \tilde{u}^3 t + (55-3\lambda) \tilde{u}^2 t^3 + 42 \tilde{u} t^4 + (3-\lambda) t^5) + \cdots, 
    \end{split}
    \]
    where the omitted terms consist of monomials of degree at least $6$ with respect to the usual grading.
    Suppose that $\lambda \ne 1$.
    Then the least weight term of $\tilde{\phi}$ with respect to $\wt (\tilde{u},t,w) = (3,2,2)$ is $\tilde{u}^2 + (1-\lambda) t^3 + t w^2$.
    By \cite[Corollary 4.7]{Paemurru}, $P \in H_x$ is a Du Val singularity of type $D_4$.
    Suppose that $\lambda = 1$.
    In order to remove the term $\tilde{u} t^2$, we make a coordinate change $\tilde{u} \mapsto \hat{u} - t^2/2$, where $\hat{u} = \tilde{u} + t^2/2$.
    Then 
    \[
    \hat{\phi} \coloneq \tilde{\phi} (\hat{u} - t^2/2,t,w) = \hat{u}^2 + (\hat{u}^2 h_1 + t w^2) + \left( \hat{u} h_3 - \frac{1}{4} t^4 \right) + \cdots,
    \]
    where $h_i = h_i (\hat{u}, t)$ is a homogeneous polynomial of degree $i$ with respect to the usual grading and the omitted terms consist of monomials of degree at least $5$.
    The least weight terms of $\hat{\phi}$ with respect to $\wt (\hat{u},t,w) = (4,2,3)$ is $\hat{u}^2 + t w^2 - t^4/4$.
    By \cite[Corollary 4.7]{Paemurru}, $P \in H_x$ is a Du Val singularity of type $D_5$.
    Therefore $H_x$ has only canonical singularities.
\end{proof}

\begin{Lem} \label{lem:N23sgqsm2}
    Suppose that $z^3 w, z^2 t^2 \notin \msF$.
    We set $Y \coloneq H_x \in |A|$ and let $H \in |2A|$ be a general member.
    Then the pair $(X, Y)$ is plt and we have $H|_Y = \Gamma + 2 \Delta$, where $\Gamma$ and $\Delta$ are quasismooth rational curves with the following properties.
    \begin{enumerate}
        \item $\Gamma \cap \Delta = \{P_z\}$ and $P_w \in \Gamma \setminus \Delta$.
        \item $(\Gamma^2) = -8/15$, $\Delta^2 = -1/12$ and $(\Gamma \cdot \Delta) = 1/3$.
    \end{enumerate}
\end{Lem}

\begin{proof}
    By Lemma \ref{lem:N23sgqsm0}, the surface $Y$ has only canonical singularities, and hence the pair $(X, Y)$ is plt.
    By the quasismoothness of $X$ at $P_w$, we have $w^2 t \in \msF$ and we may assume that $\coeff_{\msF} (w^2 t) = 1$ by rescaling $w$.
    Then we have
    \[
    \msF (0,0,z,t,w) = w^2 t.
    \]
    We set
    \[
    \begin{split}
        \Gamma &\coloneq (x = y = t = 0), \\
        \Delta &\coloneq (x = y = w = 0).
    \end{split}
    \]
    We have $H|_Y = \Gamma + 2 \Delta$.
    The curves $\Gamma$ and $\Delta$ are quasismooth.
    Hence $Y = H_x$ is quasismooth outside $\Delta$.
    It is easy to see that $Y$ is quasismooth at the intersection point $P_z$ of $\Gamma$ and $\Delta$ since $z^4 y \in \msF$.
    Hence $Y$ is quasismooth along $\Gamma$.

    The surface $Y$ is quasismooth along $\Gamma$, the curve $\Gamma$ is a smooth rational curve, $K_Y \sim 0$ and $\Sing_{\Gamma} (Y) = \{\frac{1}{3} (1,2), \frac{1}{5} (2,3)\}$.
    It follows that
    \[
    (\Gamma^2) = -2 + \frac{2}{3} + \frac{4}{5} = - \frac{8}{15}.
    \]
    By taking intersection numbers of $H|_Y = \Gamma + 2 \Delta$ and $\Gamma$, and then $\Delta$, we have $(\Gamma \cdot \Delta) = 1/3$ and $(\Delta^2) = -1/12$.
\end{proof}

\begin{Lem} \label{lem:N23sgqsm3}
    Suppose that $X$ does not contain a quasi-line of type $(1,3,4)$.
    Let $Y \in |3A|$ be a general member and set $H \coloneq H_x \in |A|$.
    Then $Y$ and $Z \coloneq H \cap Y$ are both quasismooth.
\end{Lem}

\begin{proof}
    By the quasismoothness of $X$ at $P_w$, we have $w^2 t \in \msF$ and we can write
    \[
    \msF = w^2 t + w f_9 + f_{14},
    \]
    where $f_i = f_i (x,y,z,t)$ is a homogeneous polynomial of degree $i$.
    By the quasismoothness of $X$ at $P_t$, we have $t^3 y \in \msF$ and we may assume that $\coeff_{\msF} (t^3 y) = 1$.
    We have
    \[
    \msF (0,y,0,t,w) = w^2 t + t^3 y + \theta_2 t^2 y^3 + \theta_1 t y^5 + \theta_0 y^7,
    \]
    where $\theta_0, \theta_1, \theta_2 \in \mbC$.
    By the assumption, $X$ does not contain the curve $(x = z = t = 0)$, which implies $\theta_0 \ne 0$.
    Then, since $X$ is quasismooth, we can write
    \[
    \msF (0,y,0,t,w) = w^2 t + y (t + \eta_1 y^2)(t + \eta_2 y^2)(t + \eta_3 y^2),
    \]
    where $\eta_1, \eta_2, \eta_3 \in \mbC \setminus \{0\}$ are mutually distinct.
    It is then easy to see that the scheme-theoretic intersection $Z = H \cap Y = H_x \cap H_z$ is quasismooth, where $Y \in |3A|$ is a general member.
    This in particular shows that $Y$ is quasismooth along $(x = z = 0)_X$, and hence $Y$ is quasismooth since the base locus of $|3A|$ is $(x = z = 0)_X$.
\end{proof}

\begin{Lem} \label{lem:N23sgqsm4}
    Suppose that $X$ contains a quasi-line of type $(1,3,4)$.
    Let $Y \in |3A|$ be a general member and set $H \coloneq H_x \in |A|$.
    Then $Y$ is quasismooth and $H|_Y = \Xi + \Theta$, where $\Xi$ and $\Theta$ are quasismooth curves with the following properties.
    \begin{itemize}
        \item $P_t \in \Theta \setminus \Xi$.
        \item $P_w \in \Xi \setminus \Theta$.
        \item $(\Xi^2) = -9/10$, $(\Theta^2) = -3/4$ and $(\Xi \cdot \Theta) = 1$.
    \end{itemize}
\end{Lem}

\begin{proof}
    By a suitable choice of homogeneous coordinates, we may assume that a quasi-line of type $(1,3,4)$ contained in $X$ is $\Xi \coloneq (x = z = t = 0)$.
    This implies that $w^2 y^2, y^7 \notin \msF$.
    By the quasismoothness of $X$ at $P_w$ and $P_t$, we have $w^2 t, t^3 y \in \msF$ and we may assume that $\coeff_{\msF} (w^2t) = \coeff_{\msF} (t^3 y) = 1$.
    We can write
    \[
    \msF (0,y,0,t,w) = w^2 t + t^3 y + \zeta t^2 y^3 + \eta t y^5
    = t (w^2 + t^2 y + \zeta t y^3 + \eta y^5),
    \]
    where $\zeta, \eta \in \mbC$.
    We set 
    \[
    \Theta \coloneq (x = z = w^2 + t^2 y + \zeta t y^3 + \eta y^5 = 0).
    \]
    By the quasismoothness of $X$ at $\frac{1}{2} (1,1,1)$ points, the equation $t^2 + \zeta t y^2 + \eta y^4 = 0$ has two distinct solutions, that is, that is, $\zeta^2 - 4 \eta \ne 0$.
    This implies that $\Theta$ is quasismooth.
    We see that $\Xi$ is quasismooth, $P_t \in \Theta \setminus \Xi$ and $P_w \in \Xi \setminus \Theta$.

    The base locus of $|3A|$ is the set $(x = z = 0)_X$.
    Hence $Y$ is quasismooth outside the unique intersection point $P_0$ of $\Xi$ and $\Theta$.
    We have $(\prt \msF/\prt y)(P_0) = (\prt \msF/\prt t)(P_0) = (\prt \msF/\prt w)(P_0) = 0$.
    By the quasismoothness of $X$ at $P_0$, either $(\prt \msF/\prt x)(P_0) \ne 0$ or $(\prt \msF/\prt z)(P_0) \ne 0$.
    Let $\msG \coloneq z + \alpha y x + \beta x^3 = 0$ be the equation which defines $Y$ in $X$, where $\alpha, \beta \in \mbC$ are general.
    Then $(\prt \msG/\prt x) (P_0) \ne 0$ and $(\prt \msG/\prt z) = 1$.
    This shows that $Y$ is quasismooth at $P_0$ for a general $\alpha, \beta$.
    Thus $Y$ is quasismooth.

    The curve $\Xi$ is a smooth rational curve, $\Sing_{\Xi} (Y) = \{\frac{1}{2} (1,1), \frac{1}{5} (1,2)\}$ and $K_Y \sim 2 A|_Y$.
    It follows that
    \[
    (\Xi^2) = -2 - \frac{1}{5} + \frac{1}{2} + \frac{4}{5} = - \frac{9}{10}.
    \]
    By taking intersection number of $H|_Y = \Xi + \Theta$ and $\Xi$, and then $\Theta$, we obtain
    \[
    (\Xi \cdot \Theta) = 1, \ (\Theta^2) = -\frac{3}{4}.
    \]
    This completes the proof.
\end{proof}

\subsubsection{Case (i): $P$ is the $\frac{1}{3} (1,1,2)$ point and $z^3 w, z^2 t^2 \notin \msF$}
\label{Sec:No23singiv}

Let $P = P_z$ be the $\frac{1}{3} (1,1,2)$ point.
Suppose that $z^3 w, z^2 t^2 \notin \msF$.
We set $Y = H_x \in |A|$ and let $H \in |2A|$ be a general member.
Let $H|_Y = \Gamma + 2 \Delta$ be as in Lemma \ref{lem:N23sgqsm2}.
Then $P \in \Gamma \subset Y \subset X$ is a flag of type $\mathrm{IIa}$ with
\[
(l_Y,l_H,m,n,\lambda,\mu,\nu,r_P) = (1,2,1,2,-8/15,1/12,1/3,3).
\]
By the formulae given in Proposition \ref{prop:flagIIadelta}, we have
\[
\begin{split}
    S (V_{\bullet,\bullet}^Y;\Gamma) &= \frac{180}{7} \int_0^1 \left(\int_0^{\frac{1}{4} (1-u)} \left(\frac{7}{60} (1-u)^2 - \frac{2}{15} (1-u)v - \frac{8}{15} v^2 \right) dv + \right. \\
    & \hspace{3cm} \left. + \int_{\frac{1}{4} (1-u)}^{\frac{1}{2} (1-u)} \frac{4}{5} \left(\frac{1}{2} (1-u)-v \right)^2 d v \right) d u \\
    &= \frac{19}{112}, \\
    S (W_{\bullet,\bullet,\bullet}^{Y,\Gamma};P) &= \frac{180}{7} \int_0^1 \left( \int_0^{\frac{1}{4} (1-u)} \left(\frac{1-u+8v}{15} \right)^2 d v + \right. \\
    & \hspace{2cm} \left. + \int_{\frac{1}{4} (1-u)}^{\frac{1}{2} (1-u)} \frac{16}{25} \left(\frac{1}{2} (1-u) -v \right)^2 d v \right) d u + F_P (W_{\bullet,\bullet,\bullet}^{Y,\Gamma}) \\
    &= \frac{11}{210} + \frac{1}{28} \\
    &= \frac{37}{420}
\end{split}
\]
where we have $F_P (W_{\bullet,\bullet,\bullet}^{Y,\Gamma}) = 1/28$ since $\ord_P (\Delta|_{\Gamma}) = (\Gamma \cdot \Delta) = 1/3$.
Thus we have
\[
\delta_P (X) \ge \min \left\{ 4, \ \frac{112}{19}, \ \frac{140}{37} \right\} = \frac{140}{37}.
\]

\subsubsection{Case (iii): $P$ is the $\frac{1}{4} (1,1,3)$ point and $X$ contain a quasi-line of type $(1,3,4)$}
\label{Sec:No23singvi}

Let $P = P_t$ be the $\frac{1}{4} (1,1,3)$ point of $X$ and suppose that $X$ contains a quasi-line of type $(1,3,4)$.
Let $Y \in |3A|, H = H_x$ and $H|_Y = \Xi + \Theta$ be as in Lemma \ref{lem:N23sgqsm3}.
We set $\Gamma \coloneq \Theta$ and $\Delta \coloneq \Xi$.
Then $P \in \Gamma \subset Y \subset X$ is a flag of type $\mathrm{IIa}$ with
\[
(l_Y,l_H,m,n,\lambda,\mu,\nu,r_P) = (3,1,1,1,-3/4,9/10,1,4).
\]

By the formulae given in Proposition \ref{prop:flagIIadelta}, we have
\[
\begin{split}
    S (V_{\bullet,\bullet}^Y;\Gamma) &= \frac{180}{7} \int_0^{\frac{1}{3}} \left( \int_0^{\frac{1-3u}{10}} \left( \frac{7}{20} (1-3u)^2 - \frac{1}{2} (1-3u)v - \frac{3}{4} v^2 \right) d v \right. \\
    & \hspace{3cm} \left. + \int_{\frac{1-3u}{10}}^{1-3u} \frac{13}{36} (1-3u-v)^2 d v \right) d u \\
    &= \frac{9}{35}, \\
    S (W_{\bullet,\bullet,\bullet}^{Y,\Gamma};P) &= \frac{180}{7} \int_0^{\frac{1}{3}} \left( \int_0^{\frac{1-3u}{10}} \left( \frac{1-3u+3v}{4} \right)^2 d v + \right. \\
    & \hspace{3cm} \left. + \int_{\frac{1-3u}{10}}^{1-3u} \frac{13^2}{36^2} (1-3u-v)^2 d v \right) d u \\
    &= \frac{3}{35},
\end{split}
\]
where we have $F_P (W_{\bullet,\bullet,\bullet}^{Y,\Gamma}) = 0$ since $\ord_P (\Delta|_{\Gamma}) = 0$.
Thus we have
\[
\delta_P (X) \ge \min \left\{ 12, \ \frac{35}{9}, \ \frac{35}{12} \right\} = \frac{35}{12}.
\]

\subsubsection{Case (v): $P$ is the $\frac{1}{5} (1,2,3)$ point which is nondegenerate and $X$ contains a quasi-line of type $(1,3,4)$}
\label{Sec:No23singviii}

Let $P = P_w$ be the $\frac{1}{5} (1,2,3)$ point.
Suppose that $P$ is nondegenerate and $X$ contains a quasi-line of type $(1,3,4)$.
Let $Y \in |3A|$, $H \coloneq H_x \in |A|$ and $H|_Y = \Xi + \Theta$ be as in Lemma \ref{lem:N23sgqsm4}.
We set $\Gamma \coloneq \Xi$ and $\Delta \coloneq \Theta$.
Then $P \in \Gamma \subset Y \subset X$ is a flag of type $\mathrm{IIa}$ with
\[
(l_Y,l_H,m,n,\lambda,\mu,\nu,r_P) = (3,1,1,1,-9/10,3/4,1,5).
\]

By the formulae given in Proposition \ref{prop:flagIIadelta}, we have
\[
\begin{split}
    S (V_{\bullet,\bullet}^Y;\Gamma) &= \frac{180}{7} \int_0^{\frac{1}{3}} \left( \int_0^{\frac{1-3u}{4}} \left(\frac{7}{20} (1-3u)^2 - \frac{1}{5} (1-3u)v - \frac{9}{10} v^2 \right) d v \right. \\
    & \hspace{3cm} \left. + \int_{\frac{1-3u}{4}}^{1-3u} \frac{13}{30} (1-3u-v)^2 d v \right) d u \\
    &= \frac{33}{112}, \\
    S (W_{\bullet,\bullet,\bullet}^{Y,\Gamma}; P) &= \frac{180}{7} \int_0^{\frac{1}{3}} \left( \int_0^{\frac{1-3u}{4}} \left(\frac{1-3u+9v}{10} \right)^2 d v  \right. \\
    & \hspace{3cm} \left. + \int_{\frac{1-3u}{4}}^{1-3u} \frac{13^2}{30^2} (1-3u-v)^2 d v \right) d u \\
    &= \frac{93}{1120},
\end{split}
\]
where we have $F_P (W_{\bullet,\bullet,\bullet}^{Y,\Gamma}) = 0$ since  $\ord_P (\Delta|_{\Gamma}) = 0$.
We have
\[
\delta_P (X) \ge \min \left\{ 12, \ \frac{112}{33}, \ \frac{224}{93} \right\} = \frac{224}{93}.
\]

Therefore Theorem \ref{thm:deltasingpt} is proved for family \textnumero 23.

\subsection{Family \textnumero 25: $X_{15} \subset \mbP (1, 1, 3, 4, 7)$}

Let $X$ be a member of family \textnumero 25.
Then $\Sing (X) = \{\frac{1}{4} (1,1,3)^{\QI}, \frac{1}{7} (1,3,4)^{\QI}\}$.
We have $\delta_P (X) > 1$ for any maximal center $P \in \Sing (X)$ by the following results.

\begin{itemize}
    \item[(i)] Let $P$ be the $\frac{1}{4} (1,1,3)$ point.
    Suppose that $P$ is nondegenerate.
    Then we have $\delta_P \ge 28/11$ by applying Proposition~\ref{prop:flagBLdelta} to the flag of type $\mathrm{wBL}_{(4,7)}$ centered at $P$.
    \item[(ii)] Let $P$ be the $\frac{1}{7} (1,3,4)$ point.
    Suppose that $t^2 w \in \msF$.
    Let $Y \in |3A|$ and $H \in |A|$ be general members.
    Then, by Lemma~\ref{lem:N25sgqsm1}, $P \in Z \coloneq H \cap Y \subset Y \subset X$ is a flag of type $\mathrm{I}$ with $(l_Y,l_H,e,r_P) = (3,1,1,7)$ and we obtain $\delta_P (X) \ge 16/15$ by Proposition~\ref{prop:flagIdelta}.
    \item[(iii)] Let $P$ be the $\frac{1}{7} (1, 3, 4)$ point.
    Suppose that $t^2 w \notin \msF$.
    It will be proved in \S \ref{Sec:N25sgiii} that $\delta_P (X) \ge 80/61$ by considering a flag of type $\mathrm{IIa}$ with
    \[
    (l_Y,l_H,m,n,\lambda,\mu,\nu,r_P) = (1,1,1,1,-11/28,2/7,3/7,7).
    \]

\end{itemize}

\begin{table}[h]
\caption{\textnumero \! 25: $X_{15} \subset \mbP (1,1,3,4,7)$}
\label{table:No25}
\centering
\begin{tabular}{cllll}
\toprule
Point & Case & $\delta_P$ & flag & Ref. \\
\midrule
$\frac{1}{4} (1,1,3)^{\QI}$ & (i) ndgn & $\delta_P \ge 28/11$ & $\mathrm{wBL}_{(4,7)}$ & \ref{prop:flagBLdelta} \\
\cmidrule{1-5}
$\frac{1}{7} (1,3,4)^{\QI}$ & (ii) $t^2 w \in \msF$ & $\delta_P \ge 16/15$ & $\mathrm{I}$ &  \\
\cmidrule{2-5}
& (iii) $t^2 w \notin \msF$ & $\delta_P \ge 80/61$ & $\mathrm{IIa}$ & \S \ref{Sec:N25sgiii} \\
\bottomrule
\end{tabular}
\end{table}

\subsubsection{Some results on quasismoothness}

\begin{Lem} \label{lem:N25sgqsm1}
    Suppose that $t^2w \in \msF$.
    Let $Y \in |3A|$ and $H \in |A|$ be general members.
    Then $Y$ and $Z \coloneq H \cap Y$ are both quasismooth. 
\end{Lem}

\begin{proof}
    The base locus of $|3A|$ is the set $\{P_t,P_w\}$.
    We have $t^2 w \in \msF$ and either $w^2 x \in \msF$ or $w^2 y \in \msF$.
    The equation defining $Y \in |3A|$ can be written as $z + c (x,y) = 0$, where $c (x,y)$ is a general cubic form in $x$ and $y$.
    It follows that $Y$ is quasismooth at $P_t$ and $P_w$, and hence it is quasismooth.

    Eliminating the variable $z$, we can identify $Y$ with a quasismooth hypersurface of degree $15$ in $\mbP (1,1,4,7)$.
    Let $\msG = \msG (x,y,t,w) = 0$ be the equation which defines $Y$ in $\mbP (1,1,4,7)$.
    The base locus of $|A|_Y$ is $\{P_t, P_w\}$.
    By the quasismoothness of $Y$ at $P_t$, we have $t^2 w \in \msF$ and either $w^2 x \in \msF$ or $w^2 y \in \msF$.
    The equation defining $Z$ in $Y$ can be written as $\alpha x + \beta y = 0$, where $\alpha, \beta \in \mbC$ are general.
    It follows that $Z$ is quasismooth at $P_t$ and $P_w$, and $Z$ is quasismooth.
\end{proof}

\begin{Lem} \label{lem:N25sgqsm2}
    Suppose that $t^2w \notin \msF$.
    Let $Y, H \in |A|$ be general members.
    Then $Y$ is quasismooth and $H|_Y = \Gamma + \Delta$, where $\Gamma$ and $\Delta$ are irreducible and reduced curves with the following properties.
    \begin{itemize}
        \item $\Gamma$ is quasismooth.
        \item $\Gamma \cap \Delta = \{P_w\}$ and $P_t \in \Gamma \setminus \Delta$.
        \item $(\Gamma^2) = -11/28$, $(\Delta^2) = -2/7$ and $(\Gamma \cdot \Delta) = 3/7$.
    \end{itemize}
\end{Lem}

\begin{proof}
    Let $Y, H \in |A|$ be general members.
    By the quasismoothness of $X$, we have $z^5, t^3 z \in \msF$.
    We may assume that $\coeff_{\msF} (z^5) = \coeff_{\msF} (t^3 z) = 1$.
    Then we have
    \[
    \msF (0,0,z,t,w) = t^3 z + z^5 = z (t^3 + z^4).
    \]
    We define
    \[
    \begin{split}
        \Gamma &\coloneq (x = y = z = 0), \\
        \Delta &\coloneq (x =  y= t^3 + z^4 = 0).
    \end{split}
    \]
    The curve $\Gamma$ is quasismooth, $\Delta$ is quasismooth outside $P_w$ and $\Gamma \cap \Delta = \{P_w\}$.
    It follows that $Y \in |A|$ is quasismooth outside $P_w$.
    By the quasismoothness of $X$ at $P_w$, we have either $w^2 x \in \msF$ or $w^2 y \in \msF$.
    The equation defining $Y$ in $X$ is $\alpha x + \beta y = 0$, where $\alpha, \beta \in \mbC$ are general.
    It follows that $Y$ is quasismooth at $P_w$, and hence it is quasismooth.

    We have $H|_Y = \Gamma + \Delta$.
    The curve $\Gamma$ is a smooth rational curve, $K_Y \sim 0$ and $\Sing_{\Gamma} (Y) = \{\frac{1}{4} (1,3),\frac{1}{7} (3,4)\}$.
    It follows that
    \[
    (\Gamma^2) = -2 + \frac{3}{4} + \frac{6}{7} = - \frac{11}{28}.
    \]
    By taking intersection numbers of $H|_Y = \Gamma + \Delta$ and $\Gamma$, and then $\Delta$, we obtain
    \[
    (\Gamma \cdot \Delta) = \frac{3}{7}, \ 
    (\Delta^2) = - \frac{2}{7},
    \]
    and the proof is completed.
\end{proof}
\color{black}

%
%

\subsubsection{Case (iii): $P$ is the $\frac{1}{7} (1, 3, 4)$ point and $t^2 w \notin \msF$}
\label{Sec:N25sgiii}

Let $P = P_w$ be the $\frac{1}{7} (1,3,4)$ point.
Suppose that $P$ is exceptional, that is, $t^2 w \notin \msF$.
Let $Y, H \in |A|$ be general members and let $H|_Y = \Gamma + \Delta$ be as in Lemma \ref{lem:N25sgqsm2}.
Then, $P \in \Gamma \subset Y \subset X$ is a flag of type $\mathrm{IIa}$ with
\[
(l_Y,l_H,m,n,\lambda,\mu,\nu,r_P) = (1,1,1,1,-11/28,2/7,3/7,7).
\]

By the formulae given in Proposition \ref{prop:flagIIadelta}, we have
\[
\begin{split}
    S (V_{\bullet,\bullet}^Y;\Gamma) &= \frac{84}{5} \int_0^1 \left( \int_0^{\frac{1-u}{3}} \left( \frac{5}{28} (1-u)^2 - \frac{1}{14} (1-u)v - \frac{11}{28} v^2 \right) d v + \right. \\
    & \hspace{3cm} \left. + \int_{\frac{1-u}{3}}^{1-u} \frac{1}{4} (1-u-v)^2 d v \right) d u \\
    &= \frac{19}{60}, \\
    S (W_{\bullet,\bullet,\bullet}^{Y,\Gamma}; P) &= \frac{84}{5} \int_0^1 \left( \int_0^{\frac{1-u}{3}} \left( \frac{1-u+11v}{28} \right)^2 d v \right. \\
    & \hspace{2cm} \left. + \int_{\frac{1-u}{3}}^{1-u} \frac{1}{16} (1-u-v)^2 d v \right) d u + F_P (W_{\bullet,\bullet,\bullet}^{Y,\Gamma}) \\
    &= \frac{71}{1680} + \frac{1}{15} \\
    &= \frac{61}{560}
\end{split}
\]
where we have $F_P (W_{\bullet,\bullet,\bullet}^{Y,\Gamma}) = 1/15$ since $\ord_P (\Delta|_{\Gamma}) = 3/7$.
Thus, we have
\[
\delta_P (X) \ge \min \left\{ 4, \ \frac{60}{19}, \ \frac{80}{61} \right\} = \frac{80}{61}.
\]

Therefore Theorem \ref{thm:deltasingpt} is proved for family \textnumero 25.

\subsection{Family \textnumero 33: $X_{17} \subset \mbP (1, 2, 3, 5, 7)$}

Let $X$ be a member of family \textnumero 33.
Then $\Sing (X) = \{\frac{1}{2} (1,1,1), \frac{1}{3} (1,1,2), \frac{1}{5} (1,2,3)^{\QI}, \frac{1}{7} (1,2,5)^{\QI}\}$.
We have $\delta_P (X) > 1$ for any maximal center $P \in \Sing (X)$ by the following results.

\begin{itemize}
    \item[(i)] Let $P$ be the $\frac{1}{5} (1,2,3)$ point.
    Suppose that $P$ is nondegenerate.
    Then $\alpha_P (X) \ge 14/17$ by \cite[\S 5.6.d]{KOW}, and hence $\delta_P (X) \ge 56/51$.
    \item[(ii)] Let $P$ be the $\frac{1}{7} (1,2,5)$ point.
    Suppose that $y^5 w \in \msF$.
    Let $Y \in |5A|$ be a general member and set $H \coloneq H_x$.
    Then, by Lemma~\ref{Lem:N33sgqsm2}, $P \in Z \coloneq H \cap Y \subset Y \subset X$ is a flag of type $\mathrm{I}$ with $(l_Y, l_H, e, r_P) = (5,1,1,7)$ and we have $\delta_P (X) \ge 24/17$.
    \item[(iii)]
    Let $P$ be the $\frac{1}{7} (1, 2, 5)$ point.
    Suppose that $y^5 w \notin \msF$.
    It will be proved in \S \ref{Sec:N33sgvii} that $\delta_P (X) \ge 408/275$ by considering a flag of type $\mathrm{IIa}$ with
    \[
    (l_Y,l_H,m,n,\lambda,\mu,\nu,r_P) = (5,1,1,1,-13/14,2/3,1,7).
    \]
\end{itemize}

\begin{table}[h]
\caption{\textnumero \! 33: $X_{17} \subset \mbP (1, 2, 3, 5, 7)$}
\label{table:No33}
\centering
\begin{tabular}{cllll}
\toprule
Point & Case & $\delta_P$ & flag & Ref. \\
\midrule
$\frac{1}{5} (1,2,3)^{\QI}$ & (i) ndgn & $\delta_P \ge 56/51$ & & \cite[\S 5.6.d]{KOW} \\
\cmidrule{1-5}
$\frac{1}{7} (1, 2, 5)^{\QI}$ & (ii) $y^5 w \in \msF$ & $\delta_P \ge 24/17$ & $\mathrm{I}$ & \\
\cmidrule{2-5}
 & (iii) $y^5 w \notin \msF$ & $\delta_P \ge 408/275$ & $\mathrm{IIa}$ & \S \ref{Sec:N33sgvii} \\
\bottomrule
\end{tabular}
\end{table}

\subsubsection{Some results on quasismoothness}

\begin{Lem} \label{Lem:N33sgqsm2}
    Let $Y \in |5 A|$ be a general member and set $H \coloneq H_x \in |A|$.
    Then $Y$ is quasismooth and the following hold.
    \begin{enumerate}
        \item If $y^5 w \in \msF$, then $Z \coloneq H \cap Y$ is quasismooth.
        \item If $y^5 w \notin \msF$, then $H|_Y = \Gamma + \Delta$, where $\Gamma$ is a quasismooth curve and $\Delta$ is an irreducible and reduced curve such that $P_w \in \Gamma \setminus \Delta$ and
        \[
        (\Gamma^2) = - \frac{13}{14}, \ 
        (\Delta^2) = - \frac{2}{3}, \ 
        (\Gamma \cdot \Delta) = 1.
        \]
    \end{enumerate}
\end{Lem}

\begin{proof}
    We first prove the quasismoothness of $Y$.
    The base locus of $|5A|$ is the set $(x = yz = t = 0)_X = (x = y = t = 0)_X \cup (x = z = t = 0)_X$, where
    \[
    (x = y = t = 0)_X = \{P_z, P_w\}
    \]
    and
    \[
    (x = z = t = 0)_X = 
    \begin{cases}
        \{P_y, P_w\}, & \text{if $y^5 w \in \msF$}, \\
        (x = z = t = 0), & \text{if $y^5 w \notin \msF$}.
    \end{cases}
    \]
    Let 
    \[
    t - (\lambda z y + \mu y^2 x + x^2 h_3 (x, y, z)) = 0 
    \]
    be the equation of $Y$ in $X$, where $\lambda, \mu \in \mbC$ and $h_3 \in \mbC [x, y, z]$ are general.
    It is straightforward to see that $Y$ is quasismooth at $P_y, P_z$ and $P_w$.
    If $y^5 w \in \msF$, then $Y$ is quasismooth.

    We show that $Y$ is quasismooth assuming that $y^5 w \notin \msF$.
    By the quasismoothness of $X$ at $P_w$, we have $w^2 z \in \msF$ and we may assume that $\coeff_{\msF} (w^2 z) = 1$.
    We can write
    \[
    \msF = x (\alpha w^2 y + \beta y^8) + z (w^2 + \gamma y^7) + \delta t y^6 + \msF',
    \]
    where $\alpha, \beta, \gamma, \delta \in \mbC$ and $\msF' = \msF' (x, y, z, t, w)$ is a homogeneous polynomial of degree $17$ contained in the ideal $(x, z, t)^2$.
    If $\delta \ne 0$, then we may assume that $\delta = 1$, $\gamma = \beta = \alpha = 0$ by replacing $t$ and $z$ suitably.
    In this case, we have
    \[
    J_Y|_{(x = z = t = 0)} =
    \begin{pmatrix}
        0 & 0 & w^2 & y^6 & 0 \\
        - \mu y^2 & 0 & -\lambda y & 1 & 0
    \end{pmatrix}
    \]
    and it is easy to see that $\rank J_Y (Q) = 2$ for any point $Q \in (x = z = t = 0)$.
    If $\delta = 0$ and $\gamma \ne 0$, then we may assume that $\gamma = -1$ and $\beta = 0$ by replacing $z$.
    In this case, we have 
    \[
    J_Y|_{(x = z = t = 0)} =
    \begin{pmatrix}
        \alpha w^2 y & 0 & w^2 - y^7 & 0 & 0 \\
        - \mu y^2 & 0 & - \lambda y & 1 & 0
    \end{pmatrix}.
    \]
    We have $\alpha \ne 0$ because otherwise $X$ is not quasismooth at $(0\!:\!1\!:\!0\!:\!0\!:\!1)$.
    It is then easy to see that $\rank J_Y (Q) = 2$ for any point $Q \in (x = z = t = 0)$.
    Suppose that $\delta = \gamma = 0$.
    Then we have $\beta \ne 0$ by the quasismoothness of $X$ at $P_y$.
    We may assume that $\beta = 1$ and $\alpha = 0$ by replacing $z$.
    We have
    \[
    J_Z|_{(x = z = t = 0)} =
    \begin{pmatrix}
        y^8 & 0 & w^2 & 0 & 0 \\
        - \mu y^2 & 0 & - \lambda y & 1 & 0
    \end{pmatrix},
    \]
    and we see that $\rank J_Y (Q) = 2$ for any point $Q \in (x = z = t = 0)$.
    Thus $Y$ is quasismooth.

    We next consider the intersection $Z = H \cap Y$.
    By eliminating the variable $t$, we identify $Y$ with a quasismooth hypersurface of degree $17$ in $\mbP (1,2,3,7)$.
    Let $\msG = \msG (x,y,z,w)$ be the equation which defines $Y$ in $\mbP (1,2,3,7)$.
    Suppose that $y^5 w \in \msF$.
    We may assume that $\coeff_{\msF} (y^5 w) = 1$ by replacing $y$.
    By the quasismoothness of $Y$ at $P_z$ and $P_w$, we have $z^5 y, w^2 z \in \msF$.
    We may assume that $\coeff_{\msF} (z^5 y) = \coeff_{\msF} (w^2 z) = 1$.
    Then we can write
    \[
    \msH \coloneq \msG (0,y,z,w) = w^2 z + w (\theta z^2 y^2 + y^5) + z^5 y + \zeta z^3 y^4 + \eta z y^7, 
    \]
    where $\theta, \zeta, \eta \in \mbC$.
    The intersection $Z = H \cap Y$ is isomorphic to the hypersurface in $\mbP (2_y,3_z,7_w)$ defined by $\msH = 0$.
    Note that $\msG (0,y,z,w) = \msF (0,y,z,\lambda zy,w)$ for a general $\lambda \in \mbC$.
    It follows that at least one of $\theta \in \mbC$ and $\zeta \in \mbC$ is general since either $t^2 w \in \msF$ or $t^3 y \in \msF$. 
    Replacing $w \mapsto w - \theta z y^2/2$, we may assume that $\theta = 0$.
    Then $\zeta \in \mbC$ is general.
    It can be checked by explicit computations that $\prt \msH/\prt y = \prt \msH/\prt z = \prt \msH/\prt w = 0$ does not have a non-trivial solution for a general $\zeta \in \mbC$.
    This shows that $Z$ is quasismooth.

    Suppose that $y^5 w \notin \msF$.
    We may assume that $\coeff_{\msF} (w^2 z) = \coeff_{\msF} (z^5 y) = 1$.
    We can write
    \[
    \msH \coloneq \msG (0,y,z,w) = w^2 z + \theta w z^2 y^2 + z^5 y + \zeta z^3 y^4 + \eta z y^7,
    \]
    where $\theta, \zeta, \eta \in \mbC$.
    By replacing $w \mapsto w - \theta z y^2/2$, we may assume that $\theta = 0$.
    By the same argument as above, $\zeta \in \mbC$ is general.
    The polynomial $w^2 + z^4 y + \zeta z^2 y^4 + \eta y^7$ is irreducible since $\zeta$ is general.
    We set 
    \[
    \begin{split}
        \Gamma &\coloneq (x = z = t = 0), \\
        \Delta &\coloneq (x = t - \lambda z y = w^2 + z^4 y + \zeta z^2 y^4 + \eta y^7 = 0).
    \end{split}
    \]
    We see that $\Gamma$ is quasismooth, $\Delta$ is irreducible and reduced, and $P_w \in \Gamma \setminus \Delta$.
    The curve $\Gamma$ is a smooth rational curve, $K_Y \sim 4 A|_Y$ and $\Sing_{\Gamma} (Y) = \{\frac{1}{2} (1,1),\frac{1}{7} (1,2)\}$.
    It follows that
    \[
    (\Gamma^2) = -2 - \frac{2}{7} + \frac{1}{2} + \frac{6}{7} = - \frac{13}{14}. 
    \]
    By taking intersection numbers of $H|_Y = \Gamma + \Delta$ and $\Gamma$, and then $\Delta$, we obtain
    \[
    (\Gamma \cdot \Delta) = 1, \ (\Delta^2) = - \frac{2}{3}.
    \]
    This completes the proof.
\end{proof}

\subsubsection{Case (iii): $P$ is the point of type $\frac{1}{7} (1, 2, 5)$ and $y^5 w \notin \msF$} \label{Sec:N33sgvii}

Let $P = P_w$ be the $\frac{1}{7} (1,2,5)$ point.
Suppose that $y^5 w \notin \msF$.
Let $Y \in |5A|$ be a general member, $H \coloneq H_x \in |A|$and let $H|_Y = \Gamma + \Delta$ be as in (2) of Lemma~\ref{Lem:N33sgqsm2}.
Then, $P \in \Gamma \subset Y \subset X$ is a flag of type $\mathrm{IIa}$ with
\[
(l_Y,l_H,m,n,\lambda,\mu,\nu,r_P) = (5,1,1,1,-13/14,2/3,1,7).
\]

By the formulae given in Proposition \ref{prop:flagIIadelta}, we have
\[
\begin{split}
    S (V_{\bullet,\bullet}^Y;\Gamma) &= \frac{630}{17} \int_0^{\frac{1}{5}} \left( \int_0^{\frac{1-5u}{3}} \left( \frac{17}{42} (1-5u)^2 - \frac{1}{7} (1-5u)v - \frac{13}{14} v^2 \right) dv \right. \\
    & \hspace{3cm} \left. + \int_{\frac{1-5u}{3}}^{1-5u} \frac{4}{7} (1-5u-v)^2 d v \right) d u \\
    &= \frac{65}{204}, \\
    S (W_{\bullet,\bullet,\bullet}^{Y,\Gamma};P) &= \frac{630}{17} \int_0^{\frac{1}{5}} \left( \int_0^{\frac{1-5u}{3}} \left( \frac{1-5u+13v}{14} \right)^2 d v \right. \\
    & \hspace{3cm} \left. + \int_{\frac{1-5u}{3}}^{1-5u} \frac{16}{49} (1-5u - v)^2 d v \right) d u \\
    &= \frac{275}{2856},
\end{split}
\]
where we have $F_P (W_{\bullet,\bullet,\bullet}^{Y,\Gamma}) = 0$ since $\ord_P (\Delta|_{\Gamma}) = 0$.
Thus, we have
\[
\delta_P (X) \ge \min \left\{ 20, \ \frac{204}{65}, \ \frac{408}{275} \right\} = \frac{408}{275}.
\]

Therefore Theorem \ref{thm:deltasingpt} is proved for family \textnumero 33.

\subsection{Family \textnumero 38: $X_{18} \subset \mbP (1, 2, 3, 5, 8)$}

Let $X$ be a member of family \textnumero 38.
Then $\Sing (X) = \{2 \times \frac{1}{2} (1,1,1), \frac{1}{5} (1,2,3)^{\QI}, \frac{1}{8} (1,3,5)^{\QI}\}$.
We have $\delta_P (X) > 1$ for any maximal center $P \in \Sing (X)$ by the following results.

\begin{itemize}
     \item[(i)] Let $P$ be the $\frac{1}{5} (1,2,3)$ point.
    Suppose that $P$ is nondegenerate.
    Then $\alpha_P (X) \ge 8/9$ by \cite[\S 5.6.d]{KOW}, and hence $\delta_P (X) \ge 32/27$.
    \item[(ii)] Let $P$ be the $\frac{1}{8} (1,3,5)$ point.
    Let $Y \in |5A|$ be general members and set $H \coloneq H_x \in |A|$.
    Then, by Lemma~\ref{Lem:N38sgqsm1}, $P \in Z \coloneq H \cap Y \subset Y \subset X$ is a flag of type $\mathrm{I}$ with $(l_Y, l_H, e, r_P) = (5,1,1,8)$ and we have $\delta_P (X) \ge 4/3$ by Proposition~\ref{prop:flagIdelta}.
\end{itemize}

\begin{table}[h]
\caption{\textnumero \! 38: $X_{18} \subset \mbP (1, 2, 3, 5, 8)$}
\label{table:No38}
\centering
\begin{tabular}{cllll}
\toprule
Point & Case & $\delta_P$ & flag & Ref. \\
\midrule
$\frac{1}{5} (1,2,3)^{\QI}$ & (i) ndgn & $\delta_P \ge 32/27$ & & \cite[\S 5.6.d]{KOW} \\
\cmidrule{1-5}
$\frac{1}{8} (1, 3, 5)^{\QI}$ & (ii) & $\delta_P \ge 4/3$ & $\mathrm{I}$ & \\
\bottomrule
\end{tabular}
\end{table}

\subsubsection{Some results on quasismoothness}

\begin{Lem} \label{Lem:N38sgqsm1}
    Let $Y \in |5 A|$ be a general member and let $H \coloneq H_x \in |A|$.
    Then $Y$ and $Z \coloneq H \cap Y$ are both quasismooth.
\end{Lem}

\begin{proof}
    The base locus of $|5A|$ is $(x = yz = t = 0)_X = (x = y = t = 0)_X \cup (x = z = t = 0)_X$, which consists of $P_w$ and $2$ points of type $\frac{1}{2} (1, 1, 1)$.
    Let $t - \lambda yz - x h_4 = 0$ be the equation which defines $Y$ in $X$, where $\lambda \in \mbC$ and $h_4 \in \mbC [x, y, z]$ are general.
    It is easy to see that $Y$ is quasismooth at $P_w$ and at $2$ points of type $\frac{1}{2} (1,1,1)$.
    Hence $Y$ is quasismooth.
    
    By the quasismoothness of $X$, we have $w^2 y, z^6 \in \msF$ and we may assume that their coefficients in $\msF$ are $1$.
    Moreover, again by the quasismoothness of $X$, the equation $\msF (0, y, z, 0, 0) = z^6 + \alpha_4 z^4 y^3 + \alpha_2 z^2y^6 + \alpha_0 y^9 = 0$ has $3$ distinct solutions (in $\mbP (2, 3)$) and either $t^2 w \in \msF$ or $t^3 z \in \msF$.
    We write
    \[
    \begin{split}
        \overline{\msF} &\coloneq \msF (0,y,z,\lambda yz,w) \\
        &= w^2 y + w(\alpha (\lambda) z^2y^2 + \beta y^5)+ z^6 + \gamma (\lambda) z^4y^3 + \delta (\lambda) z^2y^6 + \varepsilon y^9,
    \end{split}
    \]
    where $\alpha (\lambda), \beta, \gamma (\lambda), \delta (\lambda), \varepsilon \in \mbC$.
    Then $\beta, \varepsilon \in \mbC$ does not depend on $\lambda$ while we can view $\alpha (\lambda), \gamma (\lambda), \delta (\lambda)$ as polynomials in variable $\lambda$.
    We have $\deg_{\lambda} \alpha(\lambda) \le 2$, $\deg_{\lambda} \gamma(\lambda) \le 3$ and $\deg_{\lambda} \delta(\lambda) \le 2$.
    If $t^2 w \in \msF$, then $\deg_{\lambda} a(\lambda)=2$.
    If $t^2 w \notin \msF$, then $t^3 z \in \msF$ and we have $\deg_{\lambda} \gamma (\lambda) = 3$.
    Replacing $w \mapsto w - (\alpha (\lambda) z^2 y + \beta y^4)/2$, we may assume that
    \[
    \overline{\msF} = w^2 y + z^6 + \gamma' (\lambda) z^4y^3 + \delta' (\lambda) z^2y^6 + \varepsilon' y^9,
    \]
    where $\gamma'(\lambda) = \gamma (\lambda) - \alpha (\lambda)^2/4$, $\delta' (\lambda) = \delta (\lambda) - \alpha (\lambda) \beta/2$ and $\varepsilon' = \varepsilon - \beta^2/4$.
    If $\varepsilon' = 0$, then $P_y$ is contained in $X$ and $X$ is not quasismooth at $P_y$ since $w y^5 \notin \msF$.
    Hence $\varepsilon' \neq 0$.
    Eliminating $x$ and $t = \lambda y z + x h_4$, we have an isomorphism
    \[
    Z = H \cap Y \cong (\overline{\msF} = 0) \subset \mbP (2_y, 3_z, 8_w),
    \]
    which is quasismooth if and only if 
    \[
    \Delta (\lambda) \coloneq -4 \delta'(\lambda)^3 -27 {\varepsilon'}^2 + \gamma'(\lambda)^2 \delta'(\lambda)^2+18\gamma'(\lambda) \delta' (\lambda) \varepsilon' - 4 \gamma'(\lambda)^3 \varepsilon' \ne 0.
    \]
    We see that $\Delta (\lambda)$ is a nonzero polynomial in $\lambda$ of positive degree since $\deg_{\lambda} \gamma'(\lambda) \ge 3$ and $\deg_{\lambda} \delta'(\lambda) \le 2$.
    Hence $\Delta (\lambda) \ne 0$ since $\lambda$ is general.
    Therefore $Z$ is quasismooth.
\end{proof}

Therefore Theorem \ref{thm:deltasingpt} is proved for family \textnumero 38.

\subsection{Family \textnumero 40: $X_{19} \subset \mbP (1, 3, 4, 5, 7)$}

Let $X$ be a member of family \textnumero 40.
Then $\Sing (X) = \{\frac{1}{3} (1, 1, 1), \frac{1}{4} (1, 1, 3), \frac{1}{5} (1, 2, 3)^{\EI}, \frac{1}{7} (1, 3, 4)^{\QI}\}$.
We have $\delta_P (X) > 1$ for any maximal center $P \in \Sing (X)$ by the following results.

\begin{itemize}
    \item[(i)] Let $P$ be the $\frac{1}{5} (1,2,3)$ point.
    Then $\alpha_P (X) \ge 1$ by \cite[Proposition 5.4]{KOW}, and hence $\delta_P (X) \ge 4/3$.
    \item[(ii)] Let $P$ be the $\frac{1}{7} (1,3,4)$ point.
    Then it will be proved in \S \ref{Sec:N43sgv} that $\alpha_P (X) \ge 15/19$, and hence $\delta_P (X) \ge 20/19$.
\end{itemize}

\begin{table}[h]
\caption{\textnumero \! 40: $X_{19} \subset \mbP (1, 3, 4, 5, 7)$}
\label{table:No40}
\centering
\begin{tabular}{cllll}
\toprule
Point & Case & $\delta_P$ & flag & Ref. \\
\midrule
$\frac{1}{5} (1,2,3)^{\EI}$ & (i) & $\delta_P \ge 4/3$ & & \cite[5.4]{KOW} \\
\cmidrule{1-5}
$\frac{1}{7} (1, 3, 4)^{\QI}$ & (ii) & $\delta_P \ge 20/19$ & & \S \ref{Sec:N43sgv} \\
\bottomrule
\end{tabular}
\end{table}

\subsubsection{Case (ii): The $\frac{1}{7} (1,3,4)$ point} \label{Sec:N43sgv}

Let $P = P_w$ be the $\frac{1}{7} (1,3,4)$ point.
By the quasismoothness of $X$ at $P$, we have $w^2 t \in \msF$.
It follows that the set $\mcC \coloneq \{x, y, z\}$ isolates $P$, that is, the point $P$ is an isolated component of the set $(x = y = z = 0)_X$.
By the same argument as in the proof of \cite[Proposition 5.4]{KOW} with the choice of the divisor $S = H_x$, we conclude that $\alpha_P (X) \ge \min \{1, c\}$, where 
\[
c \coloneq \frac{1}{rne_{\max} (A^3)}.
\]
In the above definition of $c$, $r$ is the index of the singularity $P \in X$, $n$ is the positive integer such that $S \sim n A$ and $e_{\max} \coloneq \max \Set{ \deg v | v \in \mcC}$.
In our case $r = 7$, $n = 1$ and $e_{\max} = 4$. 
Thus we have $c = 15/19$ and $\alpha_P (X) \ge 15/19$.

Therefore Theorem \ref{thm:deltasingpt} is proved for family \textnumero 40.

\subsection{Family \textnumero 58: $X_{24} \subset \mbP (1, 3, 4, 7, 10)$}

Let $X$ be a member of family \textnumero 58.
Then $\Sing (X) = \{2 \times \frac{1}{2} (1, 1, 1), \frac{1}{7} (1, 3, 4)^{\QI}, \frac{1}{10} (1, 3, 7)^{\QI}\}$.
We have $\delta_P (X) > 1$ by the following results.
\begin{itemize}
    \item[(i)] Let $P$ be the $\frac{1}{7} (1, 3, 4)$ point.
    Suppose that $t^2 w \in \msF$.
    Let $Y \in |3A|$ and $H \in |A|$ be general members.
    Then, by Lemma~\ref{Lem:N58sgqsm1}, $P \in Z \coloneq H \cap Y \subset Y \subset X$ is a flag of type $\mathrm{I}$ with $(l_Y, l_H, e, r_P) = (3, 1, 1, 7)$ and we have $\delta_P (X) \ge 4$ by Proposition~\ref{prop:flagIdelta}.
    \item[(ii)] Let $P$ be the $\frac{1}{10} (1, 3, 7)$ point.
    Let $Y \in |7A|$ and be a general member and set $H \coloneq H_x \in |A|$.
    Then, by Lemma~\ref{Lem:N58sgqsm2}, $P \in Z \coloneq Y \cap H \subset Y \subset X$ is a flag of type $\mathrm{I}$ with $(l_Y, l_H, e, r_P) = (7, 1, 1, 10)$ and we have $\delta_P (X) \ge 2$ by Proposition~\ref{prop:flagIdelta}.
\end{itemize}

\begin{table}[h]
\caption{\textnumero \! 58: $X_{24} \subset \mbP (1, 3, 4, 7, 10)$}
\label{table:No58}
\centering
\begin{tabular}{cllll}
\toprule
Point & Case & $\delta_P$ & flag & Ref. \\
\midrule
$\frac{1}{7} (1, 3, 4)^{\QI}$ & (i) $t^2 w \in \msF$ & $\delta_P \ge 4$ & $\mathrm{I}$ & \\
\cmidrule{1-5}
$\frac{1}{10} (1, 3, 7)^{\QI}$ & (ii) & $\delta_P \geq 2$ & $\mathrm{I}$ & \\
\bottomrule
\end{tabular}
\end{table}

\subsubsection{Some results on quasismoothness}

\begin{Lem} \label{Lem:N58sgqsm1}
    Suppose that $t^2 w \in \msF$.
    Let $Y \in |3 A|$ be a general member and let $H \coloneq H_x \in |A|$.
    Then, $Y$ and $Z \coloneq H \cap Y$ are both quasismooth.
\end{Lem}

\begin{proof}
    We have $w^2 z, t^2 w, z^6 \in \msF$ by the quasismoothness of $X$ and by the assumption.
    We may assume that their coefficients in $\msF$ are $1$ by rescaling coordinates.
    We can write
    \[
    \msF (0, 0, z, t, w) = w^2 z + w t^2 + z^6.
    \]
    The base locus of $|3A|$ is $Z = H \cap Y$ set-theoretically, and hence $Y$ is quasismooth outside $Z$.
    Moreover, the scheme-theoretic intersection $Z \coloneq H \cap Y$ is the curve
    \[
    (x = y = w^2 z + w t^2 + z^6 = 0),
    \]
    which is quasismooth.
    Hence $Y$ and $Z$ are quasismooth.
\end{proof}

\begin{Lem} \label{Lem:N58sgqsm2}
    Let $Y \in |7A|$ be a general member and let $H \coloneq H_x \in |A|$.
    Then $Y$ and $Z \coloneq H \cap Y$ are both quasismooth.
\end{Lem}

\begin{proof}
    The linear system $|7A|$ is generated by $t, y z, x^7$ and other monomials of degree $7$ divisible by $x$.
    Hence $\Bs |7A| = (x = y = t = 0)_X \cup (x = z = t = 0)_X$.
    We have $w^2 z, z^6, y^8 \in \msF$ by the quasismoothness of $X$ and hence we may assume 
    \[
    \msF (0, 0, z, 0, w) = w^2 z - z^6, \quad 
    \msF (0, y, 0, 0, w) = y^8.
    \]
    It follows that $(x = y = t = 0)_X = \{P_w, Q_+, Q_-\}$, where $Q_{\pm} = (0\!:\!0\!:\!\pm 1\!:\!0\!:\!1) \in X$, and $(x = z = t = 0)_X = \{P_w\}$.
    Hence $\Bs |7A| = \{P_w, Q_+, Q_-\}$.

    Let $t - \lambda y z - x h_6 = 0$ be the equation which defines $Y$ in $X$, where $\lambda \in \mbC$ is general and $h_6 = h_6 (x, y, z)$ is a general homogeneous polynomial of degree $6$.
    It is then straightforward to check that $Y$ is quasismooth at $P_w, Q_+$ and $Q_-$, and hence it is quasismooth.
    We can write
    \[
    \overline{\msF} \coloneq \msF (0, y, z, \lambda yz, w) = w^2 z + \alpha (\lambda) w z^2 y^2 + z^6 + \beta (\lambda) z^3 y^4 + y^8,
    \]
    where $\alpha (\lambda), \beta (\lambda) \in \mbC$.
    We can view $\alpha (\lambda)$ and $\beta (\lambda)$ as polynomials in $\lambda$.
    We have $\deg_{\lambda} \alpha (\lambda) \le 2$ and $\deg_{\lambda} \beta (\lambda) \le 3$ and moreover either $\deg_{\lambda} \alpha (\lambda) = 2$ or $\deg_{\lambda} \beta (\lambda) = 3$ since either $t^2 w \in \msF$ or $t^3 y \in \msF$ by the quasismoothness of $X$.
    Replacing $w \mapsto w - \alpha (\lambda) z y^2/2$, we may assume that 
    \[
    \overline{\msF} = w^2 z + z^6 + \beta' (\lambda) z^3 y^4 + y^8,
    \]
    where $\beta' (\lambda) = \beta (\lambda) - \alpha (\lambda)^2/4$.
    Then, eliminating $x$ and $t$, we have an isomorphism 
    \[
    Z \coloneq H \cap Y \cong (\overline{\msF} = 0) \subset \mbP (3_y, 4_z, 10_w),
    \]
    which is quasismooth if and only if $\beta' (\lambda) \ne \pm 2$.
    We see that $\beta' (\lambda)$ is a nonzero polynomial of degree at least $3$.
    Hence $\beta' (\lambda) \ne \pm 2$ since $\lambda$ is general.
    Thus $Z$ is quasismooth.
\end{proof}

Therefore Theorem \ref{thm:deltasingpt} is proved for family \textnumero 58, and hence for all the families.

\bibliography{bibliography}
\bibliographystyle{alpha}

\end{document}